\documentclass[12pt]{amsart}
\title{Linear forms and quadratic uniformity for functions on $\Z_N$}
\author{W.T. Gowers}
\address{Department of Pure Mathematics and Mathematical Statistics, 
Wilberforce Road, Cambridge CB3 0WB, UK.}
\email{W.T.Gowers@dpmms.cam.ac.uk}
\author{J. Wolf}
\address{Rutgers, The State University of New Jersey, Department of Mathematics, 110 Frelinghuysen Rd., Piscataway, NJ 08854, U.S.A.}
\email{julia.wolf@cantab.net}

\usepackage{amsmath, wrapfig}
\usepackage{dsfont, a4wide, amsthm, amssymb, amsfonts, graphicx}
\usepackage{fancyhdr, xspace, psfrag, setspace, supertabular, color}

\newtheorem{theorem}{Theorem}[section]
\newtheorem{proposition}[theorem]{Proposition}
\newtheorem{lemma}[theorem]{Lemma}

\newtheorem{corollary}[theorem]{Corollary}
\newtheorem{conjecture}[theorem]{Conjecture}
\newtheorem*{definition}{Definition}

\def\eps{\epsilon}
\def\E{\mathbb{E}}
\def\Z{\mathbb{Z}}

\def\T{\mathbb{T}}
\def\C{\mathbb{C}}

\def\P{\mathbb{P}}
\def\F{\mathbb{F}}

\def\x{x}

\def\a{\alpha}
\def\b{\beta}
\def\g{\gamma}
\def\d{\delta}
\def\e{\epsilon}

\def\ni{\noindent}

\def\maps{\rightarrow}
\def\ra{\rightarrow}
\def\seq#1#2{#1_1,\dots,#1_#2}
\def\sm#1#2{\sum_{#1=1}^#2}
\def\sp#1{\langle #1\rangle}
\def\ol{\overline}
\def\hf{\hat{f}}
\def\hQ{\hat{Q}}

\def\ra{\rightarrow}

\begin{document}

\begin{abstract}
A very useful fact in additive combinatorics is that analytic expressions that
can be used to count the number of structures of various kinds in subsets of
Abelian groups are robust under quasirandom perturbations, and moreover
that quasirandomness can often be measured by means of certain easily
described norms, known as uniformity norms. However, determining which 
uniformity norms work for which structures turns out to be a surprisingly
hard question. In \cite{Gowers:2007tcs} and \cite{Gowers:2009lfuI,Gowers:2009lfuII}  we gave 
a complete answer to this
question for groups of the form $G=\F_p^n$, provided $p$ is not too small.
In $\Z_N$, substantial extra difficulties arise, of which the most important is
that an ``inverse theorem'' even for the uniformity norm $\|.\|_{U^3}$ requires a more sophisticated (local) formulation. When $N$ is prime, $\Z_N$ is not rich in subgroups, so one must use regular Bohr neighbourhoods instead.
In this paper, we prove the first non-trivial case of the main conjecture
from \cite{Gowers:2007tcs}.
\end{abstract}

\maketitle
\tableofcontents
\newpage

\section{Introduction}

In additive combinatorics one is often interested in counting small structures
in subsets of Abelian groups. For instance, one formulation of Szemer\'edi's
theorem is the assertion that if $\d>0$, $k$ is a positive integer and $N$ is large 
enough, then every subset $A$ of $\Z_N$ of density at least $\d$ contains
many arithmetic progressions of length $k$. 

There is also an equivalent formulation of the theorem concerning functions,
the formal statement of which is as follows.


\begin{theorem}\label{szemeredi}
Let $\d>0$ and let $k$ be a positive integer. Then there is a constant $c=c(\d,k)>0$
such that for every $N$ and every function $f:\Z_N\ra[0,1]$ with $\E_xf(x)\geq\d$,
\begin{equation*}
\E_{x,d}f(x)f(x+d)\dots f(x+(k-1)d)\geq c.
\end{equation*}
\end{theorem}

\ni Here we use the symbol ``$\E$'' to denote averages over $\Z_N$. For 
instance, $\E_{x,d}$ is shorthand for $N^{-2}\sum_{x,d\in\Z_N}$. If we take
$f$ to be the characteristic function of a set $A$ of density $\d$, then the 
conclusion of Theorem \ref{szemeredi} states that the number of arithmetic
progressions of length $k$ in $A$, or rather the number of pairs $(x,d)$ such
that $x,x+d,\dots,x+(k-1)d$ all lie in $A$, is at least $c(\d,k)N^2$. (It is not
necessary to assume that $N$ is sufficiently large, because for small $N$
the degenerate progressions where $d=0$ are numerous enough for
the theorem to be true. But $N$ has to be large for it to become a non-trivial
statement.)

There are now several known ways of proving Szemer\'edi's theorem. One of
them, an analytic approach due to the first author \cite{Gowers:2001tg}, relies heavily
on the fact that the quantity $\E_{x,d}f(x)f(x+d)\dots f(x+(k-1)d)$ is robust under
perturbations of $f$ that are quasirandom in a suitable sense. More precisely,
in \cite{Gowers:2001tg} a norm $\|.\|_{U^k}$ was defined for each $k$, which has 
the property that if $\seq f k$ are functions with $\|f_i\|_\infty\leq 1$ for every $i$, then
\begin{equation} \label{apcontrol}
|\E_{x,d}f_1(x)f_2(x+d)\dots f_k(x+(k-1)d)|\leq\min_i\|f_i\|_{U^{k-1}}
\end{equation}
From this it is simple to deduce that if $f$ and $g$ are two functions from 
$\Z_N$ to $[0,1]$, then
\begin{equation*}
\E_{x,d}f(x)f(x+d)\dots f(x+(k-1)d)-\E_{x,d}g(x)g(x+d)\dots g(x+(k-1)d)
\end{equation*}
has magnitude at most $k\|f-g\|_{U^{k-1}}$. If we choose a function $h$ randomly
by taking the values $h(x)$ to be bounded random variables of mean 0, then with
high probability $\|h\|_{U^k}$ will be very small. Thus, the $U^k$ norms are
measures of a certain kind of quasirandomness that is connected with cancellations
in expressions such as $\E_{x,d}h(x)h(x+d)\dots h(x+(k-1)d)$.

The proof of inequality (\ref{apcontrol}) is a relatively straightforward inductive
argument that involves repeated application of the Cauchy-Schwarz inequality. Once
one has this argument, it is natural to try to generalize it to other expressions
such as $\E_{x,y,z}f(x+y)f(x+z)f(y+z)$ or $\E_{x,d}f(d)f(x)f(x+d)f(x+2d)$. In general, one
can take a system of linear forms $L_1,\dots,L_r$ in $k$ variables $x_1,\dots,x_s$
(that is, for each $i$ we write $x$ for $(x_1,\dots,x_s)$ and define 
$L_i(x)=\sum_{j=1}^sa_{ij}x_j$ for some integers $a_{i1},\dots,a_{is}$) and 
examine the quantity $\E_x\prod_{i=1}^rf(L_i(x))$, or the more general quantity
$\E_x\prod_{i=1}^rf_i(L_i(x))$. We would then like to know for which uniformity
norms $U^k$ it is true that these expressions are robust under small $U^k$ 
perturbations.

This question was first addressed by Green and Tao \cite{Green:2006lep}, 
who were interested in proving asymptotic estimates for expressions such as 
these when $f$ is the characteristic function of the primes up to $n$ (or rather the
closely related von Mangoldt function). They defined a notion of complexity
for a system of linear forms. This is a positive integer $k$ with the property
that if a system has complexity $k$, then the corresponding analytic expression 
will be robust under small $U^{k+1}$ perturbations. (As we shall see, there are 
good reasons for defining complexity in a way that leads to this difference of 1.)
Roughly speaking, the property they identified picks out the minimal $k$ for
which repeated use of the Cauchy-Schwarz inequality can be used to prove 
robustness under small $U^{k+1}$ perturbations. 

However, it turns out that there are some systems of linear forms of complexity
$k$ that are robust under small perturbations in $U^{j+1}$ for some $j<k$.
(Since the $U^k$ norms increase as $k$ increases, the assumption that
a function is small in $U^{j+1}$ is weaker than the assumption that it is
small in $U^{k+1}$.) This phenomenon was first demonstrated in \cite{Gowers:2007tcs}, 
where we showed that if $G$ is the group $\F_p^n$, then there is a system of linear
forms of complexity $2$ such that the corresponding analytic expression
is robust under small $U^2$ perturbations. (A similar phenomenon in
ergodic theory was discovered independently by Leibman \cite{Leibman:2007odp}.) Because Green and Tao's
definition, appropriately modified, appears to capture all systems of linear 
forms for which Cauchy-Schwarz-type arguments work (though we have not
actually formulated and proved a statement along these lines), one must 
use additional tools. The particular tool we used was a new technique known as
\textit{quadratic Fourier analysis}, which we shall discuss in some detail in 
\S \ref{quadfourier}. A weak ``local" form of quadratic Fourier analysis was 
introduced and used in \cite{Gowers:2001tg} to prove Szemer\'edi's theorem 
(for progressions of length 4 -- higher order Fourier
analysis was needed for the general case). A more ``global" version was 
developed by Green and Tao \cite{Green:2008py} and will be essential to 
this paper.

In \cite{Gowers:2007tcs} we made the following conjecture. 


\begin{conjecture}\label{mainconjecture}
Let $L_1,\dots,L_r$ be a system of linear forms in $x=(x_1,\dots,x_s)$ and let 
$G$ be the group $\Z_N$ or $\F_p^n$ for sufficiently large $p$. Suppose also
that the $k$th powers of the forms $L_i$ are linearly independent. Then 
$\E_x\prod_{i=1}^rf(L_i(x))$ is close to $\E_x\prod_{i=1}^rg(L_i(x))$
whenever $f$ and $g$ are bounded functions and $\|f-g\|_{U^k}$ is small.
\end{conjecture}

It is not hard to prove the converse of this conjecture, so if it is true then it
identifies precisely the minimal uniformity norm with respect to which the 
multilinear expression derived from the linear forms is continuous (where
by ``continuous" we mean continuous in a way that does not depend on
the size of the group): it is given by the smallest $k$ such that the $k$th
powers of the linear forms are linearly independent. In such a case, we
shall say that the forms are \textit{$k$th-power independent}. When $k=2$
we shall say that they are \textit{square independent}. We also formulated 
a more general conjecture that covers the case of $m$ different functions. 

In \cite{Gowers:2009lfuII} we proved Conjecture \ref{mainconjecture} in the case where 
$G=\F_p^n$, using the very recent inverse theorem for the $U^k$ norm in 
that context, which was proved by Bergelson, Tao and Ziegler \cite{Bergelson:2009itu,Tao:2008icg}. 
This inverse theorem opens the way to cubic Fourier analysis, quartic Fourier analysis, 
and so on. Our result is the first application of this higher-order Fourier analysis. The
first application of higher-order Fourier analysis on $\Z_N$, which has recently become a theorem (though so far only the $k=4$ case is available \cite{Green:2009u4i}), is to linear equations in the primes:  Green and Tao have already obtained asymptotics for the numbers of solutions for all systems of finite complexity, conditional on the inverse conjecture for the $U^k$ norm in $\Z_N$, which they have now proved with Ziegler.

Quadratic Fourier analysis on $\Z_N$ has had other applications. For
example,  a modification of Theorem \ref{znquadfouriersmallk} was used 
by Candela \cite{Candela:2008cdf} to prove that if $A$ is a dense subset of $\{1,2,\dots,n\}$ 
then the set of all $d$ such that $A$ contains an arithmetic
progression of length 3 and common difference $d$ must itself contain
an arithmetic progression of length at least $(\log\log N)^c$.

In \cite{Gowers:2007tcs} we proved the first non-trivial case of Conjecture 
\ref{mainconjecture} for $\F_p^n$, which is the case of square-independent
systems of complexity 2. However, we obtained a bound of tower type, so
from a quantitative point of view this result was not very satisfactory. 
In \cite{Gowers:2009lfuI}, we improved this bound to one that was doubly 
exponential. The general inverse theorem for functions on $\F_p^n$ has 
so far been proved only as a purely qualitative
statement, so we did not obtain any bounds at all for the other cases of 
Conjecture \ref{mainconjecture}. In this paper, we shall prove Conjecture
\ref{mainconjecture} for square independent systems of complexity 2
in $\Z_N$. In other words, if $f$ and $g$ are bounded functions
and $L_1,\dots,L_r$ is a square independent system of complexity 2, 
we are interested in how small the $U^2$ norm $\|f-g\|_{U^2}$ has to be to
guarantee that $\E_x\prod_{i=1}^rf(L_i(x))$ is within $\e$ of
$\E_x\prod_{i=1}^rg(L_i(x))$. We go to considerable efforts to obtain
a respectable bound, which in the end is a doubly exponential dependence 
on $\e$. If we had worked less hard then we would have had to settle for a
tower-type bound. To obtain the good bound (relatively speaking) we 
shall use some of the ideas from \cite{Gowers:2009lfuI} as well as  
some new ideas to deal with problems that do not arise in $\F_p^n$. 

The big difference between $\F_p^n$ and $\Z_N$ is that $\F_p^n$
has many subgroups that closely resemble $\F_p^n$ itself. If $N$ is
prime, then $\Z_N$ has no non-trivial subgroups at all, and it becomes
necessary to consider subsets that are ``approximately closed" under
addition. These subsets are called regular Bohr sets, and we shall
discuss them in the next section. Here we remark that the notion of
a Bohr set originated in the study of almost periodic functions and
has played a very important role in additive combinatorics 
since Ruzsa's pioneering proof \cite{Ruzsa:1994gap} of Freiman's theorem 
\cite{Freiman:1966fst}. The additional hypothesis of regularity, which makes
it possible to treat Bohr sets like subgroups, was introduced by 
Bourgain \cite{Bourgain:1999tap} and has subsequently been used by
several authors. 

By proving our results first for $\F_p^n$ and then adapting the
arguments to the $\Z_N$ context, we are following a general course urged
by Green in \cite{Green:2006mln}. The reason for doing it is that
it splits problems into two parts. The first part, which is in a sense more
fundamental, is to get one's result in a model context where certain
distracting technicalities do not arise. Once one has done that,
one has a global structure for the proof, and one can usually 
find a proof in $\Z_N$ that has the same global structure as the
proof in $\F_p^n$.

That is the case for our result, so although we have made
this paper self-contained, the reader will almost certainly prefer to 
begin by reading \cite{Gowers:2009lfuI}. However, the adaptation
of our arguments to $\Z_N$ is by no means a completely mechanical 
process. Some parts are, by now, fairly routine, but certain 
concepts that are quite useful for proving results in $\F_p^n$ do
not have obvious analogues in $\Z_N$, and some lemmas that are
almost trivial in $\F_p^n$ become serious statements with
non-obvious proofs in $\Z_N$. We shall highlight the less
obvious parts of the adaptation as they arise, since some of 
them may well find other uses.

Very recently indeed, Green and Tao \cite{Green:2010ar} have proved 
Conjecture \ref{mainconjecture} in full generality in $\Z_N$, using the recent
inverse theorem. Their method is completely different from ours,
which was almost certainly necessary: it seems that they
have found the right framework for studying the problem if one
is content with arguments that do not give reasonable bounds. 
However, in order to obtain the quantitative statement that we
prove here, it seems to be necessary (at least given the technology
as it is at present) to use different, more ``old-fashioned" techniques.
For the time being a proof of the full conjecture with good bounds looks
out of reach: not the least of the difficulties would be obtaining a 
quantitative version of the inverse theorem.

\section{Bohr sets and their basic properties}

Let $K$ be a subset of $\Z_N$ and let $\rho>0$. The \textit{Bohr
set} $B(K,\rho)$ is the set of all $x\in\Z_N$ such
that $|\omega^{rx}-1|\leq\d$ for every $r\in K$. As we have 
just said, Bohr sets will play the role that subgroups 
played for functions defined on $\F_p^n$. However, they are not
closed under addition, and this causes problems.

The way to deal with these problems is to use the fact that 
Bohr sets do have at least some closure properties.
In particular, if $x\in B(K,\rho)$ and $y\in B(K,\sigma)$,
then $x+y\in B(K,\rho+\sigma)$. To use this fact, one takes
$\sigma$ small enough for $B(K,\rho+\sigma)$ to be approximately
equal to $B(K,\rho)$.

However, such an approach can work only if the size of the set 
$B(K,\rho)$ depends sufficiently continuously on $\rho$, which
is not always the case. This fact motivated an important
definition due to Bourgain \cite{Bourgain:1999tap}. Let $B=B(K,\rho)$
be a Bohr set. $B$ is said to be \textit{regular} if, for 
every $\e>0$, the Bohr set $B(K,\rho(1+\e))$ has
cardinality at most $|B|(1+100|K|\e)$ and the Bohr set
$B(K,\rho(1-\e))$ has cardinality at least $|B|(1-100|K|\e)$.
The precise form of this definition is what comes out of the
following lemma (see for example \cite{Tao:2006ac}), which tells us that it is easy to find regular
Bohr sets.


\begin{lemma} \label{regularbohrsets}
Let $K$ be a subset of $\Z_N$ and let $\rho_0>0$. Then there exists
$\rho$ such that $\rho \in[\rho_0,2\rho_0]$ and the Bohr set
$B(K,\rho)$ is regular. 
\end{lemma}

It will be useful to have a concise notation that allows us to talk about pairs of Bohr sets 
that have the approximate closure property under addition.

\begin{definition}\label{regularpairs}
Let $B$ be a regular Bohr set $B(K,\rho)$. Then we say that a subset 
$B'\subset B$ is $\e$-\emph{central for} $B$, and write $B'\prec_\e B$, if 
$B'=B(K,\sigma)$ for some  $\sigma\in [\e\rho/400|K|, \e\rho/200|K|]$ and
$B'$ is also regular. Given a pair of Bohr sets $B'\prec_\e B$, we define the
\emph{closure} of $B$ to be the set $B^{+}=B(K, \rho+\sigma)$ and the 
\emph{interior} to be the set $B^{-}=B(K, \rho-\sigma)$. 
\end{definition}

\ni The definitions of closure and interior depend on the central set $B'$, so they
cannot be used unless $B'$ has been specified. But this does not cause any problems.
\smallskip

Because we are dealing with quadratic rather than linear local Fourier analysis, we will 
sometimes have to repeat the closure and interior operations, which, unlike their topological
counterparts, are not idempotent (and therefore not strictly speaking closure and interior 
operations at all). Thus, we define $B^{++}$ to be $B(K, \rho+2\sigma)$ and 
$B^{--}$ to be $B(K, \rho-2\sigma)$.

Note that in many of the early lemmas we do not actually need the central Bohr set $B'$ 
to be regular. However, we often apply a sequence of such lemmas, so it is convenient
to insist on regularity at all times.

There are many closely related ways of using the regularity condition
on a Bohr set. The next lemma, which will be used later, is
a typical one. It exploits the fact that regular Bohr sets have ``small
boundaries''.


\begin{lemma}\label{boundarylemma}
Let $K$ be a subset of $\Z_N$ and let $\seq x m$ be a sequence of $m$ 
elements of $\Z_N$. Suppose that the Bohr sets $B$ and $B'$ satisfy 
$B' \prec_{\e} B$. Then for all but at most $\e m|B|$ values of $x$ the
following statement is true: for every $i$, $B'+x$ is either contained
in $B+x_i$ or disjoint from it.
\end{lemma}

\begin{proof}
If $x-x_i\in B^{-}$, then $B'+x-x_i$ is a subset
of $B^{-}+B'\subset B$, and therefore
$B'+x\subset B+x_i$. Similarly, but in the other direction, if 
$(B'+x)\cap(B+x_i)\ne\emptyset$, then $x-x_i\in B^{+}$.
Therefore, the only way that $B'+x$ can fail to be either contained in 
$B+x_i$ or disjoint from it is if 
$x-x_i\in B^{+}\setminus B^{-}$. However,
by the definition of regularity, the cardinality of
$B^{+}\setminus B^{-}$ is at most $\e|B|$. The lemma follows.
\end{proof}

Another very useful principle indeed is that if $B$ is a regular Bohr 
set, $B'$ is a central subset, and $f$ is a
bounded function, then $\E_{x\in B}f(x)$ is approximately equal to
$\E_{x\in B}\E_{y\in B'}f(x+y)$. Indeed, this is the most common way
that regularity has been applied. We shall need some less standard
(but not difficult) variants of this principle---for the convenience 
of the reader we give proofs of all the results of this kind that
we need. We shall use the notation ``$\approx_\e$" to stand for the
relation ``differs by at most $\e$ from''.


\begin{lemma}\label{bohraveraging}
Let $\eps>0$. Let $B$ and $B'$ be Bohr sets satisfying $B' \prec_{\eps}B$. Then for 
every function $f:\Z_N\ra\C$ such that $\|f\|_\infty\leq 1$ and for every function 
$g:\Z_N^2\ra\C$ such that $\|g\|_\infty\leq 1$ the following statements hold.

(i) $\E_{x\in B}f(x)\approx_\e\E_{x\in B}\E_{y\in P}f(x+y)$ for every subset $P\subset B'$.

(ii) $\E_{x\in B}f(x)\approx_{\e}\E_{x\in B^-}f(x)$.

(iii) $\E_{x\in B^-}f(x)\approx_{3\e}\E_{x\in B^-}\E_{y\in B'}f(x+y)$.

(iv) $\E_{x,x'\in B}g(x,x')\approx_{4\e}\E_{x,x'\in B}\E_{y\in B'}g(x+y,x'+y)$

(v) $\E_{x,x'\in B^-}g(x,x')\approx_{8\e}\E_{x,x'\in B^-}\E_{y\in B'}g(x+y,x'+y)$.
\end{lemma}

\begin{proof}
Since $\|f\|_\infty\leq 1$, for every $y\in B'$ we have the inequality
\[|\E_{x\in B}f(x+y)-\E_{x\in B}f(x)|\leq|B|^{-1}|B\bigtriangleup(B+y)|.\]
But $B\bigtriangleup(B+y)\subset B^{+}\setminus B^{-}$,
so the right hand side is at most $\e$, by the regularity of $B$.
Since $\E_{x\in B}\E_{y\in P}f(x+y)=\E_{y\in P}\E_{x\in B}f(x+y)$,
part (i) follows from the triangle inequality. To prove (ii), we begin by noting that 
\[\Bigl|\E_{x\in B}f(x)-|B|^{-1}\sum_{x\in B^-}f(x)\Bigr|\leq |B|^{-1}|B\setminus B^-|.\]
By regularity, the right hand side is at most $\e/2$. It is also easy to check that
\[\Bigl|\E_{x\in B^-}f(x)-|B|^{-1}\sum_{x\in B^-}f(x)\Bigr|\leq \e/2.\]
It follows that $|\E_{x\in B}f(x)-\E_{x\in B^-}f(x)|\leq \e$. Applying (ii) to both sides of (i), we deduce (iii). The proof of (iv) is very similar to that of (i). For each $y \in B'$ we have 
the inequality
\[|\E_{x,x'\in B}g(x+y,x'+y)-\E_{x,x'\in B}g(x,x')|\leq|B|^{-2}|B^2\bigtriangleup(B+y)^2|.\]
From the fact that $|B\bigtriangleup(B+y)|\leq\e|B|$ it follows that
$|B^2\bigtriangleup(B+y)^2|\leq 4\e|B|^2$. This implies (iv), just
as the analogous statement implied (i). Finally, if we apply (ii) twice to both sides of (iv) we obtain (v).
\end{proof}

\section{Quadratic Fourier analysis on $\Z_N$}\label{quadfourier}

Conventional Fourier analysis on an Abelian group $G$ decomposes a function
$f:G\ra\C$ into a linear combination of characters, which are homomorphisms 
from $G$ to $\T=\{z\in\C:|z|=1\}$. If we allow ourselves a phase shift---that is, if
we multiply a character by $e^{i\theta}$ for some $\theta$---then we obtain
a function $\g$ that may not be a group homomorphism, but it is still a (multiplicative)
\textit{Freiman homomorphism}, since it satisfies the identity 
$\g(x+d)\g(x)^{-1}=\g(y+d)\g(y)^{-1}$ for every $x,y$ and $d$ in $G$.

Quadratic Fourier analysis replaces Freiman homomorphisms by a natural quadratic
analogue. We can restate the identity above as $\g(x)\g(x+a)^{-1}\g(x+b)^{-1}\g(x+a+b)=1$ 
for every $x$, $a$ and $b$ in $G$. If $A\subset G$, then a function $\g:A\rightarrow\T$
is a (multiplicative) \textit{quadratic homomorphism} if
\begin{equation*}
\g(x)\g(x+a)^{-1}\g(x+b)^{-1}\g(x+c)^{-1}\g(x+a+b)\g(x+a+c)\g(x+b+c)\g(x+a+b+c)^{-1}=1
\end{equation*}
for every $x$, $a$, $b$ and $c$ in $G$. The word ``quadratic" is used because if $A=G=\Z_N$,
then $\g$ has to be of the form $\g(x)=e^{2\pi iq(x)/N}$ for some quadratic function 
$q:\Z_N\ra\Z_N$, and similar statements are true for several other groups. Because of this,
we shall also refer to these functions as \textit{quadratic phase functions}. For more general 
subsets $A$ it is less easy to describe quadratic homomorphisms explicitly, but if $A$ is a
sufficiently structured set, such as a coset of a subgroup of $\F_p^n$ (when $p$ is 
not too small) or a Bohr set in $\Z_N$, then for many purposes it is enough just to 
know that $\g$ is a quadratic homomorphism, though in these cases one can also 
give explicit descriptions and it is sometimes important to do so.

The basic idea of quadratic Fourier analysis is that it is possible to decompose
a function into a linear combination of a small number of quadratic phase functions
defined on regular Bohr sets, plus an error that does not affect calculations.
One can of course do the same with conventional Fourier analysis simply by
taking only the characters with large coefficients: however, there are circumstances
where the error \textit{does} affect calculations in the linear case, but does not
in the quadratic case.

A notable difference between linear and quadratic Fourier analysis is that there is
not a unique way of decomposing a function into quadratic parts, for the simple
reason that there are too many quadratic phase functions. Furthermore, there
is not even a natural notion of the ``best'' decomposition. So instead one has to
settle for decompositions that are somewhat arbitrary and try to control their
properties. In order to get started, one needs an \textit{inverse theorem},
which in our case is a statement to the effect that if $\|f\|_{U^3}$ is not small
(which is a way of saying that $f$ is not already a ``small error") then 
$f$ correlates with a quadratic phase function.

The following theorem to this effect was proved by Green and Tao \cite{Green:2008py}.


\begin{theorem}\label{inverseth}
Let $f:\Z_N\ra\C$ be a function such that
$\|f\|_\infty\leq 1$ and $\|f\|_{U^3}\geq\delta$, and let $C=2^{24}$.
Then there exists a regular Bohr set $B=B(K,\rho)$ with
$|K|\leq(2/\d)^C$ and $\rho\geq(\d/2)^C$ such that
$\E_y\|f\|_{u^3(B+y)}\geq(\d/2)^C$.
\end{theorem}

Here, $\|f\|_{u^3(B+y)}$ is defined to be the maximum correlation between
$f$ and any quadratic phase function $\g$ defined on $B+y$.
More precisely, it is the maximum over all quadratic phase functions
$\g$ from $B+y$ to $\T$ of the quantity $|\E_{x\in B+y}f(x)\g(x)^{-1}|$.

In their paper, Green and Tao remark that a slightly
more precise theorem holds. The result as stated tells us that for
each $y$ we can find a quadratic phase function $\omega^{q_y}$
defined on $B+y$ such that the average of $|\E_{x\in
B+y}f(x)\omega^{q_y(x)}|$ is at least $(\d/2)^C$. However, it is
actually possible to do this in such a way that the ``quadratic
parts'' of the quadratic phase functions $q_y$ are the same. That 
is, it can be done in such a way that each $q_y(x)$ has the
form $q(x-y)+\phi_y(x-y)$ for some (additive) quadratic homomorphism 
$q:B\ra\Z_N$ (that is independent of $y$) and some Freiman homomorphism
$\phi_y:B\ra\Z_N$.

This will be convenient to us later, so we make the following definition,
which is a modification of a definition given in \cite{Gowers:2009lfuII}
for the $\F_p^n$ case.
 
\begin{definition} Let $B$ be a regular Bohr set and let $q$ be
a quadratic map from $B$ to $\Z_N$. A \emph{quadratic average with
base} $(B,q)$ is a function of the form $Q(x)=\E_{y\in
x-B}\omega^{q_y(x)}$, where each function $q_y$ is a quadratic map
from $B+y$ to $\Z_N$ defined by a formula of the form
$q_y(x)=q(x-y)+\phi_y(x-y)$ for some Freiman homomorphism
$\phi_y:B\ra\Z_N$. 
\end{definition}

An equivalent way of defining $Q$, which may be clearer, is to 
start by defining for each $y\in\Z_N$ the function $\g_y$, which
takes the value $\omega^{q_y(x)}$ when $x\in B+y$ and 0
otherwise. Then $Q$ is $|B|^{-1}\sum_y\g_y$. Thus, the value
of $Q$ at $x$ is the average value of all the $\g_y(x)$ such
that $x$ belongs to the support of $\g_y$. 

We can use the extra observation of Green and Tao to give a 
slightly more precise version of the inverse theorem.


\begin{theorem}\label{znfancyinverse}
Let $f:\Z_N\ra\C$ be a function such that $\|f\|_\infty\leq 1$ and
$\|f\|_{U^3}\geq\delta$, and let $C_0=2^{24}$.  Then there exists a
regular Bohr set $B(K,\rho)$ with $|K|\leq(2/\d)^{C_0}$ and
$\rho\geq(\d/2)^{C_0}$, and a quadratic map $q:B\ra\Z_N$, such that
$|\sp{f,Q}|\geq(\d/2)^{C_0}/2$ for some quadratic average $Q$ with base
$(B,q)$.
\end{theorem}

\begin{proof}
The results of Green and Tao tell us that we can find a regular 
Bohr set $B=B(K,\rho)$, satisfying the above bounds, and
a quadratic function $q$, and for each $y$ we can find a
Freiman homomorphism $\phi_y:B\ra\Z_N$, such that, defining
$q_y(x)=q(x-y)+\phi_y(x-y)$ on $B+y$, we have
\[\E_y|\E_{x\in B+y}f(x)\omega^{-q_y(x)}|\geq(\d/2)^C\]
For each function $q_y$ we can add a constant $\lambda_y$ without
affecting the left-hand side. If $N\geq 3$, as we are certainly
assuming, then we can choose this constant so that

\[\Re(\E_{x\in B+y}f(x)\omega^{-q_y(x)+\lambda_y})
\geq \frac{1}{2}|\E_{x\in B+y}f(x)\omega^{-q_y(x)}|\]
Therefore, after suitably redefining the functions $q_y$ and
setting $Q(x)=\E_{y\in x-B}\omega^{q_y(x)}$, we have
\[
|\sp{f,Q}|\geq \Re(\E_x\E_{y\in x-B}f(x)\omega^{-q_y(x)})
\geq \frac{1}{2}\E_y|\E_{x\in B+y}f(x)\omega^{-q_y(x)}|,
\]
which proves the theorem.
\end{proof}

In the proof of the $\F_p^n$ case, we defined quadratic averages
in a similar way, but the role of Bohr sets was played by subgroups
(or subspaces). This was simpler for several reasons. One reason
was that translates of a subspace partition $\F_p^n$, but this
turns out not to be a significant complication of the $\Z_N$ case.
More problematic is that we made some use of the fact
that subspaces of $\F_p^n$ have a codimension, and it is not 
obvious what one would mean by the ``codimension'' of a Bohr
set. To answer this question, we focus on the two main properties of 
codimension that we used for the $\F_p^n$ case: that a subspace of codimension
$d$ has density $p^{-d}$ and that the intersection of subspaces of
codimension $d$ and $d'$ has codimension at most $d+d'$. The analogous
facts about Bohr neighbourhoods are that the intersection of the
neighbourhoods $B(K,\rho)$ and $B(L,\rho)$ equals the neighbourhood
$B(K\cup L,\rho)$, and that the density of $B(K,\rho)$ is at least
$\rho^{|K|}$. Thus, for fixed $\rho$ the cardinality of $K$ is a good
analogue of the codimension. 

At first, this seems odd, since the
cardinality of $K$ is closely connected with the \textit{dimension}
of $B$. However, it can also be seen as the number of inequalities
that a point in $B$ must satisfy, and these inequalities are analogous
to the linear constraints that a point in a subspace must satisfy.
Nevertheless, to avoid confusion we will not use the word 
``codimension'' here. Instead, we shall define the \textit{complexity}
of the Bohr set $B(K,\rho)$ to be the pair $(|K|,\rho)$.
Strictly speaking, this is not well-defined, since different pairs
$(K,\rho)$ can define the same Bohr set. So a slightly
stricter definition is as follows: the Bohr set $B$ has
\textit{complexity at most} $(d,\rho)$ if there exists a set $K$
of cardinality at most $d$ and a constant $\rho'\geq\rho$ such 
that $B=B(K,\rho')$. (We say ``at most'' because we regard a
smaller $\rho$ as giving a higher complexity.) We say that a 
quadratic average $Q$ with base $(B,q)$ has complexity at most
$(d,\rho)$ if $B$ has complexity at most $(d,\rho)$. 

Now, as we did in the $\F_p^n$ case, we can use fairly abstract
reasoning to deduce some decomposition results from the inverse 
theorem. First, we recall a result from \cite{Gowers:2009lfuI}. It is
a straightforward consequence of the Hahn-Banach theorem and
appears in \cite{Gowers:2009lfuI}, with proof, as Corollary 2.4.
It can be thought of as a general machine for converting inverse
theorems into decomposition theorems.


\begin{proposition}\label{hbcor}
Let $k$ be a positive integer and for each $i\leq k$ let 
$\|.\|_i$ be a norm defined on a subspace $V_i$ of $\C^n$.
Suppose also that $V_1+\dots+V_k=\C^n$. Let 
$\alpha_1,\dots,\alpha_k$ be positive real numbers, and suppose 
that it is not possible to write the function $f$ as a linear
sum $f_1+\dots+f_k$ in such a way that $f_i\in V_i$ for each $i$
and $\alpha_1\|f_1\|_1+\dots+\alpha_k\|f_k\|_k\leq 1$. Then 
there exists a function $\phi\in\C^n$ such that $|\langle f,\phi\rangle|\geq 1$
and such that $\|\phi\|_i^*\leq\alpha_i$ for every $i$.
\end{proposition}

The final condition on $\phi$ means
that $|\langle g,\phi\rangle|\leq\alpha_i$ for every 
$i$ and every $g\in V_i$ with $\|g\|_i\leq 1$.

We now apply Proposition \ref{hbcor} to obtain a theorem that tells
us that an arbitrary function $f$ that is bounded in $L_2$ can be 
decomposed as a linear combination of quadratic averages plus
a small error.


\begin{theorem} \label{znfancyquadfourier}
Let $f:\Z_N\ra\C$ be a function such that $\|f\|_2\leq 1$. Let
$C_0=2^{24}$. Then
for every $\d>0$ and $\eta>0$ there exist $C$, $d$ and $\rho$ such 
that $f$ has a decomposition of the form
\[
f(x)=\sum_i\lambda_iQ_i(x)+g(x)+h(x),
\]
where the functions $Q_i$ are quadratic averages of complexity
at most $(d,\rho)$, and
\[
\eta^{-1}\|g\|_1+\d^{-1}\|h\|_{U^3}+C^{-1}\sum_i|\lambda_i|\leq 1.
\]
Moreover, we can take $C=4(2/\eta\d)^{C_0}$, 
$d=(2/\d)^{C_0}$ and $\rho=(\d/2)^{C_0}$.
\end{theorem}

\begin{proof}
For every quadratic average $Q$ on $\Z_N$
of complexity at most $(d,\rho)$, 
let $V(Q)$ be the one-dimensional subspace of $\C^{\Z_N}$ 
generated by $Q$, with the norm of $\lambda Q$ set to be
$|\lambda|$. Let $\a(Q)$ be $C^{-1}$ for every $Q$. In addition,
let us take the $L_1$ norm and $U^3$ norm defined on all of 
$\C^{\Z_N}$ and associate with them the constants $\eta$ and
$\d$, respectively.

Suppose that $f$ cannot be decomposed in the desired way. Applying Proposition \ref{hbcor} to the norms, subspaces and positive constants defined above, we obtain a function $\phi:\Z_N\ra\C$ such that 
$\sp{f,\phi}\geq 1$, $\|\phi\|_\infty\leq\eta^{-1}$, $\|\phi\|_{U^3}^*\leq\d^{-1}$
and $|\sp{\phi,Q}|\leq C^{-1}$ for every quadratic average $Q$
of complexity at most $(d,\rho)$. 

Because $\|f\|_2\leq 1$ and $\sp{f,\phi}\geq 1$, we find that
$\sp{\phi,\phi}=\|\phi\|_2\geq 1$. But then 
$\|\phi\|_{U^3}\|\phi\|_{U^3}^*\geq 1$, which implies that
$\|\phi\|_{U^3}\geq\d$. Applying Theorem \ref{znfancyinverse}
to $\eta\phi$, we obtain a quadratic average $Q$ of complexity
at most $(d,\rho)$ such that
$|\sp{\phi,Q}|\geq(\eta\d/2)^{C_0}/2$, which is a contradiction
since this inner product was supposed to be at most $C^{-1}$
for all such quadratic averages.
\end{proof}

\section{Generalized quadratic averages}

Before we go any further, we must address a technical issue
that did not arise for $\F_p^n$. There, it is a triviality that
if $Q$ is a quadratic average with base $(V,q)$ and $Q'$ is
a quadratic average with base $(V',q')$, then $Q\ol{Q'}$ is
a quadratic average with base $(V\cap V',q-q')$. The analogous
statement for $\Z_N$ is false, but an approximate version of
it is true if we are prepared to generalize the notion of
a quadratic average.

In fact, we shall begin by discussing an even more basic 
statement, which again does not quite hold for $\Z_N$, 
namely the statement that if $Q$ is a quadratic average
with base $(V,q)$ and $V'$ is a subspace of $V$, then
$Q$ is a quadratic average with base $(V',q)$. 

In order to obtain an analogue of this statement for $\Z_N$,
we first define a \textit{generalized quadratic average} with
base $(B,q)$ to be any average of quadratic averages with
base $(B,q)$---that is, any function of the form
$n^{-1}(Q_1+\dots+Q_n)$, where $n$ is a positive integer
and each $Q_i$ is a quadratic average with base $(B,q)$. 


\begin{lemma} \label{smallerbase}
Let $\e>0$, let the Bohr sets $B$ and $B'$ satisfy 
$B'\prec_{\eps}B$. Let $q$ be a quadratic form on $B$ and let 
$Q$ be a generalized quadratic average with base $(B,q)$. 
Then there is a generalized quadratic average $Q'$ with base $(B',q)$
such that $\|Q-Q'\|_\infty\leq 4\e$.
\end{lemma}

\begin{proof}
We begin by proving the result when $Q$ is a quadratic average,
defined by the formula $Q(x)=\E_{y\in x-B}\omega^{q_y(x)}$.  
Let $A$ be the interior associated with the pair $B'\prec_{\eps}B$. 
From the proof of Lemma \ref{bohraveraging} (ii), 
we know that 
\[\E_{y\in x-B}\omega^{q_y(x)}\approx_{\e}\E_{y\in x-A}\omega^{q_y(x)},\]
so if we set $R(x)$ to equal $\E_{y\in x-A}\omega^{q_y(x)}$ then
$\|R-Q\|_\infty\leq \e$. 

Next, we define $S(x)$ to be $\E_{y\in x-A}\E_{u\in -B'}\omega^{q_{y+u}(x)}$.
By Lemma \ref{bohraveraging} (iii), $\|S-R\|_\infty\leq 3\e$. But $S(x)$ is 
equal
to $\E_{y\in -A}\E_{u\in x-B'}\omega^{q_{y+u}(x)}$. For each $y\in -A$
and $u\in B'$, $q_{y+u}$ is a local quadratic that is defined everywhere 
on $B+y+u$, and hence on $B'+u$ (since $B'-y\subset B'+A\subset B$).
Therefore, since translating a local quadratic has the effect of
adding a Freiman homomorphism, for each $y\in -A$ the function
$Q_y(x)=\E_{u\in x-B'}\omega^{q_{y+u}(x)}$ is a quadratic average
with base $(B',q)$. It follows that $S$ is a generalized quadratic
average with base $(B',q)$. By the triangle inequality, 
$\|Q-S\|_\infty\leq 4\e$.

The result for generalized quadratic averages follows easily: one 
simply applies the result just proved to each individual quadratic average
and uses the triangle inequality.
\end{proof}


\begin{lemma} \label{quadaveproduct}
Let $\e>0$. Let be $B_{1}$ and $B_{2}$ be regular Bohr 
sets, let $q_1$ and $q_2$ be quadratic functions defined
on them and let $Q_1$ and $Q_2$ be generalized quadratic averages with bases
$(B_{1},q_1)$ and $(B_{2},q_2)$, respectively. Suppose that the Bohr sets $B$ and $B'$ satisfy the relations $B' \prec_{\eps} B \prec_{\eps} B_{1}\cap B_{2}$. Then there exists a generalized quadratic average $Q'$ with base $(B',q_1-q_2)$
such that $\|Q_1\ol{Q_2}-Q'\|_\infty\leq 18\e$.
\end{lemma}

\begin{proof}
The argument is similar to the proof of the previous lemma, but
slightly more complicated. First of all, by that lemma we can
uniformly approximate both $Q_1$ and $Q_2$ to within $4\e$ by
generalized quadratic averages $Q_1'$ and $Q_2'$ with bases $(B,q_1)$ and $(B,q_2)$,
respectively. Suppose that $Q_i'=n_i^{-1}(Q_{i,1}+\dots+Q_{i,n_i})$. If we can
find for each pair $(r,s)$ a generalized quadratic average $Q_{rs}$
such that $\|Q_{1r}\ol{Q_{2s}}-Q_{rs}\|\leq 10\e$, then by taking the
average over all $r$ and $s$ and applying the triangle inequality,
we find that $\|Q_1'\ol{Q_2'}-Q'\|\leq 10\e$, where $Q'$ is the average
of the $Q_{rs}$. Thus, it is enough to prove the result when $Q_1'$ 
and $Q_2'$ are non-generalized quadratic averages.

Let us do this and let $Q_1'$ and $Q_2'$ be given by
the formulae $Q_1'(x)=\E_{y\in x-B}\omega^{q_{1,y}(x)}$
and $Q_2'(x)=\E_{y\in x-B}\omega^{q_{2,y}(x)}$, respectively. \

Now let us imitate the previous proof. Let $B^-$ be the interior 
associated with the pair $B' \prec_{\eps}B$. Then, as we did for $Q$ in the previous
lemma, we can uniformly approximate $Q_1'$ and $Q_2'$ to within $\e$ by
quadratic averages $R_1$ and $R_2$ that are given by the formulae
$R_1(x)=\E_{y\in x-B^-}\omega^{q_{1,y}(x)}$ and $R_2(x)=\E_{y\in
x-B^-}\omega^{q_{2,y}(x)}$, respectively.

Let us examine the product $R_1(x)\ol{R_2(x)}$. It equals
$\E_{y,z\in x-B^-}\omega^{q_{1,y}(x)-q_{2,z}(x)}$, which, by Lemma
\ref{bohraveraging} (v), differs by at most $8\e$ from
\[\E_{y,z\in x-B^-}\E_{u\in -B'}\omega^{q_{1,y+u}(x)-q_{2,z+u}(x)}
=\E_{y,z\in -B^-}\E_{u\in x-B'}\omega^{q_{1,y+u}(x)-q_{2,z+u}(x)}.\]
But the right-hand side is the formula for a generalized quadratic 
average with base $(B',q_1-q_2)$, so we are done.
\end{proof}

\section{The rank of a quadratic average}\label{ranksection}

Recall that so far we have shown how to decompose a function defined
on $\Z_N$ into a linear combination of quadratic averages and an error
that is small in a useful sense. As we did in \cite{Gowers:2009lfuI} in 
the proof for $\F_p^n$, we shall now collect these quadratic
averages into well-correlating clusters. However, before we do so, we 
must think about another concept that is very convenient when discussing
quadratic forms on $\F_p^n$ and that does not have an immediately
obvious $\Z_N$ analogue, namely the rank of a form.

Suppose that we have a quadratic form $q$ defined on a subspace $V$ 
of $\F_p^n$. Then a simple calculation shows that $|\E_{x\in V}\omega^{q(x)}|=p^{-r/2}$,
where $r$ is the rank of the bilinear form $\b(u,v)=(q(u+v)-q(u)-q(v))/2$ associated 
with $q$. Indeed, 
\begin{equation*}
|\E_{x\in V}\omega^{q(x)}|^2=\E_{x,y\in V}\omega^{q(x)-q(y)}=\E_{x,y\in V}\omega^{\b(x+y,x-y)}
=\E_{u,v\in V}\omega^{\b(u,v)}.
\end{equation*}
For each $u$, the expectation over $v$ is 0 unless $\b(u,v)=0$ for every $v$,
in which case it is 1. But the set of $u$ such that $\b(u,v)$ vanishes is a 
subspace of $V$ of codimension $r$, so it has density $p^{-r}$, which 
proves the result.

This calculation allowed us to argue as follows in \cite{Gowers:2009lfuI}. Given a 
quadratic form $q$, we looked at its rank $r$. If $r$ was large, then $q$ had a 
small average, whereas if $r$ was small then $q$ was constant on translates of 
a subspace of low codimension. This dichotomy played an important role in the proof,
so we need to find an analogue for $\Z_N$.

A close examination of the proof in \cite{Gowers:2009lfuI} shows that
the main properties of rank that we used were that the rank of a sum
of two quadratic forms is at most the sum of the ranks, and that if
$\b$ is a bilinear form of rank $r$ and $\phi$ and $\psi$ are linear
functions, then $\E_{x,y}\omega^{\b(x,y)+\phi(x)+\psi(y)}$ has size 
at most $p^{-r}$.

The first of these properties looks very much like a fact of linear algebra,
so it is tempting to try to develop an analogue of this linear algebra for
Bohr sets. Unfortunately, although some analogues of linear
algebra do exist, they are much less clean, and in any case 
they are completely inappropriate for our purposes, roughly speaking
because the codimension of a subspace of $\F_p^n$ corresponds more
to the dimension of a Bohr set in $\Z_N$. For example, if we are
guided by linear algebra then we will be inclined to say that the
function $\omega^{x^2}$ has rank 1, but in fact we want to
count it as having very high rank because the expectation
$\E_{x,y}\omega^{2xy}$ is tiny. 

There is, however, a rather easy way to define an appropriate notion
of rank in the $\Z_N$ context, which is to exploit the fact that there is
a completely different alternative definition in $\F_p^n$. We observed above that
$|\E_x\omega^{q(x)}|$ is not just at most $p^{-r/2}$ but actually \textit{equal}
to $p^{-r/2}$. Therefore, we could, if we wanted, define the rank of $q$
to be $\log_p(\a^{-1})$, where $\a=|\E_x\omega^{q(x)}|^2$. And this
gives us a definition that can be carried over much more easily to
functions defined on Bohr sets in $\Z_N$. (It can also be carried
over much more easily to polynomial forms of higher degree and 
their associated multilinear forms. This was essential to us in 
\cite{Gowers:2009lfuII}.)

We do of course pay a price for such a move. If we define rank in this 
way then it becomes true by definition that averages over quadratic phase 
functions of high rank are small. But we clearly cannot avoid doing any work: 
it is now not obvious that rank is subadditive or that
a quadratic function of low rank has linear structure. In fact, neither
of these statements is exactly true, but with some effort we will be
able to prove usable approximations to them.

One final remark is that it turns out to be more convenient to focus
on bilinear forms rather than quadratic forms. On a subspace of 
$\F_p^n$ the two are basically equivalent, but on a Bohr set
$B$ in $\Z_N$ they no longer are, because $q(a+b)-q(a)-q(b)$ is not 
defined for every $a,b\in B$. We therefore have to look at a smaller
structured set inside $B$.

Here, then, is the definition that we shall use. Some of the features
of the definition may look a bit strange, but they are chosen to make
later proofs run more smoothly.

\begin{definition}
Let $B$ be a Bohr set and let $q$ be a quadratic form on
$B$. Let $B'$ be a Bohr set such that $2B'-2B'\subset B$
and let $P$ be a subset of $B'$. 
The \emph{rank} of the local quadratic phase function
$h(x)=B(x)\omega^{q(x)}$ \emph{relative to} $P$ is $\log(1/\alpha)$,
where $\alpha$ is the quantity
\[
|\E_{a,a',b,b'\in P}\omega^{q(a+b-a'-b')-q(a-a')-q(b-b')}|
\]
If $Q$ is a generalized quadratic average with base $(B,q)$, then we define
the \emph{rank} of $Q$ \emph{relative to} $P$ to be the rank of 
the local quadratic phase function $B(x)\omega^{q(x)}$ relative
to $P$.
\end{definition}

Note that if $q$ and $q'$ are two different quadratic functions
defined on $B$, and if $q-q'$ is a Freiman homomorphism, then
any quadratic average with base $(B,q)$ is also a quadratic
average with base $(B,q')$. Therefore, one must check that the
second definition above is well-defined. But this is easy, since
if $q-q'$ is the Freiman homomorphism $\gamma$, then 
\[q(a+b-a'-b')-q(a-a')-q(b-b')=q'(a+b-a'-b')-q'(a-a')-q'(b-b')
+\gamma(0).\]
It follows that the expectation that defines the rank is
unchanged in modulus.

The next lemma tells us that the expectation of a generalized
quadratic average with high rank is small.


\begin{lemma} \label{highrankQsum}
Let $0<\e<1/20$, let $B$ and $B'$ be Bohr sets satisfying $B'\prec_{\eps}B$ and let $P\subset B'$. 
Let $q$ be a quadratic form on $B$ and let $Q$ be a generalized quadratic average
with base $(B,q)$.  Suppose that the rank of $Q$ relative to $P$ is
$r$. Then $|\E_xQ(x)|\leq(11\e+e^{-r})^{1/4}$.
\end{lemma}

\begin{proof}
Let $h$ be a local quadratic phase function defined by a formula of
the form $h(x)=B(x-y)\omega^{q(x-y)+\phi(x)}$ for some Freiman
homomorphism $\phi$ (so in particular $h$ is supported in the
Bohr neighbourhood $B+y$). Let us estimate $|\E_{x\in
B+y}h(x)|$, using Lemma \ref{bohraveraging} and the Cauchy-Schwarz
inequality. Since $\phi$ is an arbitrary Freiman homomorphism, it is
enough to do this when $y=0$, so let us assume that that is the
case. Recall that ``$\approx_\e$'' stands for the relation ``differs 
by at most $\e$ from''.

By Lemma \ref{bohraveraging} (i) applied twice, we know that 
$\E_{x\in B}h(x)\approx_{2\e}\E_{x\in B}\E_{a,b\in P}h(x+a+b)$. 
It is not hard to check that if $\e<1/20$ and $\a$ and $\b$ are complex
numbers such that $|\alpha|\leq 1$ and $\a\approx_{2\e}\b$,
then $\a^4\approx_{10\e}\b^4$. Therefore, since $\|h\|_\infty\leq 1$,
\begin{align*}
|\E_{x\in B}h(x)|^4&\approx_{10\e}|\E_{x\in B}\E_{a,b\in P}h(x+a+b)|^4\\
&\leq\E_{x\in B}|\E_{a,b\in P}h(x+a+b)|^4.\\
\end{align*}
Now let us look at the inner expectation when $x\in B^{--}$.
We have
\begin{align*}
|\E_{a,b\in P}h(x+a+b)|^4&\leq(\E_{a\in P}|\E_{b\in P}h(x+a+b)|^2)^2\\
&=(\E_{b,b'\in P}\E_{a\in P}h(x+a+b)\ol{h(x+a+b')})^2\\
&\leq\E_{b,b'\in P}|\E_{a\in P}h(x+a+b)\ol{h(x+a+b')}|^2\\
&=\E_{a,a'\in P}\E_{b,b'\in P}h(x+a+b)\ol{h(x+a+b')}\ol{h(x+a'+b)}
h(x+a'+b')\\
&=\E_{a,a'\in P}\E_{b,b'\in P}\omega^{q(a+b-a'-b')-q(a-a')-q(b-b')},
\end{align*}
which equals $e^{-r}$ by definition of the rank relative to
$P$. The proportion of $x$ that belong to $B\setminus
B^{--}$ is at most $\e$, by regularity, and for these the
inner expectation is at most 1. This proves that $|\E_{x\in
B}h(x)|^4\leq 11\e+e^{-r}$, and hence that $|\E_{x\in B}h(x)|\leq(11\e+e^{-r})^{1/4}$.

Now let $Q$ be a quadratic average given by a formula of the kind
$Q(x)=\E_{y\in x-B}\omega^{q_y(x)}$. Then 
$\E_xQ(x)=\E_y\E_{x\in B+y}\omega^{q_y(x)}$. By the estimate just
established, this has absolute value at most $(11\e+e^{-r})^{1/4}$.
Finally, this implies the same upper bound when $Q$ is a generalized 
quadratic average.
\end{proof}

We remark that for the lemma just proved to be useful, one needs
$\e$ to be comparable to or smaller than $e^{-r}$. This may seem to
be quite a strong requirement, given that we also need $B'\prec_\e B$.
However, a recurring theme in this paper is that one can afford to take
Bohr sets of small width: it is the dimension that one has to be careful
about. So in fact the bound above is not too expensive for our later
arguments to work.

We shall now prove a more general result. The proof we give is in two
senses not optimal. The first is that we obtain a bound that is weaker
than it needs to be, because we estimate an $\ell_4$ norm in terms of
an $\ell_\infty$ norm. The second, more serious, is that we use
Fourier analysis. The reason this is a defect is that it obscures the
fact that the proof can be carried out in physical space and is
therefore not hard to generalize. However, since in this paper we
shall not be dealing with the cubic case for functions defined on
$\Z_N$, this is not enough of a defect to outweigh the advantage that
the proof we give is very simple and does not involve technicalities
concerning regular Bohr sets.


\begin{lemma} \label{highrankQu2}
Let $B$ and $B'$ be Bohr sets satisfying $B'\prec_{\eps}B$ and let $P$ be
a subset of $B'$. Let $q$ be a
quadratic form on $B$ and let $Q$ be a generalized quadratic average
with base $(B,q)$.  Suppose that the rank of $Q$ relative to $P$ is
$r$. Then $\|Q\|_{U^2}\leq(11\e+e^{-r})^{1/8}$.
\end{lemma}

\begin{proof}
For every $u$ the function $Q(x)\omega^{ux}$ satisfies all the 
hypotheses of Lemma \ref{highrankQsum}. Therefore, by that lemma, 
$|\hQ(u)|\leq(11\e+e^{-r})^{1/4}$ for every $u$. Since 
$\|\hQ\|_2^2=\|Q\|_2^2\leq 1$, it follows that 
$\|\hQ\|_4^4\leq(11\e+e^{-r})^{1/2}$ and hence that
$\|Q\|_{U^2}\leq(11\e+e^{-r})^{1/8}$.
\end{proof}

The next result expresses the idea that if
two quadratic averages $Q$ and $Q'$ correlate well then they have a ``low-rank
difference'' in the exponent. Very roughly speaking, this is because $Q\ol{Q'}$ has
large average, and is therefore a low-rank quadratic. Of course, this is not quite
the correct argument, because $Q$ and $Q'$ are averages of quadratic phase
functions defined on several different Bohr neighbourhoods. However, the basic
idea is sound, as the next result shows.


\begin{corollary} \label{znquadavecorr}
Let $B$ and $B'$ be two arbitrary Bohr sets, let $q$ and $q'$ be quadratic forms on $B$ and $B'$, and let $Q$ and $Q'$ be generalized quadratic averages with bases $(B,q)$ and $(B',q')$. Let $B_{1}$, $B_{2}$ and $B_{3}$ be Bohr sets satisfying the chain of relations $B_{3}\prec_{\eps}B_{2}\prec_{\eps}B_{1}\prec_{\eps}B \cap B'$. Let $P$ be a subset
of $B_3$. Suppose that the rank of the function 
$B_2(x)\omega^{q(x)-q'(x)}$ relative to $P$ is at least $r$. Then
$|\sp{Q,Q'}|\leq 18\e+(11\e+e^{-r})^{1/4}$.
\end{corollary}

\begin{proof}
By Lemma \ref{quadaveproduct} there is a generalized quadratic
average $Q''$ with base $(B_2,q-q')$ such that 
$\|Q\ol{Q'}-Q''\|_\infty\leq 18\e$. By Lemma \ref{highrankQsum}
and our hypothesis, $|\E_xQ''(x)|\leq(11\e+e^{-r})^{1/4}$. Since
$\sp{Q,Q'}=\E_x(Q\ol{Q'})(x)$, the result follows.
\end{proof}

\section{The structure of low-rank bilinear forms on Bohr sets}\label{lowrankbilinear}

Our next task is to understand the implications if the hypotheses
of Corollary \ref{znquadavecorr} do not hold. In the $\F_p^n$
case, we argued that if $Q$ and $Q'$ are quadratic averages
and $\sp{Q,Q'}$ is not small, then $Q\ol{Q'}$ has low rank,
from which it follows that $Q\ol{Q'}$ is constant on 
cosets of a low-codimensional subspace. From this we 
deduced that $\|Q\ol{Q'}\|_{U^2}^*$ is not too large. In 
this paper, where $Q$ and $Q'$ are defined on a Bohr set $B$,
we shall argue that $Q\ol{Q'}$ is approximately
constant on translates of a small (but not too small) multidimensional
arithmetic progression $P\subset B$, and deduce that $Q\ol{Q'}$ can 
be uniformly approximated by a function with smallish $(U^2)^*$ norm.

A similar result to this was proved by Green and Tao in
\cite{Green:2008py} using a ``local Bogolyubov lemma"
that they developed specially for the purpose. Their 
argument can be used to show that $Q\ol{Q'}$ is approximately constant
on a Bohr subset $B'$ of $B$. However, the local Bogolyubov
lemma is rather expensive, in that the dimension of $B'$ is
considerably larger than that of $B$. This expense has to be
iterated, and it turns out that if we were to use their
result, then we would end up with a tower-type bound for our final
estimate. By contrast, the progression $P$ that we find has
the same dimension as that of $B$ and the final estimate we
obtain is doubly exponential. Unfortunately, our argument 
is rather uglier than that of Green and Tao since we 
rely on the fact that bilinear forms on multidimensional
arithmetic progressions can be explicitly described, rather than
just using the defining properties of bilinear forms on Bohr sets.

Although we eventually need a statement about Bohr sets, we shall begin by 
proving a dichotomy for bilinear phase functions defined on 
multidimensional progressions.
As a prelude, here is a proof for the special case of one-dimensional progressions.
It will be useful for the general case if we prove a non-symmetric result
where one variable belongs to one progression and the other to another 
of a possibly different length.
We shall also allow our bilinear forms to be non-homogeneous.
That is, we shall consider functions of the form $b(x,y)=e(\a xy+\lambda x+\mu y+\nu)$ 
with $1\leq x\leq m_1$ and $1\leq y\leq m_2$. The aim will be to prove
that such a function either has a small average (where this means smaller than a small
positive constant $c$) or is approximately constant on a reasonably large 
subgrid of $[m_1]\times[m_2]$ (where this means a subgrid of size at least $c'm_1m_2$ for
some not too small positive constant $c'$, but $c'$ is allowed to be smaller than 
$c$ and this elbow room will be quite helpful). The precise statement is as follows.
As is standard, if $\theta$ is a real number then we write $\|\theta\|$ for the distance
from $\theta$ to the nearest integer.


\begin{lemma}\label{1dcase}
Let $c>0$ and let $m_1$ and $m_2$ be positive integers. Let $\b(x,y)$ be the 
bilinear phase function $e(\a xy+\lambda x+\mu y+\nu)$, defined
when $0\leq x<m_1$ and $0\leq y<m_2$. Suppose that $|\E_{x,y}b(x,y)|\geq 2c$.
Then there exists a positive integer $q\leq 2c^{-1}$ and an integer $p$ such that 
$|\a-p/q|\leq 2c^{-2}/m_1m_2$. In particular, $\|\a xy\|\leq 2c$
whenever $x$ and $y$ are both multiples of $q$ and $x/m_1$ and $y/m_2$
are both at most $c^{3/2}$.
\end{lemma}

\begin{proof}
Observe first that $\E_{x,y}b(x,y)\leq\E_y|\E_xb(x,y)|
=\E_y|\E_xe(\a xy+\lambda x+\mu y+\nu)|$. (Here, as in the statement of the lemma,
the expectations are over all $x$ and $y$ with $0\leq x<m_1$ and $0\leq y<m_2$.) 
Now for each $y$, the quantity
$|\E_xe(\a yx+\lambda x+\mu y+\nu)|$ is at most $\min\{1,1/m_1\|\a y+\lambda\|\}$, by
the formula for summing a geometric progression. In particular, it is at most $c$ unless
$\|\a y+\lambda\|\leq C/m_1$, where $C=1/c$. So the only way that $\E_y|\E_xb(x,y)|$ can be
at least $2c$ is if $\|\a y+\lambda\|\leq C/m_1$ for at least $cm_2$ values of $y$ (since otherwise
we get less than $c+c$).

So now let us think about what is implied if $\|\a y+\lambda\|\leq C/m_1$ for at least $cm_2$ values
of $y$. We shall show that $\a$ is within $C'/m_1m_2$ of a rational with small denominator,
where with the benefit of hindsight we choose $C'$ to equal $2C/c=2c^{-2}$.
It is an easy and standard consequence of the pigeonhole principle that there is a rational 
$p/q$ with $q\leq m_1m_2/C'$ such that $|\a-p/q|\leq C'/m_1m_2q$. It follows that 
$|\a y-py/q|\leq C'/m_1q$ for every $y\leq m_2$. Therefore, either $q\leq C'/C$ or 
$\|\lambda+py/q\|\leq 2C/m_1$ whenever $\|\a y+\lambda\|\leq C/m_1$.

So now let us bound the number of multiples $py/q$ of $p/q$ such that $\lambda+py/q$ can be 
within $2C/m_1$ of an integer, given that $p$ and $q$ are coprime. To do this we split into 
cases. If $q<m_1/4C$, then $1/q >4C/m_1$, so two translates of multiples of $1/q$ that are 
distinct mod 1 cannot both be within $2C/m_1$ of an integer. But since $p$ and $q$ are 
coprime, any $q$ distinct multiples of $p/q$ are also distinct multiples of $1/q$ mod 1, so
at most one of them is within $2C/m_1$ of an integer when you add $\lambda$ to it. So if 
$q<m_1/4C$, then $\|\a y\|\leq C/m_1$
for at most $q^{-1}m_2+1$ values of $y$. (The ``$+1$" is there because $m_2$ doesn't
have to be a multiple of $q$.) If this is at least $cm_2$, then $q$ is certainly at most $2c^{-1}$.

If $q\geq m_1/4C$ then we argue differently. This time we argue that the number of 
multiples of $1/q$ that are distinct mod 1 and lie within $2C/m_1$ of an integer is
at most $2Cq/m_1$. Since $q\leq m_1m_2/C'$, this is at most $2Cm_2/C'=cm_2$.
Therefore, the number of $y$ such that $\|\a y\|\leq C/m_1$ is also at most $cm_2$. 

In conclusion, either $q\leq 2c^{-1}$ or 
$\E_y|\E_xb(x,y)|=\E_y|\E_xe(\a xy+\lambda x+\mu y+\nu)|\leq 2c$.
In the first case, we have $|\a-p/q|\leq C'/m_1m_2q=2c^{-2}/m_1m_2q$, which implies
the first assertion. If $x$ and $y$ are multiples
of $q$ then $\|\a xy\|\leq 2c^{-2}xy/m_1m_2$. Therefore, if in addition $xy\leq c^3m_1m_2$,
then we have that $\|\a xy\|\leq 2c$. In particular, $\|\a xy\|\leq 2c$
if $x$ and $y$ are both multiples of $q$ and $x\leq c^{3/2}m_1$ and $y\leq c^{3/2}m_2$.
\end{proof}

Let us now see how Lemma \ref{1dcase} generalizes to a similar statement for 
bilinear phase functions on $d$-dimensional arithmetic progressions. This turns
out to follow fairly straightforwardly from the one-dimensional case.

So now $x=(x_1,\dots,x_d)$ and $y=(y_1,\dots,y_d)$ range over $d$-dimensional
arithmetic progressions, and we are looking at a function of the form
$b(x,y)=e(\sum_{i,j}\a_{ij}x_iy_j)$. We would like to show that either the average of
$b(x,y)$ is small or every $\a_{ij}$ is extremely close to a rational with small
denominator. In the latter case, we will be able to restrict to a subprogression 
of the same dimension that is not too much smaller on which $b$ is approximately
constant.


\begin{corollary} \label{ddcase}
Let $c>0$, let $m_1,\dots,m_d$ be positive integers and let $P$ be the 
multidimensional progression $\prod_{i=1}^d\{0,1,\dots,m_i-1\}$.
Let $b$ be a bilinear phase function on $P$ given by the formula 
$b(x,y)=e(\sum_{i,j}\a_{ij}x_iy_j+\sum_i\lambda_ix_i+\sum_j\mu_jy_j+\nu)$. Then either $|\E_{x,y}b(x,y)|\leq 2c$ or there exist positive integers $q_{rs}\leq 2c^{-1}$ and integers $p_{rs}$
such that $|\alpha_{rs}-p_{rs}/q_{rs}|\leq 2c^{-2}/m_rm_s$ for every $r,s\leq d$.
In the second case, there are positive integers $q_1,\dots,q_d\leq (2c^{-1})^{2d}$ 
such that $\|\sum_{rs}\a_{rs}x_ry_s\|\leq 2d^2c$ whenever $x$ and $y$ belong
to the subprogression $P'$ that consists of all $z\in P$ such that each 
$z_r$ is of the form $h_rq_r$ for some $h_r$ between $0$ and $c^{3/2}m_r$.
\end{corollary}

\begin{proof}
Let us fix all coordinates of $x$ and $y$ apart from $x_r$ and $y_s$ and
estimate the quantity $|\E_{x_r}\E_{y_s}b(x,y)|$. We can write this
expression in the form $|\E_{x_r,y_s}e(\a_{rs}x_ry_s+\lambda x_r+\mu y_s+\nu)|$, 
where $\lambda$, $\mu$ and $\nu$ depend on the other coordinates of $x$ 
and $y$. Therefore, by Lemma \ref{1dcase},  either 
$|\E_{x_r,y_s}b(x,y)|\leq 2c$ or there exists a positive 
integer $q_{rs}\leq 2c^{-1}$ and an integer $p_{rs}$ such that 
$|\alpha_{rs}-p_{rs}/q_{rs}|\leq 2c/m_rm_s$.

In the second case, $\|\a_{rs}x_ry_s\|\leq 2c$ whenever $x_r$
and $y_s$ are both multiples of $q_{rs}$ and $x_r/m_r$ and $y_s/m_s$ are
both at most $c^{3/2}$. A quick examination of the proof of Lemma \ref{1dcase}
shows that the choice of $q$ did not depend on $\lambda$, $\mu$ or $\nu$,
but only on the rational approximations to $\a$. Therefore, by averaging
over all possible values of the other coordinates of $x$ and $y$ we may
conclude that either $|\E_{x,y}b(x,y)|\leq 2c$ or there exists a positive 
integer $q_{rs}\leq 2c^{-1}$ and an integer $p_{rs}$ such that 
$|\alpha_{rs}-p_{rs}/q_{rs}|\leq 2c/m_rm_s$. (This is the same conclusion
as that of the previous paragraph, but the assumption is different, since now
we are averaging over all $x$ and $y$ rather than fixing all but one coordinate.) 
In the second case, $\|\a_{rs}x_ry_s\|\leq 2c$ whenever $x_r$ and $y_s$ are 
both multiples of $q_{rs}$ and $x_r/m_r$ and $y_s/m_s$ are
both at most $c^{3/2}$.

Since this is true for every $r$ and $s$, either $|\E_{x,y}b(x,y)|\leq 2c$ 
or there are $d^2$ positive integers $q_{rs}\leq 2c^{-1}$ such that $\|\a_{rs} x_ry_s\|\leq 2c$ 
whenever $x_r$ and $y_s$ are both multiples of $q_{rs}$ and $x_r/m_r$ and $y_s/m_s$
are both at most $c^{3/2}$. For each $r$, let $q_r$ be the product of all the $q_{rs}$
and all the $q_{sr}$. Then $q_r$ is at most $(2c^{-1})^{2d}$, and if $x_r$ is a multiple
of $q_r$ and $y_s$ is a multiple of $q_s$ with $x_r/m_r$ and $y_s/m_s$ both at most 
$c^{3/2}$, then again $\|\a_{rs}x_ry_s\|\leq 2c$. But if that is true for every $r$ and 
every $s$, then $\|\sum_{r,s}\a_{rs}x_ry_s\|\leq 2d^2c$, which proves the result.
\end{proof}

Our next target is to prove that quadratic averages either have small $U^2$ norms
or are uniformly close to functions with moderately small $U^2$-dual norms. We 
begin with a lemma about linear phase functions on subsets of $\Z_N$. Before 
stating it, let us give a definition that generalizes our earlier concepts of interior,
closure and boundary to arbitrary pairs of sets.

\begin{definition}
Given a pair $(A,B)$ of subets of $\Z_N$, define the \textit{closure} of
$A$ (relative to $B$) to be $A+B$ and the \textit{interior} to be $\{x:x+B\subset A\}$.
Denote these by $A^+$ and $A^-$, respectively.
Define the \textit{boundary} of $A$ to be $A^+\setminus A^-$ and denote it 
by $\partial A$. 
\end{definition}

As before, when we use the notation $A^+$, $A^-$ and $\partial A$, it will always
be clear from the contexts what the set $B$ is that we are implicitly talking about.


\begin{lemma}\label{convolutionapprox}
Let $A$ be a subset of $\Z_N$, let $\phi:A\ra\Z_N$ be a Freiman homomorphism,
let $B$ be a subset of $A-A$ that contains 0, and let $\psi$ be the function 
$\psi(d)=\phi(x+d)-\phi(x)$ for some $x\in A\cap(A-d)$, which is well-defined 
everywhere on $B$. Let the densities of $A$ and $B$ be $\g$ and
$\theta$. Let $f$ be the function defined by taking $f(x)=\g^{-1}\omega^{\phi(x)}$ 
when $x\in A$ and $0$ otherwise, and let $g$ be defined by taking 
$g(d)=\theta^{-1}\omega^{\psi(d)}$ whenever $d\in B$, and 0 otherwise. Then
$\|f-f*g\|_\infty\leq 2\g^{-1}$, and $f-f*g$ is supported inside the boundary $\partial A$.
\end{lemma}

\begin{proof}
First let us deal with the uniform bound for $f-f*g$. Since $\|f\|_\infty\leq\g^{-1}$,
it is enough to prove that $\|f*g\|_\infty\leq\g^{-1}$. But this is clear because
$f*g(x)=\E_{d\in B}f(x-d)\omega^{\psi(d)}$, which is an average of numbers 
with absolute value at most $\g^{-1}$. (This equality is the reason for 
normalizing $g$ with the constant $\theta^{-1}$.)

If $x\notin A^+=A+B$, then $f(x-d)=0$ for every $d\in B$, so $f*g(x)=0$. Since $0\in B$
and $f$ is supported in $A$, $f(x)=0$ as well.

If $x\in A^-$, then 
\begin{equation*}
f*g(x)=\E_{d\in B}f(x-d)\omega^{\psi(d)}=\E_{d\in B}\omega^{\phi(x-d)+\psi(d)}
=\E_{d\in B}\omega^{\phi(x)}=f(x).
\end{equation*}
This proves the lemma.
\end{proof}

We would like to think of $f*g$ as approximating $f$, so we shall apply Lemma
\ref{convolutionapprox} to a pair of sets $A$ and $B$ such that $\partial A$
is small. We have already seen such pairs in the context of regular Bohr neighbourhoods,
but we now need to look at multidimensional arithmetic progressions as well.


\begin{lemma}\label{progressioncentre}
Let $P$ be a proper $d$-dimensional arithmetic progression consisting of all points
$x_0+\sum_{i=1}^da_ix_i$ such that $0\leq a_i<m_i$, let $\e>0$, and let $Q$ be
the progression consisting of all points $\sum_{i=1}^db_ix_i$ such that $0\leq b_i<\e m_i/d$.
Let the density of $P$ be $\g$. Then the density of $P^+\setminus P^-$ is at most
$3\e\g$ and the density of $Q$ is at least $(\e/d)^d\g$. 
\end{lemma}

\begin{proof}
The number of integers of the form $r+s$, where $r$ is an integer between $0$ and $m-1$
and $s$ is an integer such that $0\leq s\leq\eta m$ is at most $(1+\eta)m$, since we have
equality when $\eta m$ is an integer, and if we increase $\eta m$ towards the next
integer then we increase $(1+\eta)m$ without increasing the number of elements of 
the set.

Now suppose that $r$ is an integer and that $r\geq\lfloor\eta m\rfloor$. Then $r-s\geq 0$ 
whenever $s$ is an integer and $s<\eta m$. The number of integers less than $m$ with
this property is $m-\lfloor\eta m\rfloor\geq m(1-\eta)$.

From these two calculations, we find that $P^+$ has density at most $(1+\e/d)^d\g$
and $P^-$ has density at least $(1-\e/d)^d\g$. The first result now follows from the simple
estimates $(1+\e/d)^d\leq 1+2\e$ and $(1-\e/d)^d\geq 1-\e$.

Also, the number of integers $r$ such that $0\leq r<\eta m$ is $\lceil\eta m\rceil\geq\eta m$,
so the density of $Q$ is at least $(\e/d)^d$ times that of $P$, so we have the second
assertion as well.
\end{proof}

Next, we show why approximating a function by a convolution of two 
functions helps us to control its $U^2$-dual norm.


\begin{lemma}\label{u2*ofconvolution}
Let $A$ and $B$ be two sets and let $f$ and $g$ be two functions such that
$|f|$ is bounded above by the characteristic measure of $A$ and $g$ is bounded
above by the characteristic measure of $B$. Suppose that the density of $A$ is
$\g$ and the density of $B$ is $\theta$. Then $\|f*g\|_{U^2}^*\leq\g^{-1/2}\theta^{-1/4}$.
\end{lemma}

\begin{proof}
Let $h$ be any other function, and define $g^*$ by $g^*(x)=\ol{g(-x)}$ for
every $x$. Then 
\begin{equation*}
\sp{f*g,h}=\sp{f,g^**h}\leq\|f\|_2\|g^**h\|_2\leq\g^{-1/2}\|g\|_{U^2}\|h\|_{U^2}
\leq\g^{-1/2}\theta^{-1/4}\|h\|_{U^2},
\end{equation*}
from which the result follows. Here we have used the fact that the characteristic
measure of a set of density $\d$ has $L_2$ norm at most $\d^{-1/2}$ and 
$U^2$ norm at most $\d^{-1/4}$. We have also made use of the inequality
$\|u*v\|_2\leq\|u\|_{U^2}\|v\|_{U^2}$, which can be thought of as a special 
case of Young's inequality or as a special case of Lemma 3.8 of 
\cite{Gowers:2001tg}, a Cauchy-Schwarz inequality for the uniformity norms.
\end{proof}

Putting the last three lemmas together, we deduce the following.


\begin{lemma}\label{linearphaseu2*}
Let $P$, $\g$, $\e$ and $Q$ be as in Lemma  \ref{progressioncentre}. Let $\phi$ 
be a Freiman homomorphism defined on $P$, and let $f(x)=\g^{-1}\omega^{\phi(x)}$ 
if $x\in P$ and 0 otherwise. Then there exists a function $h$ such that 
$\|f-h\|_\infty\leq 2\g^{-1}$, $\|h\|_{U^2}^*\leq\g^{-3/4}(\e/d)^{-d/4}$, and $f-h$
is supported in $P^+\setminus P^-$, which has density at most $3\e\g$.
\end{lemma}

\begin{proof}
Let us apply Lemma \ref{u2*ofconvolution} with $A=P$, $B=Q$, $f$ as given
in this lemma, and $g$ as defined in the statement of Lemma \ref{convolutionapprox}.
We shall prove that we can take $h$ to be the function $f*g$.
Lemma \ref{convolutionapprox} tells us that $\|f-f*g\|_\infty\leq 2\g^{-1}$ and
that $f-f*g$ is supported in $P^+\setminus P^-$. Lemma \ref{progressioncentre}
tells us that $P^+\setminus P^-$ has density at most $3\e\g$, and lemma
\ref{u2*ofconvolution} tells us that $\|f*g\|_{U^2}\leq\g^{-1/2}\theta^{-1/4}$,
where $\theta$ is the density of $Q$. Lemma \ref{progressioncentre} tells
us that $\theta$ is at least $(\e/d)^d\g$, and this completes the proof.
\end{proof}

We are about to prove a slightly complicated technical lemma that will help 
us handle error terms without cluttering up proofs. Before we do so, here 
is a much simpler technical lemma that will help us to prove the complicated
one without cluttering up \textit{its} proof.


\begin{lemma} \label{uniformerrors}
Let $\alpha, \beta, \rho$ and $\sigma$ be positive constants. Let $U$ and $V$ be 
subsets of $\Z_N$ of density $\sigma\a$ and $\b$, respectively. For each
$y\in V$ let $g_y$ be a function supported in $y+U$ such that 
$\|g_y\|_\infty\leq\rho\a^{-1}$. Then $\|\E_{y\in V}g_y\|_\infty\leq\rho\sigma\b^{-1}$.
\end{lemma}

\begin{proof}
For each $x$, 
\begin{equation*}
|\E_{y\in V}g_y(x)|\leq\rho\a^{-1}\P[x\in y+U|y\in V]\leq\rho\a^{-1}\sigma\a\b^{-1}=\rho\sigma\b^{-1}.
\end{equation*}
The lemma follows.
\end{proof}


\begin{lemma}\label{boundarylemma2}
Let $(A,B)$ and $(C,D)$ be two pairs of subsets of $\Z_N$ with $C^+\subset B$. Let 
the densities of $A$ and $C$ be $\b$ and $\g$, respectively. 
Suppose also that $\partial C$ has density
at most $\e\g$. Let $g$ be a function defined on $\Z_N$ such that $|g(x)|\leq\b^{-1}$
for every $x\in A$ and $g(x)=0$ for every $x\notin A$. For each $y\in A$ let $g_y$ 
be a function such that $\|g_y\|_\infty\leq\g^{-1}$ and $g_y$ is supported in $y+C$.
Suppose that $\E_{y\in A}g_y(x)=g(x)$ for every $x\in A^-$. Now suppose that for each
$y\in A^-$ there is a function $h_y$ such that $|g_y(x)-h_y(x)|\leq\theta\g^{-1}$ for every
$x\in y+C^-$, $|g_y(x)-h_y(x)|\leq \lambda\g^{-1}$ for every $x\in y+\partial C$, and 
$g_y(x)=h_y(x)=0$ whenever $x\notin y+C^+$. And for each $y\in A\setminus A^-$, 
let $h_y$ be identically zero. Let $h(x)=\E_{y\in A}h_y(x)$
for every $x\in\Z_N$. Then $|g(x)-h(x)|\leq(\theta+\lambda\e)\b^{-1}$ for every $x\in A^-$, 
$|g(x)-h(x)|\leq (4+\lambda\e)\b^{-1}$ for every $x\in\partial A$, and $g(x)=h(x)=0$ for
every $x\notin A^+$.
\end{lemma}

\begin{proof}
If $x\in A^-$ then $\E_{y\in A}g_y(x)=g(x)$, by hypothesis. If $x\in\partial A$, then Lemma \ref{uniformerrors} (with $U=C$ and $V=A$) implies that $|E_{y\in A}g_y(x)|\leq\b^{-1}$, 
which implies that $|g(x)-\E_{y\in A}g_y(x)|\leq 2\b^{-1}$. And if $x\notin A^+$, then both 
$g(x)$ and $\E_{y\in A}g_y(x)$ are zero.

Let us write $u_y$ for the restriction of $g_y-h_y$ to $y+C^-$ and $v_y$
for the restriction of $g_y-h_y$ to $y+\partial C$. Then $g_y-h_y=u_y+v_y$
for every $y\in A$. If $y\in A^-$, then $\|u_y\|_\infty\leq\theta\g^{-1}$ and 
$\|v_y\|_\infty\leq \lambda\g^{-1}$. If $y\in A\setminus A^-$ then 
$\|u_y\|_\infty$ and $\|v_y\|$ are both at most $\g^{-1}$.

For every $x\in A^-$, $|\E_{y\in A}u_y(x)|\leq\theta\b^{-1}$ by Lemma
\ref{uniformerrors} (with $U=C^-$ and $V=A$). For every 
$x\in\partial A$, $|\E_{y\in A}u_y(x)|\leq 2\b^{-1}$, again by Lemma
\ref{uniformerrors}. (In this case, we have the bound $\|u_y\|_\infty\leq 2\g^{-1}$.)
And for every $x\notin A^+$, $\E_{y\in A}u_y(x)=0$.

If $x\in A^+$, then $|\E_{y\in A}v_y(x)|\leq \lambda\e\b^{-1}$, again by Lemma 
\ref{uniformerrors} (this time with $U=\partial C$). And if $x\notin A^+$, 
then $\E_{y\in A}v_y(x)=0$.

Adding these estimates together, we find that $|\E_{y\in A}g_y(x)-\E_{y\in A}h_y(x)|$
is at most $(\theta+\lambda\e)\b^{-1}$ if $x\in A^-$, at most $(2+\lambda\e)\b^{-1}$ if 
$x\in\partial A$, and 0 if $x\notin A^+$. Finally, combining this with the estimates
for $g-\E_{y\in A}g_y$ in the first paragraph, we obtain the result stated.
\end{proof}

In the next statement, it may not be clear why $\eta$ cannot be taken to be arbitrarily
small. The reason is that the maximum possible density $\g$ decreases with $\eta$,
so in fact the bound on $\|Q''\|_{U^2}^*$ increases as $\eta$ decreases. 


\begin{corollary}\label{rankdichotomy}
Let $0<\e\leq 1$ and let $0<\eta\leq 1/20$.
Let $B$ be a regular Bohr set of density $\b$ and let $B'$ be a Bohr subset with 
$B'\prec_\eta B$. Let $P$ be a $d$-dimensional arithmetic progression of density 
$\g$ such that $P+P$ lives inside $B'$, let $q$ be a quadratic form on $B$ and 
let $Q$ be a generalized 
quadratic average with base $(B,q)$. Then for every $\a>0$, either 
$\|Q\|_{U^2}\leq(11\eta+\a)^{1/8}$ or there exists a function $Q''$ such that
$\|Q-Q''\|_\infty\leq 4\pi d^2\a+2\e+7\eta$ and $\|Q''\|_{U^2}^*\leq\g'^{-3/4}(\e/d)^{-d/4}$,
where $\g'\geq (\a/4)^{4d^2}\g$.
\end{corollary}

\begin{proof}
Suppose first that $Q$ has rank at most $\log(1/\a)$ relative to $P$. In this
case, we are immediately done, since Lemma \ref{highrankQu2} tells us that
$\|Q\|_{U^2}\leq(11\eta+\a)^{1/8}$.

Now suppose that $Q$ has rank at least $\log(1/\a)$ relative to $P$. As usual, 
let us begin by assuming that $Q$ is a non-generalized quadratic average, so that 
it has a formula of the form $Q(x)=\E_{y\in x-B}\omega^{q_y(x)}$, where 
$q_y(x)=q(x-y)+\phi_y(x-y)$ for some Freiman homomorphism $\phi_y$
defined on $B$. For each $y$ let us define $f_y(x)$ to be $\b^{-1}\omega^{q_y(x)}$
if $x\in y+B$ and $0$ otherwise. Then, as we commented after defining quadratic
averages, $Q$ is the average of all the functions $f_y$. 
The strategy of our proof will be to show that each function $f_y$ can be approximated
by a function with small $U^2$-dual norm in such a way that the average of all the
errors is uniformly small.

We shall begin by examining $f=f_0$, which is supported in $B$. By the definition of 
rank, we have the inequality
\begin{equation*}
|\E_{a,a',b,b'\in P}\omega^{q(a+b-a'-b')-q(a-a')-q(b-b')}|\geq\a.
\end{equation*}
It follows that there exist $a'$ and $b'$ in $P$ such that 
\begin{equation*}
|\E_{a,b\in P}\omega^{q(a+b-a'-b')-q(a-a')-q(b-b')}|\geq\a.
\end{equation*}
Choose such an $a'$ and $b'$, and write $\b(u,v)$ for
$q(u+v)-q(u)-q(v)$. Then 
\begin{equation*}
q(a+b-a'-b')-q(a-a')-q(b-b')
=\b(a-a',b-b')
\end{equation*}
which we can expand into the homogeneous part $\b(a,b)$ and the linear
and constant terms $-\b(a',b)-\b(a,b')+\b(a',b')$.

Let us discuss further the relationship between $q$ and $\b$. A quadratic
homomorphism on $P$ must be given by a formula of the form
$q(x)=\sum_{i,j}a_{ij}x_ix_j+\sum_ib_ix_i+c$ for a matrix
$(a_{ij})$ that we may take to be symmetric (since we can replace
it by $(a_{ij}+a_{ji})/2$). Then $\b(u,v)$ works out to be $2\sum_{ij}a_{ij}u_iv_j$,
and there are coefficients $b_i'$ and $c_j'$ and $d'$ such that 
$\b(u-a',v-b')=2\sum_{ij}a_{ij}u_iv_j+\sum_ib_i'u_i+\sum_jc_j'v_j+d'$. Moreover, $q(x)=\b(x,x)/2$ for every $x$.

Since $|\E_{a,b\in P}\omega^{\b(a-a',b-b')}|\geq\a$,
Corollary \ref{ddcase} (with $c=\a/2$) implies
that there is a subprogression $P'$ of $P$ of density at 
least $(\a/4)^{2d^2}(\a/2)^{3d/2}\g$ and of dimension $d$ such that 
$|1-\omega^{\b(a,b)}|\leq 2\pi d^2\a$ for every $a$ and $b$ in $P'$. 
If we restrict further, to pairs $(a,b)$ such that all their coordinates are
even, then we obtain a progression $P''$ of density 
$\g'\geq 2^{-d}(\a/4)^{2d^2}(\a/2)^{3d/2}\g\geq(\a/4)^{4d^2}\g$ and of dimension 
$d$ such that $|1-\omega^{\b(a,b)/2}|\leq 2\pi d^2\a$ for every $a$ and 
$b$ in $P''$. Let us set $\theta$ to be $2\pi d^2\a$. Then there is a Freiman 
homomorphism $\phi$ defined on 
$P''$ such that $|\omega^{q(a)}-\omega^{\phi(a)}|\leq\theta$ for
every $a\in P''$. 

We now apply Lemma \ref{linearphaseu2*} to the function 
$l$ defined by $l(x)=\g'^{-1}\omega^{\phi(x)}$ when $x\in P''$
and $l(x)=0$ otherwise. It gives us a subprogression $P_{3}\subset P''$
and a function $h'$ such that $\|l-h'\|_\infty\leq 2\g'^{-1}$, 
$\|h'\|_{U^2}^*\leq\g'^{-3/4}(\e/d)^{-d/4}$, and $l-h'$ is supported
on $\partial P''$, which has density at most $3\e\g'$. (Here the boundary is taken with respect to $P_{3}$, and $\eps$ denotes the proportion of $P''$ that we take to lie in $P_{3}$).

Let us define $f'(x)$ to be $\g'^{-1}\omega^{q(x)}$ if $x\in P''$ and 0 
otherwise. The above calculations show that $|f'(x)-h'(x)|$ is at most
$\theta\g'^{-1}$ when $x\in P''^-$, at most $(2+\theta)\g'^{-1}$
when $x\in\partial P''$, and 0 when $x\notin P''^+$. Moreover, 
the density of $\partial P''$ is at most $3\e\g'$.

We are preparing to apply Lemma \ref{boundarylemma2}. Our pairs
of sets are $(B,B')$ and $(P'',Q)$, which satisfy the hypothesis
since $P''^+=P''+P_{3}\subset P+P\subset B'$. Our function $g$ is
defined by taking $g(x)=\b^{-1}\omega^{q(x)}$ if $x\in B$ and
0 otherwise. If $y\in B^-$ then we shall define 
$g_y(x)$ to be $\g'^{-1}\omega^{q(x)}$ if $x\in y+P''$ and 0
otherwise. (This is the \textit{normalized restriction} of
$\omega^{q(x)}$ to $y+P''$.) If $y\notin B^-$ we shall define 
$g_y$ to be identically zero. Then if $x\in B^-$, we have
$\E_{y\in B}g_y(x)=\g'^{-1}\b g(x)\P[x\in y+P''|y\in B]=g(x)$,
so the hypotheses about $g$ and the $g_y$ are satisfied.

The function $f'$ just discussed was equal to $g_0$. For each
fixed $y\in B^-$ the function $q(x)-q(x-y)$ is a Freiman homomorphism
on $y+P''$, so the argument used for $f_0$ can also be used to
provide for us a function $h_y$ such that 
$\|h_y\|_{U^2}^*\leq\g'^{-3/4}(\e/d)^{-d/4}$, and such that
$|g_y(x)-h_y(x)|$ is at
most $\theta\g'^{-1}$ when $x\in y+P''^-$, at most $(2+\theta)\g'^{-1}$
when $x\in y+\partial P''$, and 0 otherwise. Thus, in Lemma
\ref{boundarylemma2} we can take $\b$ to be $\b$, $\g$ to be 
$\g'$, $\theta$ to be $\theta$, $\e$ to be $\e$, and $\lambda$ 
to be $2+\theta$.

We then set $h(x)=\E_{x\in B}h_y(x)$. By Lemma \ref{boundarylemma2},
$|g(x)-h(x)|$ is at most $(\theta+2\e+\theta\e)\b^{-1}$ for every $x\in B^-$,
at most $(4+2\e+\theta\e)\b^{-1}$ for every $x\in\partial B$, and $g(x)=h(x)=0$
for $x\notin B^+$. Moreover, since $h$ is just the average of all the $h_y$,
the triangle inequality implies that $\|h\|_{U^2}^*\leq\g'^{-3/4}(\e/d)^{-d/4}$.

Now $g$ is the function $f_0$ defined at the beginning of the proof, where
we defined $f_y(x)$ to be $\b^{-1}\omega^{q_y(x)}$ if $x\in y+B$ and 0
otherwise. Since $q_y(x)-q(x-y)$ is a Freiman homomorphism on $y+B$,
the same argument gives us a function $k_y$ such that $|f_y(x)-k_y(x)|$ 
is at most $(\theta+2\e+\theta\e)\b^{-1}$ for every $x\in y+B^-$,
at most $(4+2\e+\theta\e)\b^{-1}$ for every $x\in\partial y+B$, and $f_y(x)=k_y(x)=0$
for $x\notin y+B^+$. Also, $\|k_y\|_{U^2}^*\leq\g'^{-3/4}(\e/d)^{-d/4}$.

We now apply Lemma \ref{boundarylemma2} once again, but this time it is 
simpler because our set $A$ will have empty boundary. Indeed, we take
$A$ and $B$ to be $\Z_N$, $C$ to be $B$, and $D$ to be $B'$. This time
round we can take $\beta$ to be 1, $\g$ to be $\b$,  
$\e$ to be $\eta$, $\theta$ to be $\theta+2\e+\theta\e$, and $\lambda$ 
to be $4+2\e+\theta\e$. Then $Q(x)=\E_{y\in x-B}\omega^{q_y(x)}$ by
definition, and this is equal to $\E_{y\in\Z_N}f_y(x)$. Let 
$Q''(x)=\E_{y\in\Z_N}k_y(x)$. Then Lemma \ref{boundarylemma2} tells
us that $\|Q-Q''\|_\infty\leq\theta+2\e+\theta\e+(4+2\e+\theta\e)\eta$.
Moreover, $\|Q''\|_{U^2}^*\leq \g'^{-3/4}(\e/d)^{-d/4}$, again by the
triangle inequality. 
\end{proof}

We now combine Corollary \ref{rankdichotomy} with a result of Ruzsa so that we can 
say something about bilinear phase functions defined on Bohr sets.


\begin{theorem}\label{u2u2*dichotomy}
Let $Q$ be a generalized
quadratic average of complexity $(d,\rho)$. Then for every $\a$ with $0<\a\leq 1/20$, 
either $\|Q\|_{U^2}\leq(12\a)^{1/8}$ or there exists a function $Q''$ such that
$\|Q-Q''\|_\infty\leq 16d^2\a$ and $\|Q''\|_{U^2}^*\leq(4/\a)^{4d^2}(800d^2/\rho)^d$.
\end{theorem}

\begin{proof}
Suppose that $Q$ has base $(B,q)$, where $B=B(K,\rho)$ and $K$ has 
cardinality $d$. Let $\eta=\a$ and let $B'\prec_\a B$. Then $B'=B(K,\sigma)$
for some $\sigma\geq\a\rho/400d$. A theorem of Ruzsa \cite{Ruzsa:1994gap} (see also \cite{Nathanson:1996fp}) tells us that
$B'$ contains a proper $d$-dimensional arithmetic progression of density
at least $(\sigma/d)^d$. Therefore, there is a proper $d$-dimensional
arithmetic progression $P$ of density $\g\geq(\sigma/2d)^d$ such that
$P+P\subset B'$. By Corollary \ref{rankdichotomy} with $\eta=\a$ and $\e=\a d^2$ (if 
$\e>1$ then Corollary \ref{rankdichotomy} is trivial so we do not need to worry about
this), either $\|Q\|_{U^2}\leq(12\a)^{1/8}$ or there exists a function $Q''$ 
such that $\|Q-Q''\|_\infty\leq 16 d^2\a$ and 
$\|Q''\|_{U^2}^*\leq(\a/4)^{-3d^2}(\sigma/2d)^{-3d/4}(\a d)^{-d/4}$. A
small back-of-envelope calculation shows that this is at most the bound
stated for $\|Q''\|_{U^2}^*$.
\end{proof} 

\section{A more precise decomposition theorem}

Theorem \ref{znfancyquadfourier} stated that every function that is bounded
above in $L_2$ can be decomposed into a linear combination of quadratic
averages plus a sum of two error terms, one of which is small in $U^3$ and
one in $L_1$. The aim of this section is to prove a refinement of this statement. 
Once again, we shall show that a function $f$ with $\|f\|_2\leq 1$ can be decomposed 
as a linear combination of quadratic averages plus a small error. However, 
we shall collect these quadratic averages into a small number of ``clusters''
in such a way that two quadratic averages that belong to the same cluster
will have a low-rank difference. Then the results of the previous section
will allow us to express each cluster as a product of just one quadratic
average with a function with small $U^2$-dual norm. We proved an 
analogous theorem for $\F_p^n$: after the hard work of the previous 
section, the rest of the adaptation is relatively routine.

First let us combine Theorem \ref{u2u2*dichotomy} with Lemma 
\ref{quadaveproduct} in order to describe what happens if two generalized
quadratic averages have a significant correlation. The following lemma
should be thought of as a companion to Corollary \ref{znquadavecorr}.
The appearance of the generalized quadratic average $Q_0$ in the
statement may look a bit strange: it is there for technical reasons that
will be explained later.


\begin{lemma} \label{closeinu2*}
Let $B$ and $B'$ be two arbitrary Bohr sets and let the
complexity of $B\cap B'$ be $(d,\rho)$. Let $q$ and $q'$
be quadratic forms on $B$ and $B'$. Let $Q$ and $Q'$ be generalized
quadratic averages with bases $(B,q)$ and $(B',q')$ and suppose that
$\sp{Q,Q'}\geq\zeta$. Let $Q_0$ be another generalized quadratic average 
with base $(B,q)$. Then there exists a function $Q'''$ such that 
$\|Q_0\ol{Q'}-Q'''\|_\infty\leq\zeta/2+d^2\zeta^8/2^5$ and
$\|Q'''\|_{U^2}^*\leq(2^{11}/\zeta^8)^{4d^2}(800d^2/\rho)^d$.
\end{lemma}

\begin{proof}
Let $\eta=\zeta/36$. Let $B_1$ and $B_2$ be regular Bohr sets such that 
$B_2\prec_\eta B_1\prec_\eta B\cap B'$. Then Lemma \ref{quadaveproduct}
tells us that there is a generalized quadratic average $Q''$ with base
$(B_2,q-q')$ such that $\|Q\ol{Q'}-Q''\|_\infty\leq 18\eta=\zeta/2$.

Since $Q_0$ also has base $(B,q)$, the same argument gives us a generalized
quadratic average $Q_0''$ with base $(B_2,q-q')$ such that 
$\|Q_0\ol{Q'}-Q_0''\|_\infty\leq\zeta/2$.

Now $\E_xQ''(x)\geq\zeta-18\eta=\zeta/2$, so $\|Q''\|_{U^2}\geq\zeta/2$. Therefore,
if we set $\a$ to be $\zeta^8/2^9$, then Theorem \ref{u2u2*dichotomy} implies
that there exists a function $Q'''$ such that $\|Q''-Q'''\|_\infty\leq 16d^2\a$ and
$\|Q'''\|_{U^2}^*\leq(4/\a)^{4d^2}(800d^2/\rho)^d$.

However, we wanted a similar statement for $Q''_0$ rather than $Q''$. This
does not quite follow from Theorem \ref{u2u2*dichotomy}, but it follows from
the proof. A quick examination of Corollary \ref{rankdichotomy} reveals that
the alternatives in question depend just on the rank of $Q$ and not on $Q$
itself. (To be precise, if two quadratic forms have the same base and the same
rank with respect to $P$, then there must be one half of the dichotomy that
applies to both forms.) Therefore, we obtain the result stated. 
\end{proof}

What will be crucial to us later, if we want a reasonable bound, is that the
$U^2$-dual norm of $Q'''$ in the above lemma depends polynomially on
$\zeta$ for fixed $d$. It is for this reason that it would have been too 
expensive to use Green and Tao's local Bogolyubov lemma to prove
Theorem \ref{u2u2*dichotomy}. That would have allowed us to prove 
an analogue of Corollary \ref{ddcase} for bilinear phase
functions defined on Bohr sets. However, the subset we passed to would then
have been a Bohr
set whose dimension depended polynomially on $c$, whereas in fact we passed
to a multidimensional progressions without any increase in dimension. 
That would have translated into an exponential dependence on $\zeta$ in 
Lemma \ref{closeinu2*}.
\medskip

Unfortunately, before we prove our more precise decomposition result we 
must deal with another technical difficulty that did not arise for quadratic
averages on $\F_p^n$, which is that $Q^{-1}$ does not in general
equal $\ol{Q}$. In our previous paper it was convenient to write
$Q_j$ as $Q_i\ol{Q_i\ol{Q_j}}$. In order to do something similar in
the $\Z_N$ case we shall first show that for every regular Bohr
neighbourhood $B$ and every quadratic function $q:B\ra\Z_N$ we can
find a smaller Bohr neighbourhood $B'$ and a quadratic average $Q$
with base $(B',q)$ such that $|Q(x)|=1$ for almost every $x$. The
statement of Lemma \ref{closeinu2*} is designed so that we will then be 
able to replace any given $Q_i$ by a quadratic average with this convenient
property and with the same base.

We begin by proving, using a very standard argument, that $\Z_N$ can be
covered fairly efficiently by copies of $B$. 


\begin{lemma} \label{bohrcover}
Let $B=B(K,\rho)$ be a Bohr set and write $d$ for 
the size of $k$ and $\b$ for the density of $B$. Then 
there is a set $\{\seq B m\}$ of translates of $B$ such that 
$m\leq 5^d\b^{-1}$ and every point in $\Z_N$ belongs to at 
least one $B_i$.
\end{lemma}

\begin{proof} 
A basic fact about Bohr sets is that the Bohr set
$B''=|B(K,\rho/2)|$ has density at least $5^{-d}\b$. (See for example
\cite{Green:2009ve}, Lemma 8.1.) Let $\seq x m$ be a maximal collection
of points with the property that the translates $x_i+B''$ are disjoint.
Then $m\leq 5^d\b^{-1}$. Also, the sets $x_i+B$ cover $\Z_N$, since
if $x\notin x_i+B$ for any $i$, then $x+B''$ and $x_i+B''$ are disjoint (or $x$
would belong to $x_i+B''-B''\subset x_i+B$).
\end{proof}

The condition $B'\prec_{\e/5^d}B$ that appears in the next corollary
may look rather expensive with its exponential dependence on $d$,
but the effect on our eventual bound is not particularly serious:
the density of $B'$ is exponential in $d^2$ instead of $d$. When we
come to apply the result, $d$ will be bounded above by $(2/\d)^{C_0}$
for some absolute constant $C_0$, so this decrease in the density
is comparable to the result of replacing $C_0$ by $2C_0$.


\begin{corollary} \label{mod1average}
Let $\e>0$, let $B=B(K,\rho)$ be a regular Bohr set, and
let $q$ be a quadratic function defined on $B$. Let $d=|K|$, let $B'\prec_{\e/5^d}B$ and let $B''$ be a Bohr set such
that $B''-B''\subset B'$. Then there is a quadratic average
$Q$ with base $(B'',q)$ such that for all but at most $\e N$
values of $x$ the restriction of $Q$ to $x+B''$ is a quadratic
phase function. In particular, $|Q(x)|=1$ for all but at most
$\e N$ values of $x$.
\end{corollary}

\begin{proof}
Let $\seq B m$ be a sequence of translates of $B$ given by the 
previous lemma, with $B_i=x_i+B$. On each $B_i$, let $q_i$ be the
function $q_i(x)=q(x-x_i)$. Now let us greedily make the sets
$B_i$ disjoint, by letting 
$B_i'=B_i\setminus(B_1\cup\dots\cup B_{i-1})$ for each $i$. 

We are trying to define a function of the form $Q(x)=\E_{y\in
x-B''}\omega^{q_y(x)}$, so it remains to choose the functions $q_y$
appropriately. This we do by letting $q_y(x)=q_i(x)$ for the unique
$i$ such that $y\in B_i'$.  Since $q_i(x)=q(x-x_i)=q((x-y)-(x_i-y))$,
this is of the form $q(x-y)+\phi_y(x-y)$ for some Freiman homomorphism
$\phi_y:B'\ra\Z_N$, as required.

Now each $q_i$ is a quadratic homomorphism on $B_i$, so the 
restriction of $Q$ to $x+B''$ will be a quadratic phase function
if there exists $i$
such that $x+B''-B''\subset B_i'$. A sufficient condition for
this is that, for every $i$, either $x-B'\subset B_i$ or
$(x-B')\cap B_i=\emptyset$. But Lemma \ref{boundarylemma} implies
that this is true for all but at most $5^{-d}\e m|B|\leq\e N$ values 
of $x$, as claimed.
\end{proof}

The property we have just obtained is a useful one, so let
us give it a name. Note that the Bohr set $B$ from Corollary \ref{mod1average} is no longer explicitly mentioned, but its width and dimension appear (in disguised form) as the parameter $m$ below.

\begin{definition}We say that a quadratic average $Q$ with base $(B'',q)$ is $(\e,m)$-\emph{special} if the following holds.
There exist at most $m$ elements $x_{1}, \dots, x_{m} \in \Z_{N}$ such that for all but at most $\e N$ points $x\in \Z_{N}$ the restriction of $Q$ to $x+B''$ 
is equal to the restriction of $\omega^{q_{i}}$ to $x+B''$, where $q_{i}(x)=q(x-x_{i})$. 
\end{definition}

We shall not need this definition in the rest of this
section, but it will be used in the next section.

The following lemma (which has a simple proof) appears in 
\cite{Gowers:2007tcs} as Corollary 2.11.


\begin{lemma} \label{unitvectors}
Let $\seq u n$ be a collection of vectors of norm at most 1 in a
Hilbert space $H$, let $\seq \lambda n$ be scalars with $\sm i
n|\lambda_i|\leq C$ and let $\d>0$. Then there are vectors
$u_{i_1},\dots,u_{i_k}$ and a set $A\subset\{1,2,\dots,n\}$ such that
$k\leq 2C^2/\d^2$, and with the following properties:  
$\|\sum_{i\in A}\lambda_iu_i\|_2\leq\d$, and for every $i\notin A$
there exists $j$ such that $|\sp{u_i,u_{i_j}}|\geq\d^2/2C$.
\end{lemma}

We are now ready to state and prove the main result of this section. It is important for us to be able to vary the parameter $\eps$ below independently of the quantity $C$.


\begin{theorem} \label{znquadfouriersmallk}
Let $f:\Z_N\ra\C$ be a function such that $\|f\|_2\leq 1$, and
let $\delta>0$. Let $C_0=2^{24}$, $d=(2/\d)^{C_0}$,
$\rho=(\d/2)^{C_0}$ and $C=4(2/\d^{2})^{C_0}$, and let $\e>0$ be at most $(\d/2)^{5C_0}$. 
Then $f$ has a decomposition
\begin{equation*}
f(x)=\sm i k Q_i'(x)U_i(x)+g(x)+h(x),
\end{equation*}
with the following properties: $k\leq 2C/\d^2$, the $Q_i'$ are
quadratic averages on $\Z_N$ with complexity at most $(d,\e\rho/800d5^d)$,
$\sum_{i=1}^k\|U_i\|_{U^2}^*\leq (8/\e^8)^{4d^2}(2^{20}d^35^d/\e\rho)^dC$,
$\sm i k\|U_i\|_\infty\leq 2C$, $\|g\|_1\leq 3\d$ and $\|h\|_{U^3}\leq\d$. Moreoever, the quadratic averages $Q_{i}'$ are $(\eps,(5/\rho)^{d})-special$.
\end{theorem}

\begin{proof}
By Theorem \ref{znfancyquadfourier}, $f$ can be decomposed  
into a sum $\sum_i\lambda_iQ_i(x)+g'(x)+h(x)$, where each 
$Q_i$ is a quadratic average of complexity at most $(d,\rho)$, 
and $\|g'\|_1\leq\d$, $\|h\|_{U^3}\leq\d$ and 
$\sum_i|\lambda_i|\leq C$. 

Suppose that $Q_i$ has base $(B_i,q_i)$. Lemma \ref{bohrcover} and Corollary \ref{mod1average} 
tell us that if $B_i'\prec_{\e/5^d}B_i$ then there is a quadratic
average $Q_i'$ with base $(B_i',q_i)$ which is $(\eps, 5^{d}\beta^{-1})$-special, where $\beta$ is the density of the base of $Q_{i}$. In particular, $|Q_i'(x)|=1$ for
all but at most $\e N$ values of $x$. Furthermore, Lemma
\ref{smallerbase} gives us a generalized quadratic average $Q_i''$ with
base $(B_i',q_i)$ such that $\|Q_i-Q_i''\|_\infty\leq 6\e/5^d$, which
is at most $2\e$. Note that the complexities of $Q_i'$ and $Q_i''$
are at most $(d,\e\rho/800d5^d)$. (The additional factor of 2 stems from the requirement that $B''-B'' \subseteq B'$ in Corollary \ref{mod1average}.)

Now we apply Lemma \ref{unitvectors} to the linear combination
$\sum_i\lambda_iQ_i$. Without loss of generality, the functions that
it gives us are $\seq Q k$. Then Corollary \ref{unitvectors} tells us
that we can write $\sum_i\lambda_iQ_i$ in the form
$\sm i k \sum_{j \in A_i} \lambda_j Q_j+g''$, where $k\leq 2C^2/\d^2$, $\|g''\|_2\leq\d$ and $|\sp{Q_i,Q_j}|\geq\d^2/2C$ for every $i\leq k$ and every $j\in A_i$.
In order to proceed, we must rewrite this decomposition in terms of
the functions $Q_i'$. That is, we wish to take the sum $\sum_{j\in A_i}\lambda_{j}Q_j$
and replace it by $Q_i'\sum_{j\in A_i}\lambda_{j}\ol{Q_i'}Q_j$. 

Since $\|Q_i-Q_i''\|_\infty\leq 2\e\leq\d^2/4C$, we have that
$|\sp{Q_i'',Q_j}|\geq 2\e$ for every $j\in A_i$. Therefore, since
$Q_i'$ and $Q_i''$ have the same base, Lemma \ref{closeinu2*} (with
$\zeta=2\e$ and $\rho$ replaced by $\e\rho/800d5^d$) tells us that 
each function $\ol{Q_i'}Q_j$ with $j\in A_i$ can 
be written as $F+G$, with $\|G\|_\infty\leq\e+8d^2\e^8$ and 
$\|F\|_{U^2}^*\leq (8/\e^8)^{4d^2}(2^{20}d^35^d/\e\rho)^d$.
Therefore, $\sum_{j\in A_i}\lambda_j\ol{Q_i'}Q_j$ can be written as
$F+G$ with $\|G\|_\infty\leq(\e+8d^2\e^8)\sum_{j\in A_i}|\lambda_j|$ and
$\|F\|_{U^2}^*\leq (8/\e^8)^{4d^2}(2^{20}d^35^d/\e\rho)^d\sum_{j\in A_i}|\lambda_j|$.  
(In this proof the functions $F$ and $G$ may vary from line to line.) This implies
also that $\|F\|_\infty\leq 2\sum_{j\in A_i}|\lambda_j|$ (since, as can easily be
checked, $\e+8d^2\e^8\leq 1$).

Since $|Q_i'(x)|^2=1$ for all but at most $\e N$ values of $x$, 
we have the estimate $\|1-|Q_i'|^2\|_2\leq\e$. It follows that
\[\Bigl\|\sum_{j\in A_i}\lambda_jQ_j
-Q_i'\sum_{j\in A_i}\lambda_j\ol{Q_i'}Q_j\Bigr\|_1
\leq\e\sum_{j\in A_i}|\lambda_j|.\]
Hence, $\sum_{j\in A_i}\lambda_jQ_j$ can be written
in the form $Q_i'U_i+V_i$ with $\|V_i\|_1\leq(2\e+8d^2\e^8)\sum_{j\in A_i}|\lambda_j|$
and $\|U_i\|_{U^2}^*\leq (8/\e^8)^{4d^2}(2^{20}d^35^d/\e\rho)^d\sum_{j\in A_i}|\lambda_j|$. 
Our bound for $\|F\|_\infty$ in the previous paragraph also gives
us that $\|U_i\|_\infty\leq 2\sum_{j\in A_i}|\lambda_j|$.

Putting all this together, we find that $\sum_{i=1}^k\sum_{j\in A_i}\lambda_jQ_j$
can be written as $\sum_{i=1}^kU_iQ_i'+V$, with
$\sum_{i=1}^k\|U_i\|_{U^2}^*\leq (8/\e^8)^{4d^2}(2^{20}d^35^d/\e\rho)^dC$,
$\sum_{i=1}^k\|U_i\|_\infty\leq 2C$, and $\|V\|_1\leq (2\e+8d^2\e^8)C\leq\d$.
We have therefore written $f$ as $\sum_{i=1}^kU_iQ_i'+(V+g''+g')+h$. Since
all of $\|V\|_1$, $\|g''\|_1$ and $\|g'\|_1$ are at most $\d$, the theorem is proved.
\end{proof}

\section{The structure of a function $QU$ when $Q$ has low rank}

The main result of the previous section gives us a decomposition of the
form $f=\sum_{i=1}^kQ_i'U_i+g+h$, where $g$ and $h$ are
error terms, the functions $U_i$ have bounded $U^2$-dual norms, and
the $Q_i'$ are quadratic averages. Moreover, the quadratic averages are 
$(\eps,m)$-special, an important property which we shall make use of shortly.

The aim of this section is to find a ``structured set" $S$ such that the 
functions $\omega^{q(x-x_i)+\phi_i(x-x_i)}$ are all approximately 
$S$-invariant, where this means that they do not vary much if you 
add an element of $S$ to $x$. We already have many of the tools
to do this: the main task of this section will be to develop a little further 
some of the results of the last two sections. We shall soon say what a
structured set is, but one can think of it as a set that resembles a
lattice convex body in the way that a Bohr set or a multidimensional
arithmetic progression does.

As in Section \ref{lowrankbilinear} we shall make use of the fact 
that a quadratic homomorphism defined on a multidimensional
arithmetic progression can be explicitly described. We shall 
use elements from the proofs of some of the lemmas in that section.

It may seem as though the next lemma has basically already been
proved in Section \ref{lowrankbilinear}. In a sense, that is true, but
we need to run the argument again in order to make very clear that
the phase function $f$ that appears in the statement below is 
independent of the translate of $P'$. Later we shall see why that
is so important.


\begin{lemma}\label{linearapprox}
Let $B$ be a regular Bohr set, let $P$ be a $d$-dimensional 
arithmetic progression such that $P+P\subset B$, let $q$ be a
quadratic homomorphism defined on $B$, let $Q(x)=\omega^{q(x)}$ for
every $x\in B$, and suppose that the rank of $Q$ with respect to $P$ 
is at most $\log(1/\alpha)$. Let $\e>0$ and let $\theta=\a^2\e/8d^2$. 
Then there is a subprogression $P'\subset P$ of dimension $d$ and 
size at least $(\a/8)^{2d^2}\theta^d|P|$, and a multiplicative Freiman 
homomorphism $f$ from $P$ to the unit circle in $\C$, such that 
$|Q(x)f(x)-Q(y)f(y)|\leq\e$ whenever $x-y\in P'$. Moreover, if
$Q'=Qg$ for some multiplicative homomorphism $g$, then we 
can choose the same subprogression $P'$ to work for $Q'$.
\end{lemma}

\begin{proof}
First, recall from the proof of Corollary \ref{rankdichotomy} that
the restriction of $q$ to $P$ is given by a formula of the form
$q(x)=\sum_{i,j}a_{ij}x_ix_j+\sum_ib_ix_i+c$. Here, we are writing
a typical point $x\in P$ as $u_0+\sum_ix_iu_i$, where $0\leq x_i<m_i$.
Moreover, there are coefficients $b_i'$ and $c_i'$ such that, setting 
$\b(u,v)=2\sum_{ij}a_{ij}u_iv_j+b_i'u_i+c_i'v_i$, we have $|\E_{u,v\in P}e(\b(u,v))|\geq\alpha$.

Corollary \ref{ddcase} (with $\alpha=2c$) then gives us rational approximations 
$|2a_{ij}-p_{ij}/q_{ij}|\leq 8\a^{-2}/m_im_j$, with $q_{ij}\leq 4\alpha^{-1}$. 

Now let $x=u_0+\sum_ix_iu_i$ and $y=u_0+\sum_i(x_i+w_i)u_i$ 
be two points in $P$. Then 
\begin{equation*}
q(y)-q(x)=\sum_{i,j}a_{ij}w_iw_j+\sum_i(b_i+\sum_ja_{ij}x_j)w_i
+\sum_j(b_j+\sum_ia_{ij}x_i)w_j.
\end{equation*}
As in the proof of Corollary \ref{ddcase}, let $q_i=\prod_jq_{ij}\times\prod_jq_{ji}$.
Then $q_i\leq (4\alpha^{-1})^{2d}$. Suppose now that each $w_i$ is even and a multiple
of $q_i$ and that $w_i\leq \theta m_i$. Then $\|a_{ij}w_ix_j\|$ and $\|a_{ij}w_iw_j\|$
are both at most $8\a^{-2}\theta$, since $w_i$ is an even multiple of $q_{ij}$,
$|2a_{ij}-p_{ij}/q_{ij}|\leq 8\a^{-2}/m_im_j$, and $w_j$ and $x_j$ are both at
most $m_j$. Therefore, if $\theta\leq\a^2\e/8d^2$, we find that 
\begin{equation*}
\omega^{q(y)-q(x)}\approx_\e e(2\sum_ib_iw_i).
\end{equation*}
Let us therefore define $f(x)$ to be $e(-2\sum_ib_ix_i)$. Then
\begin{equation*}
|Q(y)f(y)-Q(x)f(x)|=|\omega^{q(y)-q(x)}e(-2\sum_ib_iw_i)-1|\leq\e,
\end{equation*}
which proves the first statement.

The second statement is trivial: if $Q'=Qg$ then all we have to do is
choose the same subprogression $P'$ and replace $f$ by $fg^{-1}$.
\end{proof}

It follows from this lemma that if $X$ is a set on which $f$ is
approximately equal to $1$, then $Q$ is roughly constant on translates
of $X\cap P'$. We shall now prove that such sets have a structure that
is similar to that of Bohr sets.

To do this, we shall make use of the notion of \textit{Bourgain
systems}. This is an abstract notion introduced by Green and Sanders
\cite{Green:2006qvi} that is designed to capture the properties one
actually uses of Bohr sets in most applications. A Bourgain system of
dimension $d$ is a collection of sets $X_\rho$, one for each
$\rho\in[0,4]$, satisfying the following properties.

\begin{itemize}
\item If $\rho'\leq\rho$ then $X_{\rho'}\subset X_\rho$.
\item $0\in X_0$.
\item $X_\rho=-X_\rho$.
\item If $\rho+\rho'\leq 4$ then $X_\rho+X_{\rho'}\subset X_{\rho+\rho'}$.
\item If $\rho\leq 1$ then $|X_{2\rho}|\leq 2^d|X_\rho|$.
\end{itemize}

An important fact about Bourgain systems is that there is an analogue
of the notion of a regular Bohr set. The next lemma is Lemma 4.12 of
\cite{Green:2006qvi} (though we have stated it slightly differently).

\begin{lemma}\label{regularbourgain}
Let $(X_\rho)$ be a Bourgain system of dimension $d$ and let 
$0<\tau\leq 1$. Then there exists $\rho\in[\tau/2,\tau]$ such that
$|X_{\rho(1+\kappa)}|\leq(1+10d\kappa)|X_\rho|$ and 
$X_{\rho(1-\kappa)}\geq(1-10d\kappa)|X_\rho|$ whenever 
$0\leq 10d\kappa\leq 1$.
\end{lemma}

If $\rho$ has this property, we shall call $X_\rho$ a \textit{regular
set} in the system $(X_\rho)$. (This terminology is not quite the same
as that of Green and Sanders, but is close to the standard terminology
for Bohr sets.)

As we did for Bohr sets, we define a notion of one set in a Bourgain
system being ``central" in another.

\begin{definition} Let $(X_\rho)$ be a Bourgain system and let $0<\sigma<\rho\leq 1$.
We shall say that $X_\sigma$ is $\e$-\emph{central in} $X_\rho$ and write 
$X_\sigma\prec_\e X_\rho$ if both $X_\rho$ and $X_\sigma$ are regular sets,
and $\sigma\in[\e\rho/400d,\e\rho/200d]$.
\end{definition}

Note that by Lemma \ref{regularbourgain} we know that if $X_\rho$ is regular 
then there exists $\sigma$ such that $X_\sigma\prec_\e X_\rho$.
Lemma 4.4 of \cite{Green:2006qvi} asserts that if $(X_\rho)$ is a Bourgain
system of dimension $d$ and $\eta\in[0,1]$, then $|X_{\eta\rho}|\geq(\eta/2)^d|X_\rho|$.
Therefore, if $X_\sigma\prec_\e X_\rho$, we know that 
$|X_\sigma|\geq(\eps/800d)^{d}|X_\rho|$. We also obtain a lower 
bound for the sizes of the sets in a dilated system 
$(Y_\rho)=(X_{\eta\rho})$, which will be useful to us later.

The next lemma we state without proof because the proof is almost identical
to that of Lemma \ref{bohraveraging} (i).

\begin{lemma} \label{bourgainaveraging}
Let $\e>0$. Let $(X_\rho)_{0\leq\rho\leq 4}$ be a Bourgain system and let
$0<\sigma<\rho$ be such that $X_\sigma\prec_\e X_\rho$. Let $f$ be any
function from $\Z_N$ to $\C$ such that $\|f\|_\infty\leq 1$. Then
\begin{equation*}
\E_{x\in X_\rho}f(x)\approx_\e \E_{x\in X_\rho}\E_{y\in X_\sigma} f(x+y).
\end{equation*}
\end{lemma}

Obvious candidates for Bourgain systems are families of subgroups, Bohr sets and multidimensional arithmetic progressions. For example, given a Bohr set $B=B(K,\sigma)$, the set
\[X_{\rho}=\{x \in \Z_{N} : |1-e(rx/N)| \leq \rho \sigma \mbox{ for all } r \in K \}\]
obviously satisfies the first four of the above properties, and it also satisfies the final one with $2^{d}$ replaced by $5^{|K|}$ (see for example Section 8 of \cite{Green:2008py}). Therefore the sets $X_\rho$ can be viewed as forming a Bourgain system of dimension $d\leq 3|K|$. A similar statement holds for a family of multidimensional arithmetic progressions with the same basis but differing widths.

We shall not yet explain in detail why Bourgain systems are
useful. Instead, we shall introduce the Bourgain system we wish to use,
prove that it \textit{is} a Bourgain system, and then when we need it
to satisfy various properties we shall quote appropriate results that
tell us that all Bourgain systems have those properties. The proofs
are not too hard and can be found in \cite{Green:2006qvi}.


\begin{lemma}\label{bourgainsystem}
Let $m_1,\dots,m_d$ be positive integers and let $P$ be the set 
$\prod_{i=1}^d[-m_i,m_i]$. Let $f_1,\dots,f_M$ be multiplicative
Freiman homomorphisms from $P$ to $\T$ that take the value 
$1$ at $0$, and for each $\rho\in[0,4]$ let 
\begin{equation*}
X_\rho=\{x\in P:|1-f_j(x)|\leq\rho\ \mathrm{for\ every}\ j\leq M\}
\end{equation*}
Then the sets $(X_\rho)$ form a $2M$-dimensional Bourgain system.
Moreover, the relative density of $X_\rho$ in $P$ is at least 
$3^{-d}(\rho/2\pi)^M$.
\end{lemma}

\begin{proof}
The first three properties hold trivially. The fourth is almost trivial: the
simple calculation needed is that if $f$ is a multiplicative homomorphism
to $\T$, $|1-f(x)|\leq\rho$, and $|1-f(y)|\leq\rho'$, then 
\[|1-f(x+y)|\leq |1-f(x)|+|f(x)-f(x+y)|=|1-f(x)|+|1-f(y)|\leq \rho+\rho'.\]
The only real work comes in proving the fifth property, which
bounds the size of $B_{2\rho}$ in terms of the size of $B_\rho$.
The argument here is essentially the same as it is for Bohr sets. 
Let us define a map $\psi:P\ra\T^d$ by 
$\psi(x)=(f_1(x),\dots,f_M(x))$. 
Then $\psi(X_{2\rho})\subset\T_{2\rho}^M$,
where $\T_{2\rho}=\{z\in\C:|z|=1,|1-z|\leq 2\rho\}$.

We can cover $\T_{2\rho}$ by four segments of
the circle that have diameter at most $\rho$. Let us use
all $4^{M}$ possible products of these sets to cover the set
$\T_{2\rho}^M$. If $Z$ is one of these
products and $\psi(x)$ and $\psi(y)$ both belong to $Z$,
then $\psi(x-y)\in\T_\rho^M$, which 
implies that $x-y\in X_\rho$, or equivalently that $x\in y+X_\rho$. 
It follows that if we choose one $y$ for each $Z$ for which 
there exists $y$ with $\psi(y)\in Z$, then we have a system of
at most $4^{M}$ translates of $X_\rho$ that cover 
$X_{2\rho}$. Therefore, the sets $X_\rho$ form a 
Bourgain system of dimension $2M$, as claimed.

Now let us turn to the density estimate, which is proved in 
a similar way. For each $z=(z_1,\dots,z_M)\subset\T^M$, 
let $\T_{\rho/2}^{M}(z)$ be the set of all $w=(w_1,\dots,w_M)\in\T^M$
such that $|z_i-w_i|\leq\rho/2$ for every $i$. For any $z_i\in\T$, the arc of 
points $w_i\in\T$ such that $|z_i-w_i|\leq\rho/2$ has length
at least $\rho$, so the density of $\T_{\rho/2}(z)$ is at least
$(\rho/2\pi)^M$.

Let us write $P/2$ for the set $\prod_{i=1}^d[-m_i/2,m_i/2]$.
Then $|P/2|\geq 3^{-d}|P|$ (because the worst case is when
every $m_i$ is equal to 1). Hence, by averaging we can find
$z\in\T^M$ such that $\psi(x)\in\T_{\rho/2}^{M}(z)$ for at least 
$3^{-d}(\rho/2\pi)^M|P|$ points $x\in P/2$. Let $x$ be any such
point. If $y$ is any other such point, then $y-x\in P$
and $\psi(x)$ and $\psi(y)$ both belong to $\T_{\rho/2}^{M}(z)$, 
which implies that $\psi(y-x)=\psi(y)\ol{\psi(x)}\in\T_\rho^{M}$, 
so $y-x\in X_\rho$. Hence $X_{\rho}$ must contain at least 
$3^{-d}(\rho/2\pi)^M|P|$ distinct points.
\end{proof}

Let $P\subset B'$ be a proper generalized progression. We
shall need a lemma to tell us that we can cover $\Z_N$
reasonably efficiently with translates of $P$. The proof
is essentially the same as the proof of Lemma \ref{bohrcover}.


\begin{lemma}\label{GAPcover}
Let $P\subset\Z_N$ be a proper arithmetic progression of dimension 
$d$ and density $\g$. Then there is a system of at most 
$3^d\g^{-1}$ translates of $P$ that covers $\Z_N$.
\end{lemma}

\begin{proof}
Let $P=\{\sum_{i=1}^da_ix_i:0\leq a_i<m_i\}$ and let
$P'=\{\sum_{i=1}^da_ix_i:0\leq a_i<m_i/2\}$. Then let 
$u_1,\dots,u_M$ be a maximal set such that the sets
$P'+u_i$ are disjoint. Note that $P'-P'=\{\sum_{i=1}^da_ix_i:
-m_i/2<a_i<m_i/2\}\subset P-\sum_i\lfloor m_i/2\rfloor x_i$.
Let $z=\sum_i\lfloor m_i/2\rfloor x_i$.

Then the sets $P+u_i-z$ form a
cover, since for every $x$ there exists $u_i$ such
that $(x+P')\cap(u_i+P')\ne\emptyset$, which implies
that $x\in u_i+P'-P'\subset u_i+P-z$. Since $P'$ has
cardinality at least $3^{-d}|P|$ and therefore density
at least $3^{-d}\g$, the result is proved.
\end{proof}

For the next lemma we shall make use of the concept of
``special" quadratic averages, which was defined just after
the proof of Corollary \ref{mod1average}.

\begin{corollary}\label{specialquadaves}
Let $B$ be a Bohr set of dimension $d$ and let $q$ be a quadratic homomorphism defined on $B$. Let $m$ be a positive integer and let $B'\prec_{\eps/5^{d}} B$ be another Bohr set. Let $Q$ be an 
$(\e,m)$-special quadratic average with base $(B',q)$; in other words, for all but at most $\e N$ points $x\in \Z_{N}$ the restriction of $Q$ to $x+B'$ is equal to the restriction of $\omega^{q_{i}}$ to $x+B'$, where $q_{i}$ is one of at most $m$ translates of $q$. 
Let $P\subset\Z_N$ be a proper generalized arithmetic 
progression of dimension $d$ and density $\g$ such 
that $2P-2P\subset B'$. Then there is a set $V$ of size
at most $3^{d}\g^{-1}$ such that for at least $(1-\e)N$ values
of $x$ there exists $i\leq m$ and $v\in V$ such that 
$x+P-P\subset v+2P-P$ and the restriction of $Q$ to 
$v+2P-P$ is equal to $\omega^{q_i}$.
\end{corollary}

\begin{proof}
By Lemma \ref{GAPcover}, there is a set $V$ of size at
most $2^d\g^{-1}$ such that every $x$ is in $v+P$ for
some $P$. If $x\in v+P$ then $u\in x-P$, so $v+P\subset x+P-P$.
Since we assumed that $P$ was such that $2P-2P\subset B'$,
we find that $x+P-P\subset v+2P-P\subset x+B'$. Since 
$Q$ is $(\e,m)$-special, the proportion of $x$ such that 
the restriction of $Q$ to $x+B'$ is equal to $\omega^{q_i}$
for some $i$ is at least $1-\e$. 

Therefore, as claimed, for at least this
proportion of $x$, we have some $v\in V$ such that 
$x+P-P\subset v+2P-P$ and the restriction of $Q$ to
$v+2P-P$ is equal to $\omega^{q_i}$ for some~$i$. 
\end{proof}

We now come to the main result of this section. The bound may
look somewhat complicated, so let us draw attention to the one
feature of it that is very important to us: that the dependence on
$\alpha$ is of a power type rather than exponential. It is for this
that we have put in the work of the last three sections rather than
simply applying the local Bogolyubov lemma. (The fact that the
power depends on $d$ is quite expensive, but it produces a
doubly exponential bound rather than the tower-type bound
that would have resulted if $\alpha$ had appeared in the exponent.)


\begin{lemma} \label{linearfactor}
Let $Q$ be a quadratic average that satisfies all the assumptions 
of the previous lemma and suppose that the rank of $Q$ is
at most $\log(1/\a)$ with respect to $P$. Let $\eta>0$ and let $\theta=\a^2\eta/8d^2$.
Then there is a subprogression $P'\subset P$ of relative density at 
least $(\a/8)^{2d^2}\theta^d$ and a Bourgain
system $(X_\rho)$ of dimension $2m$ such that each $X_\rho$ is a subset of $P'$,
the relative density of $X_\rho$ is at least $3^{-d}(\rho/2\pi)^m$
inside $P'$, and for every $\rho$ and all but at most $\e N$ values of $x$, 
$|Q(y)-Q(x)|\leq \eta+\rho$ for every $y\in x+X_\rho$.
\end{lemma}

\begin{proof}
Corollary \ref{specialquadaves} implies that for at 
least $(1-\e)N$ values of $x$ there is some $v\in V$
such that $x+P-P\subset v+2P-P$ and the restriction
of $Q$ to $v+2P-P$ is $\omega^{q_i}$ for some translate
$q_i$ of $q$. Lemma \ref{linearapprox} then gives us a 
progression $P'$ of the density stated, and a 
multiplicative homomorphism $f_i$, such that
$|Q(y)f_i(y)-Q(z)f_i(z)|\leq\eta$ whenever $y,z\in v+2P-P$
and $y-z\in P'$. In particular, $|Q(x)f_i(x)-Q(y)f_i(y)|\leq\eta$ 
whenever $y\in x+P'$.

If in addition $|1-f_i(y-x)|\leq\rho$ for each fixed $i$, then 
\[|Q(x)-Q(y)|=|Q(x)f_i(x)-Q(y)f_i(x)|\leq |Q(x)f_i(x)-Q(y)f_i(y)|+|Q(y)||f_i(y)-f_i(x)|,\]
which, using the multiplicative property of $f_{i}$, equals 
\[|Q(x)f_i(x)-Q(y)f_i(y)|+|f_i(y-x)-1|\]
and can therefore be bounded above by $\eta+\rho$. 

By Lemma \ref{bourgainsystem}, the sets 
$X_\rho=\{z\in P':|1-f_i(z)|\leq\rho\ \mathrm{for\ every}\ i\leq m\}$
form a Bourgain system of dimension $2m$, such that
$X_\rho$ has relative density at least $3^{-d}(\rho/2\pi)^m$
inside $P'$. 
\end{proof}

We have just shown that one special quadratic average $Q$ is 
roughly invariant under convolution by sets $X_\rho$ that come
from a certain Bourgain system. We now want to obtain a similar
statement for a combination $\sum_{i=1}^kQ_iU_i$ of functions with small $U^{2}$ dual norm.
The rough idea is to choose for each function $Q_i$ and each
function $U_i$ a set from a Bourgain system with respect to which
it is roughly translation invariant, and then to intersect all these sets.
We shall use a lemma of Green and Sanders \cite{Green:2006qvi} to prove that this
intersection is reasonably large. 

The next lemma is a standard application of Bogolyubov's method.


\begin{lemma}\label{bogolyubov}
Let $f$ be a function from $\Z_N$ to $\C$ and suppose that
$\|f\|_{U^2}^*\leq T$ and $\|f\|_\infty\leq C$. Let 
$K=\{r\in\Z_N:|\hf(r)|\geq\rho\}$ and let $B$ be the Bohr set $B(K,\rho)$. Then 
\begin{equation*}
\E_x|f(x+d)-f(x)|^2\leq\rho^2C^2+4T^{4/3}\rho^{2/3}
\end{equation*}
for every $d\in B$.
\end{lemma}

\begin{proof}
We apply the Fourier inversion formula and split the expectation into two 
parts in the usual manner:
\[\E_x|f(x+d)-f(x)|^2=\E_x|\sum_r\hf(r)(\omega^{r(x+d)}-\omega^{rx})|^2\leq\sum_{r\in B}|\hf(r)|^2|\omega^{rd}-1|^2+4\sum_{r\notin B}|\hf(r)|^2,\]
which is bounded above by
\[\rho^2\|f\|_2^2+4\|\hf\|_{4/3}^{4/3}\rho^{2/3}
\leq \rho^2C^2+4T^{4/3}\rho^{2/3}\]
as claimed, using the fact that $\|\hf\|_{4/3}=\|f\|_{U^{2}}^{*}$.
\end{proof}


\begin{corollary}\label{approxQU}
Let $Q$ be a quadratic average and let $U$ be a function such that 
$\|U\|_\infty\leq C$. Let $X$ be a set such that for at least $(1-\e)N$ values
of $x\in\Z_N$ we have $|Q(x+d)-Q(x)|\leq\eta$ for every $d\in X$. Let $B$ be a set
such that $\E_x|U(x+d)-U(x)|^2\leq\g$ for every $d\in B$. Then
$\E_x|Q(x+d)U(x+d)-Q(x)U(x)|^2\leq 2\eta^2C^2+2\g+4\e C^2$ for 
every $d\in B\cap X$. Consequently, if $S$ is any subset of $B\cap X$ and
$\sigma$ is the characteristic measure of $S$, then 
$\|QU-(QU)*\sigma\|_2\leq 2\eta C+2\g^{1/2}+2\e^{1/2}C$.
\end{corollary}

\begin{proof}
Let $d\in B\cap X$ and let $x\in \Z_N$ be such that $|Q(x+d)-Q(x)|\leq\eta$ for 
every $d\in X$. Then
\[|Q(x+d)U(x+d)-Q(x)U(x)|\leq |Q(x+d)-Q(x)||U(x+d)|+|Q(x)||U(x+d)-U(x)|,\]
which, by assumption, is at most
\[\eta C+|U(x+d)-U(x)|.\]
It follows that for every such $x$ and every $d\in B\cap X$, we have
\[|Q(x+d)U(x+d)-Q(x)U(x)|^2\leq 2\eta^2C^2+2|U(x+d)-U(x)|^2.\]
The proportion of $x$ to which this applies is at least $1-\e$, by hypothesis. For all other $x$,
we can at least say that $|Q(x+d)U(x+d)-Q(x)U(x)|^2\leq 4\|U\|_\infty^2\leq 4C^2$. 
The first statement follows upon taking expectations.

Now
\[\|QU-(QU)*\sigma\|_2^2=\E_x|\E_{d\in S}Q(x+d)U(x+d)-Q(x)U(x)|^2,\]
which by Cauchy-Schwarz and the first assertion is bounded above by
\[\E_{d\in S}\E_x|Q(x+d)U(x+d)-Q(x)U(x)|^2\leq 2\eta^2C^2+2\g+4\e C^2.\]
This proves the second statement.
\end{proof}

Recall that the aim of this section is to deal with a sum $\sum_iQ_i'U_i$ in which the functions $U_{i}$ have small $U^{2}$ dual norm and the quadratic averages $Q_{i}'$ have low rank. 
Putting together what we have proved so far enables us to find, for each $i$, a structured set $S_i$ with characteristic measure $\sigma_i$ such that $(Q_i'U_i)*\sigma_i$ is close to $Q_i'U_i$ in $L_2$. Thus, if we let $S=S_1\cap\dots\cap S_k$ then we have a measure $\sigma$ such that $(\sum_{i=1}^kQ_i'U_i)*\sigma$ is close to $\sum_{i=1}^kQ_i'U_i$ in $L_2$. As well as making these steps formal, we shall need to prove a lower bound for the size of $S_1\cap\dots\cap S_k$.

In order to do so, we generalize a lemma of Green and Sanders about
intersections of sets from Bourgain systems. (It appears in a slightly different form in their paper \cite{Green:2006qvi} as Lemma 4.10.)


\begin{lemma}\label{intBS}
Let $(X_\rho)$ and $(Y_\rho)$ be two Bourgain systems in $\Z_N$ 
of dimensions $d$ and $d'$, and let the densities of each $X_\rho$ and $Y_\rho$ be $\mu_\rho$ and $\nu_\rho$, respectively. Then $(X_\rho \cap Y_\rho)$ is a Bourgain system of dimension at most $4(d+d')$ and $X_\rho\cap Y_\rho$ has density at least $2^{-3(d+d')}\mu_\rho\nu_\rho$ whenever $\rho\leq 1$.
\end{lemma}

We will need to have a similar lemma for more than two Bourgain systems.
We could imitate the proof of Green and Sanders for the case of two systems, 
but for simplicity let us just apply their result and obtain a slightly worse bound.


\begin{corollary}\label{intmultBS}
For $i=1,2,\dots,s$, let $(X_\rho^{(i)})$ be a Bourgain system in $\Z_N$
of dimension $d_i$ and let $X_\rho^{(i)}$ have density $\mu_\rho^{(i)}$. Then the sets $X_\rho^{(1)}\cap\dots\cap X_\rho^{(s)}$ form a 
Bourgain system of dimension at most $4s^2(d_1+\dots+d_s)$ and have
density at least $2^{-4s^2(d_1+\dots+d_s)}\mu_\rho^{(1)}\dots\mu_\rho^{(s)}$
whenever $\rho\leq 1$. 
\end{corollary}

\begin{proof}
It is enough to prove the result when $\rho=1$, since for smaller $\rho$
we can take a dilated system. This allows us to simplify our notation and
write $\mu_i$ for $\mu_1^{(i)}$.

We begin by assuming that $s=2^r$ for some positive integer $r$.
Then we form a new collection of $2^{r-1}$ Bourgain systems by 
intersecting the old ones in pairs. For instance, one of the new
systems is $(X_\rho^{(1)}\cap X_\rho^{(2)})$, which has dimension
at most $4(d_1+d_2)$, and the density of $X_1^{(1)}\cap X_1^{(2)}$
is at least $2^{-3(d_1+d_2)}\mu_1\mu_2$ by Lemma \ref{intBS} above.

Now we pair off the new systems. The dimension of the first system
that results will be at most $16(d_1+d_2+d_3+d_4)$ and the density
when $\rho=1$ will be at least $2^{-15(d_1+d_2+d_3+d_4)}\mu_1\mu_2\mu_3\mu_4$.

In general, after $q$ stages we have a dimension of at most
$4^q(d_1+\dots+d_{2^q})$ and a density when $\rho=1$ of at least 
$2^{-(4^q-1)(d_1+\dots+d_{2^q})}\mu_1\dots\mu_{2^q}$,
as can easily be checked by induction.

This proves the result when $s$ is a power of 2, with bounds of
$s^2(d_1+\dots+d_s)$ and $2^{-(s^2-1)(d_1+\dots+d_s)}\mu_1\dots\mu_s$. 
For general $s$, one can simply take a few more Bourgain systems 
for which every set is equal to $\Z_N$ in order to make up their number 
to the next power of~2.
\end{proof}

We are about to tackle one of the main results of this section, which will eventually allow us to eliminate the low-rank phases from the decomposition when the function $f$ to be decomposed has a sufficiently small $U^2$ norm. Very roughly, we shall find a structured set $S$ that is not too small such that when we convolve the low-rank part of the decomposition with the characteristic measure $\sigma$ of $S$, it remains approximately unchanged. Later, we shall also show that convolving $f$ and the rest of the decomposition of $f$ by $\sigma$ creates a function that is small. From this it follows that the low-rank part of the decomposition is small. This will give us the $\Z_{N}$ analogue of Theorem 5.7 in \cite{Gowers:2009lfuI}. (The proof has the same structure as well, but here the argument is substantially more complicated.)

The parameters in Proposition \ref{approxLRsumQU} below are chosen so that the proposition can be readily applied to the quadratic averages in the decomposition arising from Theorem \ref{znquadfouriersmallk}. An important feature of the precise statement is that the dimension of the Bourgain system $(S_{\rho'})$ it produces does not depend on the rank-related quantity $\alpha$.


\begin{proposition}\label{approxLRsumQU}
Suppose that $\alpha, \d$ and $\zeta$ are positive reals. Let $C_0=2^{24}$, $d=(2/\d)^{C_0}$, $C=4(2/\d^{2})^{C_0}$ and $\rho=(\d/2)^{C_0}$. Let $k$ and $m$ be integers bounded above by $2C/\d^{2}$ and $(5/\rho)^{d}$, respectively. Let $\e>0$ be at most $(\zeta/20kC)^{2}$ and let $T=(8/\e^8)^{4d^2}(2^{20}d^35^d/\e\rho)^dC$. For each $i=1,2,\dots, k$, let $Q_{i}'$ be a quadratic average with base $(B_{i}',q_{i})$ of complexity at most $(d,\eps\rho/800d5^{d})$. Moreover, suppose that each $Q_{i}'$ is an $(\eps,m)$-special average, and that its rank with respect to some $d'$-dimensional progression $P$ of density $\gamma'$ satisfying $2P-2P\subseteq B_{i}'$ for each $i$ is at most $\log(1/\alpha)$. Suppose further that $\sum_{i=1}^{k}\|U_{i}\|_{\infty} \leq 2C$ and that $\sum_{i=1}^{k}\|U_{i}\|_{U^{2}}^{*} \leq T$. Then there exists a Bourgain system $(S'_{\rho'})$ of dimension at most $32k^{3}(m+ 2^{33}k^{6}T^{4}C^{2}/\zeta^{6})$ such that each $S'_{\rho'}$ has density at least 
\[\gamma'^{k} \left(\frac{\alpha^{4}\zeta }{2^{15}kCd'^{2}}\right)^{d'^{2}k}  \left(\frac{\zeta^{4}\rho'}{2^{27}k^{4}CT^{2}}\right)^{64k^{3}(m+ 2^{33}k^{6}T^{4}C^{2}/\zeta^{6})} \] 
such that
\[\left\|\sum_{i=1}^{k}Q_{i}'U_{i}-\left(\sum_{i=1}^{k}Q_{i}'U_{i}\right)*\sigma_{\rho'}\right\|_{2} \leq \zeta\]
for every $\rho'\leq 1$, where $\sigma_{\rho'}$ is the characteristic measure of $S'_{\rho'}$.
\end{proposition}

\begin{proof} Let us begin by fixing some $i\in\{1,2,\dots, k\}$. Let $\eta=\zeta/(20kC)$, and set $\theta=\alpha^{2}\eta/8d'^{2}$. First we apply Lemma \ref{linearfactor} to obtain a subprogression $P_{i}' \subseteq P$ of relative density at least $(\alpha/8)^{2d'^{2}}\theta^{d'}$ and a Bourgain system $(X_{\rho'}^{(i)})$ of dimension at most $2m$ such that each $X_{\rho'}^{(i)}$ is a subset of $P_{i}'$, the relative density of $X_{\rho'}^{(i)}$ inside $P_{i}'$ is at least $3^{-d'}(\rho'/2\pi)^{m}$, and for every $\rho'$ and for all but at most $\eps N$ values of $x$, we have $|Q_{i}'(x+y)-Q_{i}'(x)| \leq \eta+\rho'$ for all $y \in X_{\rho'}^{(i)}$.

Set $\xi=\zeta^{3}/(2^{15}k^{3}T^{2})$, in which case we can check that $\xi$ also satisfies $4\xi C \leq \zeta/5k$. Apply Lemma \ref{bogolyubov} with $\rho=\xi$ and $C$ replaced with $2C$ to find a set $K_{i}$ of cardinality at most $(2C/\xi)^2$ such that 
\[\E_{x}|U_{i}(x+y)-U_{i}(x)|^{2}\leq 4\xi^{2}C^{2}+4T^{4/3}\xi^{2/3}\]
for every $y$ in the Bohr set $B(K_{i},\xi)$, which has density at least $\xi^{(2C/\xi)^2}$. From this we can create a Bourgain system $(A_{\rho'}^{(i)})$ of dimension at most $3(2C/\xi)^{2}$ by setting $A_{\rho'}^{(i)}$ to be the Bohr set $B(K_{i},\rho'\xi)$, in which case the above inequality holds whenever $\rho'\leq 1$ and $y\in A_{\rho'}^{(i)}$. 

Note that for any value of $\rho'$, the function $U_{i}$ and the sets $A^{(i)}_{\rho'}$ and $X_{\rho'}^{(i)}$ satisfy the hypotheses of Corollary \ref{approxQU}. More precisely, for any fixed $\rho'$, Corollary \ref{approxQU} with $X=X_{\rho'}^{(i)}$, $B=A^{(i)}_{\rho'}$, $\gamma=4\xi^{2}C^{2}+4T^{4/3}\xi^{2/3}$, $\eta$ replaced by $\eta+\rho'$ and $C$ replaced with $2C$ tells us that if $S_{i}$ is any subset of $A^{(i)}_{\rho'} \cap X_{\rho'}^{(i)}$, and $\sigma_{i}$ is the characteristic measure of $S_{i}$, then
\[\|Q_{i}' U_{i}-(Q_{i}'U_{i})*\sigma_{i}\|_{2} \leq 4(\eta+\rho')C+2(4\xi^{2}C^{2}+4T^{4/3}\xi^{2/3})^{1/2}+4\eps^{1/2}C.\]
Our parameters $\eta$, $\xi$ and $\eps$ were chosen so that
\[\|Q_{i}' U_{i}-(Q_{i}'U_{i})*\sigma_{i}\|_{2} \leq \zeta/k\]
for each $i=1,2, \dots, k$, provided that $\rho' \leq \zeta/(20kC)$. In particular, letting $S_{\rho'}=(A^{(1)}_{\rho'}\cap X_{\rho'}^{(1)})\cap \dots \cap (A^{(k)}_{\rho'}\cap X_{\rho'}^{(k)})$, and writing $\sigma_{\rho'}$ for the corresponding characteristic measure, we conclude that
\[\left\|Q_{i} 'U_{i}-Q_{i}'U_{i}*\sigma_{\rho'}\right\|_{2} \leq \zeta/k\]
for each $i=1,2, \dots, k$, and hence that
\[\left\|\sum_{i=1}^{k}Q_{i} 'U_{i}-\left(\sum_{i=1}^{k}Q_{i}'U_{i}\right)*\sigma_{\rho'}\right\|_{2} \leq \zeta.\]
Unfortunately, since we are placing a restriction on the size of $\rho'$, the Bourgain system $(S_{\rho'})_{0\leq\rho'\leq 4}$ is not quite the one we are looking for. However, we can easily get round this problem by rescaling: for each $\rho'\in[0,4]$ let us define $S'_{\rho'}$ to $S_{\rho'\zeta/80kC}$ and let us take the Bourgain system $(S'_{\rho'})_{\rho' \in [0,4]}$.

It remains to verify the statements about the dimension and density of the sets $S'_{\rho'}$.
Recall that each set $X_{\rho'}^{(i)}$ had relative density $3^{-d'}(\rho'/2\pi)^{m}$ with respect to $P_{i}'$. This subprogression $P_{i}'$ itself had relative density $(\alpha/8)^{2d'^{2}}\theta^{d'}$ with respect to $P$, and $P$ in turn was assumed to have density $\gamma'$ inside $\Z_{N}$. Therefore, the density of $X_{\rho'}^{(i)}$ inside $\Z_{N}$ is at least
\[\gamma_{\rho'}=\gamma'\left(\frac{\alpha}{8}\right)^{2d'^{2}}\left(\frac{\alpha^{2}\zeta}{480kCd'^{2}}\right)^{d'}\left(\frac{\rho'}{2\pi}\right)^{m}.\]
The dimension of each $X_{\rho'}^{(i)}$ was simply $2m$, and the dimension of $A^{(i)}_{\rho'}$ at most $12(C/\xi)^{2}$. Hence by Lemma \ref{intBS} we find that each $(A^{(i)}_{\rho'}\cap X_{\rho'}^{(i)})$ is a Bourgain system of dimension at most $4(2m+12(C/\xi)^{2})$, and the density of $A^{(i)}_{\rho'}\cap X_{\rho'}^{(i)}$ is at least $2^{-3(2m+12(C/\xi)^{2})}\xi^{4(C/\xi)^{2}}\gamma_{\rho'}$. Finally, by Lemma \ref{intmultBS}, we establish that the Bourgain system $(S_{\rho'})=((A^{(1)}_{\rho'}\cap X_{\rho'}^{(1)}) \cap \dots \cap (A^{(k)}_{\rho'}\cap X_{\rho'}^{(k)}))$ has dimension at most $16k^{3}(2m+12(C/\xi)^{2})$, and that $S_{\rho'}$ has density at least $2^{-(16k^{3}+3k)(2m+12(C/\xi)^{2})}\xi^{4k(C/\xi)^{2}} \gamma_{\rho'}^{k} $, and therefore the dilated Bourgain system $(S'_{\rho'})$ has dimension at most $16k^{3}(2m+12(C/\xi)^{2})$, and $S'_{\rho'}$ has density at least $(\zeta/160kC)^{16k^{3}(2m+12(C/\xi)^{2})}$ times the density of $S_{\rho'}$.

Revisiting our choice of $\xi$, we find that $12(C/ \xi)^{2}\leq 2^{34}k^{6}T^{4}C^{2}/\zeta^{6}$, and hence the dimension of the Bourgain system $(S'_{\rho'})$ satisfies the desired bound. The density of $S'_{\rho'}$ is at least $(\zeta/320kC)^{32k^{3}(2m+2^{33}k^{6}T^{4}C^{2}/\zeta^{6})}(\zeta^{3}/2^{15}k^{3}T^{2})^{2^{33}k^{7}T^{4}C^{2}/\zeta^{6}}\gamma_{\rho'}^{k}$, which can be simplified and bounded below by the quantity given in the statement of the proposition.
\end{proof}

Next, we need a technical lemma that we shall use repeatedly in the rest of the paper. It states that the rank of a quadratic average does not decrease too much when taken with respect to a slightly smaller set. This statement was proved for $\F_{p}^{n}$ using a simple algebraic argument in \cite{Gowers:2007tcs}. As we have already discussed, arguments that depend on dimensions of subspaces do not have direct analogues in $\Z_N$, so instead we shall give an analytic proof. If $\b$ is a bilinear form, let us define $\a_{P}(\b)$ to be $\E_{a,a',b,b' \in P}\omega^{\b(a-a',b-b')}$, and $r_{P}(\beta)=\log \alpha_{P}^{-1}$. Note that if $q$ is a quadratic function that is defined where it needs to be and $\b(a,b)=q(a+b)-q(a)-q(b)$, then $r_P(\b)=r_P(q)$, so all we are doing is attaching the rank of a quadratic function to the associated bilinear function as well. (By a ``bilinear function" we mean a function that is a Freiman homomorphism in each variable separately.)


\begin{lemma}\label{rankdown2}
Let $B'$ be a Bohr set, let $\b$ be a bilinear function defined on $B'\times B'$ and let $P$ and $B''$ be subsets of $B'$ such that $2P-2P \subseteq B'$ and $2B''-2B'' \subseteq B'$. Then 
\[\alpha_{P}(\b) \geq \left( \frac{|P \cap B''|}{|P|} \right)^{4} \alpha_{P \cap B''}.\]
\end{lemma}

\begin{proof}
We shall repeatedly make use of the positivity property of the exponential sum that we used to define the rank of a bilinear form. We start by writing
\[\alpha_{P}(\b)=\E_{x,x',y,y' \in P} \omega^{\beta(x-x',y-y')}=\E_{x \in P} \E_{x' \in P} | \E_{y \in P}\omega^{\beta(x-x',y)}|^{2}=\E_{x \in P} g_{1}(x),\]
where we have written $g_{1}(x)=\E_{x' \in P}|\E_{y \in P}\omega^{\beta(x-x',y)}|^{2}$. Note that $g_{1}$ maps into $[0,1]$. Let $\rho=|P \cap B''|/|P|$. Then the positivity of $g_1$ implies that
\[\alpha_{P}(\b)\geq  \rho\; \E_{x \in P \cap B''}g_{1}(x)= \rho\; \E_{x' \in P} \E_{x \in P \cap B''}|\E_{y \in P}\omega^{\beta(x-x',y)}|^{2}=\rho\; \E_{x' \in P} g_{2}(x'), \]
where this time we have written $g_{2}(x')=\E_{x \in P \cap B''}|\E_{y \in P}\omega^{\beta(x-x',y)}|^{2}$. Again, $g_{2}$ is non-negative so that
\[\alpha_{P}(\b)\geq  \rho^{2} \;\E_{x' \in P \cap B''}g_{2}(x')=  \rho^{2} \;\E_{x,x' \in P \cap B''}|\E_{y \in P}\omega^{\beta(x-x',y)}|^{2}.\]
Interchanging summation, the latter expression equals
\[ \rho^{2}\; \E_{y,y' \in P}|\E_{x \in P \cap B''}\omega^{\beta(x,y-y')}|^{2}= \rho^{2}\; \E_{y \in P}g_{3}(y)\geq  \rho^{3}\; \E_{y \in P \cap B''}g_{3}(y),\]
with $g_{3}(y)=\E_{y'\in P}|\E_{x \in P\cap B''}\omega^{\beta(x,y-y')}|^{2}$, which is again non-negative. Applying the same argument one final time, we see that
\[ \rho^{3}\;\E_{y' \in P} \E_{y \in P \cap B''}|\E_{x \in P\cap B''}\omega^{\beta(x,y-y')}|^{2}= \rho^{3}\;\E_{y' \in P}g_{4}(y') \geq  \rho^{4}\;\E_{y' \in P \cap B''}g_{4}(y'),\]
where $g_{4}(y')= \E_{y \in P \cap B''}|\E_{x \in P\cap B''}\omega^{\beta(x,y-y')}|^{2}$ is non-negative. We have thus shown that 
\[\alpha_{P}(\b)\geq \rho^{4}\; \E_{x,x',y,y' \in P\cap B''} \omega^{\beta(x-x',y-y')}= \rho^{4}\, \alpha_{P \cap B''}(\b),\]
which proves the result.
\end{proof}

We shall also need the following lemma from \cite{Gowers:2009lfuI} that enables us to take a set of not too many quadratic functions and partition it into a ``low-rank part" and a ``high-rank part" in such a way that there is a large gap between the ranks in the two parts. We shall present the lemma in a slightly modified form and give the simple proof of the precise statement we need.


\begin{lemma}\label{rankgap}
Let $R_0$, $b\geq 2$ and $t >1$ be constants. For each $i=1,2,\dots, k$, let $Q_{i}$ be a quadratic average with base $(B_{i}',q_{i})$. Then for any $P \subset \bigcap_{i=1}^{k}B_{i}'$ there is a partition of
$\{1,2,\dots,k\}$ into two sets $L$ and $H$, and a constant
$R\in[R_0,b^k(R_0+t)]$, such that the rank of $Q_i$ with respect to $P$ is at most $R$ for every 
$i\in L$ and at least $bR+t$ for every $i\in H$.
\end{lemma}

\begin{proof} 
Without loss of generality the $Q_i$ are arranged in increasing order of rank with respect
to $P$. If there is no $i$ such that $Q_i$ has rank at least $b^i(R_0+t)$ with respect to $P$,
then let $L=\{1,2,\dots,k\}$ and let $R=b^k(R_0+t)$ and we are done.

Otherwise, let $i$ be minimal such that $Q_i$ has rank at least $b^iR_0+(1+b+\dots+b^{i-1})t$. Set $R=b^{i-1}R_0+(1+b+\dots+b^{i-2})t$. Then for every $j<i$ the rank of $Q_j$ is at
most $R$, and for every $j\geq i$ the rank of $Q_j$ is at least $bR+t$. Since 
$R\leq b^kR_0+(1+b+\dots+b^{k-1})t\leq b^k(R_0+t)$, the lemma is proved. 
\end{proof}


\begin{lemma} \label{highrankcorru2}
Let $Q$ be a quadratic average with base $(B,q)$, let $B_1\prec_\eta B$ and suppose that $Q$ has rank $r$ with respect to a subset $P\subset B_1$. Let $Q'$ be another quadratic average, with base $(B',q')$, where $B'$ has complexity at most $(d,\rho)$. Suppose that $\e$ and $\a$ are positive constants such that 
\begin{equation*}
16d^2\a+(11\eta+e^{-r})^{1/8}(4/\a)^{4d^2}(800d^2/\rho)^d\leq2\e.
\end{equation*}
Then if $\sp{Q,Q'}\geq 2\e$, it follows that $\|Q'\|_{U^2}\leq(12\a)^{1/8}$.
\end{lemma}

\begin{proof}
The basic idea is that if $\|Q'\|_{U^2}$ is not small, then by Theorem \ref{u2u2*dichotomy} we can approximate it by a quadratic average with smallish $U^2$ dual norm, which shows that $Q'$ cannot after all correlate with $Q$, which has small $U^{2}$ norm.

More precisely, Theorem \ref{highrankQu2} tells us that $\|Q\|_{U^2}\leq(11\e+e^{-r})^{1/8}$. Let $\a>0$ and suppose that $\|Q'\|_{U^2}>(12\a)^{1/8}$. Then Theorem \ref{u2u2*dichotomy} gives us a function $Q''$ such that $\|Q'-Q''\|_\infty<16d^2\a$ and $\|Q''\|_{U^2}^*<(4/\a)^{4d^2}(800d^2/\rho)^d$. It follows that
\begin{equation*}
\sp{Q,Q'}<\|Q\|_1\|Q'-Q''\|_\infty+\|Q\|_{U^2}\|Q''\|_{U^2}^*\leq16d^2\a+(11\e+e^{-r})^{1/8}(4/\a)^{4d^2}(800d^2/\rho)^d,
\end{equation*}
which we are assuming to be at most $2\e$. This proves the lemma.
\end{proof}

It turns out that we need to look some distance ahead in order to determine with respect to what sort of substructure we would like our quadratic averages to have large rank. So for the time being our choice of substructure will look rather arbitrary. For further justification the reader may wish to consult the proof of Proposition \ref{struccomput} a few pages further along.


It may help if we point out that the unpleasant bound for $c$ in the theorem below is exponential in $R_{0}$ and doubly exponential in $\d$. This, rather than the precise form of the bound, is what mainly matters to us.

\begin{theorem}\label{HRdecomp}
Let $C_0=2^{24}$, let $\d>0$ and let $C=4(2/\d^2)^{C_0}$. Let $f:\Z_N\ra\C$ be a function such that $\|f\|_2\leq 1$ and let $R_0$ be a positive real number. Let $d=(2/\d)^{C_0}$, $\rho=(\d/2)^{C_0}$, let $\eps>0$ be bounded above by $\delta^{6}/2^{12}C^{5}$, let $T=(8/\e^8)^{4d^2}(2^{20}d^35^d/\e\rho)^dC$ and let $c>0$ be at most
\[e^{-2^{15k} d^{7k} k^{6k}R_{0}}\left(\frac{\delta^{2}\rho^{2}\eps^{2}\Phi}{2^{48k}d^{4}5^{d}C} \right)^{2^{26k}d^{10k}k^{6k}},\]
where 
\[\Phi=\Phi(\delta,\e)=\left(\frac{\delta^{5}\eps \rho}{2^{54}k^{8}d^{4}5^{d}C^{2}T^{2}} \right)^{64k^{3}(m+d^{2}+ 2^{33}k^{6}T^{4}C^{2}/\d^{6})}.\]
Let $f$ be any function such that $\|f\|_{U^{2}}\leq c$. Then $f$ has a decomposition of the form
\begin{equation*}
f(x)=\sm i k Q_i'(x)U_i(x)+g(x)+h(x),
\end{equation*}
where $k\leq 2C/\d^2$ and the $Q_i'$ are quadratic averages on $\Z_N$ with base $(B_{i}',q_{i})$ and of complexity at most $(d,\e\rho/800d5^d)$, such that $\sum_{i=1}^k\|U_i\|_{U^2}^*\leq T$, $\sm i k\|U_i\|_\infty\leq 2C$, $\|g\|_1\leq 10\d$ and $\|h\|_{U^3}\leq 2\d$. Moreover, each quadratic average $Q_{i}'$  is $(\eps,m)$-special for $m\leq(5/\rho)^{d}$, and there exists a proper generalized arithmetic progression $P$ inside $B'=\bigcap_{i=1}^{k}B_{i}'$  of dimension $d' \leq kd$ and density $\gamma'\geq (\e\rho/2^{12}d'd5^{d})^{d'}$, such that each $Q_{i}'$ has rank at least $R_{0}$ with respect to $P$.
\end{theorem}

\begin{proof}
Because $\eps\leq\delta^{6}/2^{12}C^{5}$, it is also at most $(\delta/2)^{5C_{0}}$ and therefore satisfies the hypothesis of Theorem \ref{znquadfouriersmallk}. We deduce that $f$ has a decomposition of the form
\begin{equation*}
f(x)=\sm i k Q_i'(x)U_i(x)+g'(x)+h'(x),
\end{equation*}
with the following properties: $k\leq 2C/\d^2$, the $Q_i'$ are
quadratic averages on $\Z_N$ with base $(B_{i}',q_{i})$ and of complexity at most $(d,\e\rho/800d5^d)$,
$\sum_{i=1}^k\|U_i\|_{U^2}^*\leq T$, $\sm i k\|U_i\|_\infty\leq 2C$, $\|g'\|_1\leq 3\d$ and $\|h'\|_{U^3}\leq\d$. Moreover, each average $Q_{i}'$ is $(\eps,m)$-special for $m\leq (5/\rho)^{d}$.

By a lemma of Ruzsa \cite{Ruzsa:1994gap} (see also \cite{Nathanson:1996fp}) there is a proper generalized arithmetic progression $P\subset B'$ with the properties claimed in the theorem. (The additional factor of $1/4$ in the density of this progression arises from the requirement that $2P-2P \subseteq B_{i}'$, which we need in order to be able to talk about the rank of the quadratic average with respect to $P$.) Let us assume that the quadratic averages are arranged in increasing order of rank with respect to $P$. 

Applying Lemma \ref{rankgap} with $b=2^{13}d^{5}k^{3}$ and 
\[t= 2^{11}d^{3}\log\left( \frac{2^{3(k+15)}d^{4}5^{d}C}{\delta\rho^{2}\eps^{2}\Phi} \right) ,\] 
we obtain positive integers $R \in [R_{0},b^{k}(R_{0}+t)]$ and $s \in \{0,1, \dots, k\}$ such that $Q_{i}'$ has rank at most $R$ when $i\leq s$ and rank at least $bR+t$ when $i>s$. We collect together the low- and high-rank quadratic phases by setting $f_{L}=\sum_{i=1}^{s}Q_{i}'U_{i}$ and $f_{H}=\sum_{i=s+1}^{k}Q_{i}'U_{i}$.

Because $\eps \leq \delta^{6}/2^{12}C^{5}$ and $k\leq 2C/\d^2$, we also have $\eps \leq (\delta/20kC)^{2}$, so it satisfies the hypothesis of Proposition \ref{approxLRsumQU} with $\zeta=\d$. Setting $\log(1/\alpha)=R$ in Proposition \ref{approxLRsumQU} we obtain a Bourgain system $(S'_{\rho'})$ of dimension at most $32k^{3}(m+ 2^{33}k^{6}T^{4}C^{2}/\d^{6})$ such that $\|f_{L}-f_{L}*\sigma\|_{2} \leq \delta$, where $\sigma$ is the characteristic measure of $S'_1$. That proposition also gives us a lower bound for the density $\g$ of $S'_1$ of 
\[\left(\frac{\alpha^{4}\delta\eps\rho}{2^{27}k^{4}d^{4}5^{d}C}\right)^{d^{2}k^{3}}\left(\frac{\delta^{4}}{2^{27}k^{4}CT^{2}}\right)^{64k^{3}(m+ 2^{33}k^{6}T^{4}C^{2}/\d^{6})}\geq e^{-4d^{2}k^{3}R} \cdot \Phi(\delta,\e).\]

Now let us reconsider our original decomposition $f=f_{L}+f_{H}+g'+h'$. We shall convolve this equation with the measure $\sigma$ on both sides. We shall show that all of $f*\sigma$, $f_H*\sigma$, $g'*\sigma$ and $h'*\sigma$ are small and we have already seen that $f_L*\sigma\approx f_L$. From this it will follow that $f_L$ is small enough to be absorbed into the error terms.

Let us deal with the easy parts first. Since $\|\sigma\|_1=1$ and the $L_1$ norm is translation invariant, the triangle inequality implies that $\|g'*\sigma\|_1\leq\|g'\|_1\leq 3\d$. Similarly, $\|h'*\d\|_{U^3}\leq\d$, since the $U^3$ norm is also translation invariant. 

Next, let us estimate $\|f*\sigma\|_1$. The Cauchy-Schwarz inequality (applied to the Fourier transform, though a direct argument is also possible) gives us that $\|f*\sigma\|_1\leq\|f*\sigma\|_{2} \leq \|f\|_{U^{2}} \|\sigma\|_{U^{2}}$. But we are assuming that $\|f\|_{U^{2}}\leq c$ and we know that $\|\sigma\|_{U^{2}}\leq  \|\sigma\|_{\infty}= \gamma^{-1}$. Thus provided that $c \leq \delta \gamma$, we obtain the bound $\|f*\sigma\|_{2} \leq \delta$. This gives us the upper bound that $c$ will be required to satisfy for the theorem to hold.

Our one remaining task is to show that $\|f_H*\sigma\|_1$ is small (when the parameters are appropriately chosen). This is significantly harder, and we shall need to use Lemma \ref{highrankcorru2}. 

Recall first that each quadratic average $Q_{i}'$ that appears in $f_{H}$ has base $(B_{i}',q_{i})$ and rank at least $bR+t$ with respect to the progression $P \subseteq B'$. We also recall from the proof of Theorem \ref{znquadfouriersmallk} that $Q_{i}'U_{i}(x)$ can be written as $\sum_{j \in A_{i}}\lambda_{j}Q_{j}(x)+V_{i}$, where $\|V_{i}\|_{1}\leq (2\e+8d^{2}\e^{8})\sum_{j \in A_{i}}|\lambda_{j}|$. The functions $Q_{j}$ are quadratic averages with base $(B_{j}, q_{j})$ and complexity at most $(d,\rho)$. Let $f_{H}'(x)=\sum_{i>s}\sum_{j \in A_{i}}\lambda_{j}Q_{j}(x)$. We have 
\[\|f_{H}*\sigma\|_{1} \leq \|f_{H}'*\sigma\|_{1}+\|(f_{H}-f_{H}')*\sigma\|_{1}\leq \|f_{H}'*\sigma\|_{1}+\sum_{i>s}\sum_{j \in A_{i}} |\lambda_{j}|(2\e+8d^{2}\e^{8}).\]
The latter term was shown to be at most $\delta$ in the proof of Theorem \ref{znquadfouriersmallk}. It follows that $\|f_{H}*\sigma\|_{1} \leq \|f_{H}'*\sigma\|_{1}+\delta$, and thus it suffices to estimate $\|f_{H}'*\sigma\|_{1}$. In fact, we shall obtain an upper bound for $\|f_H'*\sigma\|_2$.

By the Cauchy-Schwarz inequality as used on $\|f_{L}*\sigma\|_{2}$ earlier we have 
\[\|f_{H}'*\sigma\|_{2}\leq \|f_{H}'\|_{U^{2}}\|\sigma\|_{U^{2}} \leq  \gamma^{-1}\sum_{i>s}\sum_{j \in A_{i}}|\lambda_{j}| \|Q_{j}\|_{U^{2}}\]
with $ \sum_{i>s}\sum_{j \in A_{i}}|\lambda_{j}| \leq C$. In order to prove that $\|f_h'*\sigma\|_1\leq\d$ it will therefore be enough to show that each $Q_{j}$ has $U^{2}$ norm at most $\d\g/C$. To do this, we extract further information from the proof of Theorem \ref{znquadfouriersmallk}. It tells us that there is another quadratic average $Q_{i}''$ with the same base $(B_{i}',q_{i})$ as $Q_{i}'$ such that $\sp{Q_{i}'',Q_{j}} \geq 2\eps$. Since $Q_{i}''$ has the same high rank as $Q_{i}'$ and correlates with $Q_{j}$, we are in a position to apply Lemma \ref{highrankcorru2}.

To do this, we set $Q=Q_i''$, $B=B_i'$ and $q=q_i$. We shall let 
\[\eta= \left(\frac{\e}{4}\right)^{8} \left(\frac{\rho}{800d^{2}}\right)^{8d} \left(\frac{\delta\gamma}{4C}\right)^{2^{8}d^{2}} \]
and we shall take $B'$ to be a Bohr subset $B_i''$ of $B_i'$ such that $B_i''\prec_\eta B_i'$. We then take $Q'$ to be $Q_j$, remarking that $Q_j$ has base $(B_j,q_j)$ for some $B_j$ of complexity at most $(d,\rho)$. 

We shall take the set $P$ in Lemma \ref{highrankcorru2} to be the set $P\cap B_i''$ here. We now need a lower bound for the rank of $Q=Q_i''$ with respect to $P\cap B_i''$, or equivalently an upper bound for the quantity $\a_{P\cap B_i''}(Q)$. Lemma \ref{rankdown2} tells us that $\alpha_{P\cap B_{i}''}(Q) \leq \beta^{-4} \;\alpha_{P} \leq \beta^{-4}e^{-(bR+t)}$, where $\beta=|P\cap B_{i}''|/|P|$ is the relative density of $P\cap B_{i}''$ in $P$. By Lemma \ref{intBS} we find that $\beta$ is at least $2^{-3(kd+3d)}$ times the density of $B_{i}''$, so $\beta \geq (\eta\e\rho/2^{3(k+10)}d^{2}5^{d})^{d}$. Therefore, we can take $e^{-r}$ in Lemma \ref{highrankcorru2} to be $\beta^{-4}e^{-(bR+t)}$ with this value of $\b$. It can now be checked (the checking, though painful, is routine) that if we take $\a=(\d\g/C)^8/12$, then the conditions for Lemma \ref{highrankcorru2} are satisfied. Therefore, by that lemma, $\|Q_j\|_{U^2}\leq\d\g/C$.

This completes the proof that $\|f_{H}*\sigma\|_{1}\leq 2\delta$. We have therefore demonstrated that it is possible to write $f_{L}$ as a sum $g''+h''$ with $\|g''\|_{1}\leq 7\delta$ and $\|h''\|_{U^{3}}\leq \delta$. It follows that $f$ has a decomposition $f=f_{H}+g+h$ with $\|g\|_{1} \leq 10 \delta$ and $\|h\|_{U^{3}}\leq 2 \delta$ as claimed. Finally, we remark that the rank $R$ was at most $b^{k}(R_{0}+t)$, a condition which we insert into our bound for the uniformity parameter $c$ to obtain the theorem as stated.
\end{proof}

\section{Some facts about ranks of quadratic and bilinear functions on Bohr sets}

In $\F_{p}^{n}$, it was more or less self-evident that the rank of the sum of two quadratic forms was bounded above by the sum of the individual ranks. Such subadditivity, even in approximate form, is no longer a trivial statement for forms of higher degree such as those in \cite{Gowers:2009lfuII}, and, as it turns out, for the locally defined quadratic forms that we are dealing with in this paper. Here we shall use regular sets from Bourgain systems to adapt the analytic proof of subadditivity for $\F_{p}^{n}$ given in \cite{Gowers:2009lfuII} to $\Z_{N}$. The reader may wish to consult the finite-fields argument in that paper before embarking on this section. 

The following standard identity is the key ingredient in the proof of subadditivity.


\begin{lemma} \label{usualidentity}
Let $B \subseteq \Z_{N}$ and let $\b:B^2\ra\Z_{N}$ be a bilinear function and let $f(x,y)=\omega^{\b(x,y)}$. Then 
\begin{equation*}
f(a-a',b-b')=f(x+a,y+b)\ol{f(x+a,y+b')f(x+a',y+b)}f(x+a',y+b')
\end{equation*}
provided that all of $a-a',b-b',x+a,x+a',y+b,y+b'$ lie in $B$.
\end{lemma}

\begin{proof}
This follows immediately from the identity
\begin{equation*}
\b(a-a',b-b')=\b(x+a,y+b)-\b(x+a',y+b)-\b(x+a,y+b')+\b(x+a',y+b'),
\end{equation*}
which can easily be checked by hand.
\end{proof}


\begin{lemma}\label{convreg}
Let $B,B'$ be two sets from a Bourgain system and suppose that $B' \prec_{\eps }B$. Write $\pi$ for the characteristic measure of $B$. Then for every $s \in B'$ and every function $j: \Z_{N}\ra \C$ with $\|j\|_{\infty}\leq 1$, we have
\[\E_{u \in \Z_{N}} \pi*\pi(u)j(u+s) \approx_{\eps}\E_{u \in \Z_{N}} \pi*\pi(u)j(u).\]
In particular, it follows immediately that for any $A \subseteq B'$,
\[\E_{u \in \Z_{N}} \pi*\pi(u)j(u) \approx_{\eps}\E_{u \in \Z_{N}} \E_{s \in A} \pi*\pi(u)j(u+s).\]
\end{lemma}

\begin{proof} We estimate the difference between the left- and right-hand side above by expanding out the convolution and using the triangle inequality.
\begin{eqnarray*}
|\E_{u \in \Z_{N}} \pi*\pi(u) j(u+s)- \pi*\pi(u)j(u)|&=&|\E_{u,z \in \Z_{N}} \pi(z)\pi(u-z)(j(u+s)-j(u))|\\
&\leq &\E_{z \in \Z_{N}}\pi(z)|\E_{u\in \Z_{N}} \pi(u) (j_{z}(u+s)-j_{z}(u))|\\
&=&\E_{z \in \Z_{N}}\pi(z)|\E_{u\in B} j_{z}(u+s)-j_{z}(u)|,
\end{eqnarray*}
where $j_{z}(u)=j(u+z)$ for all $u$. The inner expectation is at most $\eps$ for every $s \in B'$ by Lemma \ref{bourgainaveraging}.
\end{proof}

We now apply Lemma \ref{convreg} to derive an inequality reminiscent of the usual lemmas that say that a function behaves quasirandomly if its $U^{2}$ norm is small. However, our inequality concerns a ``local" version of the $U^2$ norm. Given two sets $B'\prec_\e B$ from a Bourgain system and a function $h:\Z_N\ra\C$, we shall define $\|h\|_{U^2(B+B,B')}$ by the formula
\begin{align*}
\|h\|_{U^2(B+B,B')}^4=&\E_{x,y} \pi*\pi(x)\pi*\pi(y)\\&\E_{a,a',b,b' \in B'}h(x+a,y+b)\ol{h(x+a',y+b)h(x+a,y+b')}h(x+a',y+b'),
\end{align*}
where $\pi$ is the characteristic measure of $B$.


\begin{lemma} \label{u2lowerbound}
Let $\e>0$, let $B' \prec_{\eps }B$ be a regular Bourgain pair and write $\pi$ for the characteristic measure of $B$. Then for any function $h:(\Z_{N})^2\ra\C$ with $\|h\|_{\infty}\leq 1$, we have the estimate
\[|\E_{x,y \in \Z_{N}} \pi*\pi(x)\pi*\pi(y)h(x,y)|^{4} \leq \|h\|_{U^2(B+B,P)}^4 +6\eps. \] 
\end{lemma}

\begin{proof}
The Cauchy-Schwarz inequality implies that 
\begin{equation*}
|\E_{x,y \in \Z_{N}} \pi*\pi(x)\pi*\pi(y)h(x,y)|^{4}\leq |\E_{x \in \Z_{N}}\pi*\pi(x)|\E_{y \in \Z_{N}}\pi*\pi(y)h(x,y)|^2|^{2}.
\end{equation*}
Lemma \ref{convreg} tells us that 
\begin{equation*}
\E_{y \in \Z_{N}}\pi*\pi(y)h(x,y)\approx_\e\E_{y \in \Z_{N}}\E_{b \in B'}\pi*\pi(y)h(x,y+b)
\end{equation*}
for every $x$, from which it follows that 
\begin{equation*}
|\E_{y \in \Z_{N}}\pi*\pi(y)h(x,y)|^2\approx_{2\e}|\E_{y \in \Z_{N}}\E_{b \in B'}\pi*\pi(y)h(x,y+b)|^2.
\end{equation*}
From this it follows that
\begin{equation*}
 |\E_{x \in \Z_{N}}\pi*\pi(x)|\E_{y \in \Z_{N}}\pi*\pi(y)h(x,y)|^2|^{2}
 \leq|\E_{x\in\Z_N}\pi*\pi(x)|\E_{y \in \Z_{N}}\E_{b \in B'}\pi*\pi(y)h(x,y+b)|^2|^2+4\e.
 \end{equation*}
(For these last two approximations we have used the fact that if $a\approx_\e b$ and
$a$ and $b$ both have modulus at most 1, then $a^2\approx_{2\e}b^2$, which follows
from the fact that $a^2-b^2=(a+b)(a-b)$.) By the Cauchy-Schwarz inequality,
\begin{align*}
&|\E_{x\in\Z_N}\pi*\pi(x)|\E_{y \in \Z_{N}}\E_{b \in B'}\pi*\pi(y)h(x,y+b)|^2|^2\\
&\leq|\E_{x \in \Z_{N}}\pi*\pi(x)\E_{y \in \Z_{N}}\pi*\pi(y)|\E_{b \in B'}h(x,y+b)|^2|^2\\
&= |\E_{y \in \Z_{N}}\pi*\pi(y)\E_{b,b' \in B'}\E_{x \in \Z_{N}}\pi*\pi(x)h(x,y+b)\ol{h(x,y+b')}|^2.\\
\end{align*}
Applying Lemma \ref{convreg} in a similar way a second time, we see that this is
at most
\begin{align*}
&|\E_{y \in \Z_{N}}\pi*\pi(y)\E_{b,b' \in B'}\E_{x \in \Z_{N}}\E_{a\in B'}\pi*\pi(x)h(x+a,y+b)\ol{h(x+a,y+b')}|^2+2\eps\\
&\leq  \E_{x,y \in \Z_{N}}\pi*\pi(x)\pi*\pi(y)\E_{b,b' \in B'}|\E_{a\in B'}h(x+a,y+b)\ol{h(x+a,y+b')}|^2+2\eps,
\end{align*}
which equals $\|h\|_{U^2(B+B,B')}^4 + 2\eps$. This proves the lemma.
\end{proof}

Finally, we need to exploit regularity once more to be able to shift our variables at a certain point in the proof. We isolate the lemma, which is very similar to Lemma \ref{convreg}, in order to keep the proof of the main result tidy.


\begin{lemma}\label{shift} Let $B$ and $B'$ be sets from a Bourgain system with $B' \prec_{\eps }B$, and write $\pi$ for the characteristic measure of $B$. Write $\sigma(x)=\pi*\pi(x)$, $\rho$ for the density of $B$ and let $j: \Z_{N} \maps \C$ be an arbitrary function with $\|j\|_{\infty} \leq 1$. Then for any $a \in B'$,
\[\E_{x\in \Z_{N}} \sigma(x+a)^2j(x)\approx_{2\eps /\rho}\E_{x\in \Z_{N}} \sigma(x)^2j(x).\]
\end{lemma}

\begin{proof} As usual, we shall attempt to bound the difference between the two sides in absolute value.
\begin{eqnarray*}
|\E_{x}( \sigma(x+a)^2-\sigma(x)^2)j(x)|&=&|\E_{x}(\sigma(x+a)+\sigma(x))( \sigma(x+a)-\sigma(x))j(x)|\\
&=&|\E_{x}(\sigma(x+a)+\sigma(x))j(x)\E_{v}\pi(v)(\pi(x+a-v)-\pi(x-v))|\\
&\leq& 2\rho^{-1} \;\E_{x}|\E_{v}\pi(v)(\pi(x+a-v)-\pi(x-v))|\\
&\leq& 2\rho^{-1}\; \E_{v}\pi(v)\E_{x}|\pi(x+a-v)-\pi(x-v)|
\end{eqnarray*}
The expression $|\pi(x+a-v)-\pi(x-v)|$ is non-zero if and only if $x \in (v+B) \bigtriangleup (v-a+B)$, which by regularity assumptions is the case for at most $\eps|B|$ values of $x$. The non-zero value taken is $\rho^{-1}$, and we conclude that $\E_{x}|\pi(x+a-v)-\pi(x-v)| \leq \eps$. The lemma follows.
\end{proof}

We are now fully prepared to prove subadditivity. We remind the reader that $\a_{P}(\b)=\E_{a,a',b,b' \in P}\omega^{\b(a-a',b-b')}$, and $r_{P}(\beta)=\log \alpha_{P}^{-1}$ for any bilinear form $\beta$ defined on a set that contains $P-P$. 


\begin{lemma}\label{subaddrank}
Let $\b_1$ and $\b_2$ be bilinear forms defined on a Bohr set $B$, and let $(B_\rho)$ be a Bourgain system of dimension $d$ such that $B_1$ has density $\g<1/2$ and $2B_1-2B_1 \subseteq B$. Let $\eps>0$. Then
\[(\alpha_{B_1}(\b_{1})\alpha_{B_1}(\b_{2}))^{4}\leq\g^{-6}(800d/\eps)^{4d}\alpha_{B_1}(\b_{1}+\b_{2})+9\eps\g^{-7}.\]
\end{lemma}

\begin{proof}
Let $B' \prec_{\eps} B_1$, write $\g'$ for the density of $B'$, and note that $\g'\geq(\e/800d)^d\g$. We shall begin to prove the subadditivity statement by considering the expression
\begin{align*}
\alpha_{B_{1}}(\b_{1})\alpha_{B_{1}}(\b_{2})&=\E_{x,x',y,y' \in B_{1}} f(x-x',y-y')\E_{u,u',v,v' \in B_{1}}g(u-u',v-v')\\
&=\E_{x,y \in \Z_{N}} \pi*\pi(x)\pi*\pi(y)f(x,y)\E_{u,v \in \Z_{N}} \pi*\pi(u)\pi*\pi(v)g(u,v),
\end{align*}
where $\pi$ is the characteristic measure of $B_{1}$. Shifting two of the variables, we obtain
\begin{align*}\label{zero}
&\E_{x,y \in \Z_{N}} \pi*\pi(x)\pi*\pi(y)f(x,y)\E_{u,v \in \Z_{N}} \pi*\pi(x+u)\pi*\pi(y+v)g(x+u,y+v).
\end{align*}
Writing $\sigma=\pi*\pi$, we apply H\"older's inequality (or the Cauchy-Schwarz inequality twice) to show that
\begin{align*}
(\alpha_{B_{1}}(\b_{1})\alpha_{B_{1}}(\b_{2}))^{4}&\leq \E_{u,v} |\E_{x,y} \sigma(x)\sigma(y)f(x,y)\sigma(x+u)\sigma(y+v)g(x+u,y+v)|^{4}\\
&=\E_{u,v} |\E_{x,y} \sigma(x)\sigma(y)h_{u,v}(x,y)|^{4}
\end{align*}
where we have set $h_{u,v}(x,y)=f(x,y)g(x+u,y+v)\sigma(x+u)\sigma(y+v)$. For every fixed value of $u$ and $v$, we shall apply Lemma \ref{u2lowerbound}. From this, we deduce that 
\begin{align*}
(\alpha_{B}(\b_{1})\alpha_{B}(\b_{2}))^{4}\leq\E_{u,v }\E_{x,y }\sigma(x)&\sigma(y)\E_{a,a',b,b' \in B'}h_{u,v}(x+a,y+b)\\
&\ol{h_{u,v}(x+a',y+b)h_{u,v}(x+a,y+b')}h_{u,v}(x+a',y+b') +12\e.
\end{align*}
(We have omitted the condition $\e<1/3$ from the statement of this lemma since if $\e\geq 1/3$ then the lemma holds trivially.) Next, we expand out $h_{u,v}$, which replaces the right-hand side by
\begin{align*}
\E_{u,v} \E_{x,y} &\sigma(x)\sigma(y)\E_{a,a',b,b' \in B'}f(x+a,y+b)\ol{f(x+a',y+b)f(x+a,y+b')}f(x+a',y+b')\\
g(x&+u+a,y+v+b)\ol{g(x+u+a',y+v+b)g(x+u+a,y+v+b')}g(x+u+a',y+v+b')\\
& \sigma(x+u+a)^2\sigma(x+u+a')^2\sigma(y+v+b)^2\sigma(y+v+b')^2+12\eps.
\end{align*}
Setting $x'=x+u, y'=y+v$, we can rewrite this expression as 
\begin{align*}
\E_{x,y,x',y'} &\sigma(x)\sigma(y)\E_{a,a',b,b' \in B'}f(x+a,y+b)\ol{f(x+a',y+b)f(x+a,y+b')}f(x+a',y+b')\\
g(x'&+a,y'+b)\ol{g(x'+a',y'+b)g(x'+a,y'+b')}g(x'+a',y'+b')\\
& \sigma(x'+a)^2\sigma(x'+a')^2\sigma(y'+b)^2\sigma(y'+b')^2+12\eps.
\end{align*}
Lemma \ref{usualidentity} tells us that this expression is equal to
\[\E_{x,y,x',y'} \sigma(x)\sigma(y)\E_{a,a',b,b' \in B'}\sigma(x'+a)^2\sigma(x'+a')^2\sigma(y'+b)^2\sigma(y'+b')^2f(a-a',b-b')g(a-a',b-b')+12\eps.\]
Since $\E_x\sigma(x)=1$, this in turn equals
\[\E_{a,a',b,b' \in B'}f(a-a',b-b')g(a-a',b-b')\E_{x',y'}\sigma(x'+a)^2\sigma(x'+a')^2\sigma(y'+b)^2\sigma(y'+b')^2+12\eps.\]
We would like to be able to evaluate the inner expectation independently of the choice of $a,a',b,b'$. We cannot do this exactly, but Lemma \ref{shift} tells us that $\sigma$ is approximately translation invariant, so we can do it if we introduce a small error. For instance, if we apply it to the first occurrence of the function $\sigma^2$ and let $j(x')=\g^6\sigma(x'+a')^2\sigma(y'+b)^2\sigma(y+b')^2$, then $\|j\|_\infty\leq 1$, so we find that 
\begin{equation*}
\E_{x',y'}\sigma(x'+a)^2\sigma(x'+a')^2\sigma(y'+b)^2\sigma(y'+b')^2
\leq\E_{x',y'}\sigma(x')^2\sigma(x'+a')^2\sigma(y'+b)^2\sigma(y'+b')^2+2\g^{-7}\e.
\end{equation*}
Applying the lemma three more times in this way, we find that
\[(\alpha_{B_{1}}(\b_1)\alpha_{B_{1}}(\b_2))^{4}\leq \E_{a,a',b,b' \in B'}f(a-a',b-b')g(a-a',b-b')\E_{x'}\sigma^{4}(x')\E_{y'}\sigma^{4}(y')+8\g^{-7}\eps +12\e.\]
But since $\E_{a,a',b,b' \in B'}f(a-a',b-b')g(a-a',b-b')=\alpha_{B'}(\b_1+\b_2)$ and $\E_{x}\sigma^{4}(x)\leq \g^{-3} $, we have shown that 
\[(\alpha_{B_{1}}(\b_{1})\alpha_{B_{1}}(\b_{2}))^{4}\leq \g^{-6}\alpha_{B'}(\b_{1}+\b_{2})+8\g^{-7}\eps+12\e.\]
Unfortunately the exponential sum on the right-hand side is taken over $B'$, or we would be done. But we can remedy this situation by applying Lemma \ref{rankdown2}, which implies that $\alpha_{B'} \leq (\g/\g')^{4} \alpha_{B_{1}}$. Therefore, 
\[(\alpha_{B_{1}}(\b_{1})\alpha_{B_{1}}(\b_{2}))^{4}\leq \g^{-2}\g'^{-4}\alpha_{B_{1}}(\b_{1}+\b_{2})+8 \g^{-7}\eps+12\e.\]
The result follows from the lower bound for $\g'$ mentioned at the beginning of the proof. 
\end{proof}

We need a slight generalization of Lemma \ref{subaddrank} to be able to sum arbitrarily many bilinear forms. In fact, we shall not use Lemma \ref{subaddrank} as stated to carry out the induction, but rather the main intermediate result in the proof above that related the rank of $\b_{1}+\b_{2}$ with respect to $B'$ to the individual ranks with respect to $B_{1}$.


\begin{lemma}\label{multsubaddrank}
Let $\eps>0$. For $i=1,2, \dots, m$, let $\beta_{i}$ be a bilinear form defined on a set $B$, and let $(B_\rho)$ be a Bourgain system of dimension $d$ such that $B_1$ has density $\g<1/8$ and $2B_1-2B_1 \subseteq B$. Then
\[\prod_{i=1}^{m} \alpha_{B}(\beta_{i}) \leq  \g^{-2m^{2}}(800d/\eps)^{d \log m/m^{2}}\alpha_{B}(\sum_{i=1}^{m}\beta_{i})^{1/m^{2}}+8\g^{-2m^{2}}(\eps^{1/4}/\g^{2})^{1/m^{2}}.\]
\end{lemma}

\begin{proof}
Let us start off by considering the case when $m=2^{s}$. Let $B_{s+1}\prec_{\eps}B_{s}\prec_{\eps}... \prec_{\eps}B_{2}\prec_{\eps}B_{1}$ be a sequence of sets from the Bourgain system $(B_\rho)$. (Thus, the indices do not indicate values of $\rho$.) We shall prove that
\[\prod_{i=1}^{2^{s}} \alpha_{B_{1}}(\beta_{i})  \leq A^{4^{s}}\alpha_{B_{s+1}}(\sum_{i=1}^{2^{s}}\beta_{i})^{1/4^{s}}+4a^{1/4^{s}}A^{4^{s}},\]
where $A=\g^{-3/2}$ and $a=2\g^{-2}\eps^{1/4}$, and proceed by induction on $s$. The case where $s=1$ is guaranteed by the proof of Lemma \ref{subaddrank}. Indeed, before we switched from $B'$ back to $B_{1}$, the inequality we had implied that
\[\alpha_{B_{1}}(\b_{1})\alpha_{B_{!}}(\b_{2})\leq A\alpha_{B'}(\b_{1}+\b_{2})^{1/4}+a,\]
on the assumption that $B' \prec_{\eps} B_{1}$. If we take $B'=B_{2}$, then this is in fact stronger than the case $s=1$ of this lemma. 

Suppose now that the statement is true for $s$, and consider
\begin{eqnarray*}
\prod_{i=1}^{2^{s+1}} \alpha_{B_{1}}(\beta_{i}) & \leq& (A^{4^{s}}\alpha_{B_{s+1}}(\sum_{i=1}^{2^{s}}\beta_{i})^{1/4^{s}}+4a^{1/4^{s}}A^{4^{s}})(A^{4^{s}}\alpha_{B_{s+1}}(\sum_{i=2^{s}+1}^{2^{s+1}}\beta_{i})^{1/4^{s}}+4a^{1/4^{s}}A^{4^{s}})\\
& \leq& A^{2\cdot4^{s}}(\alpha_{B_{s+1}}(\sum_{i=1}^{2^{s}}\beta_{i})\alpha_{B_{s+1}}(\sum_{i=2^{s}+1}^{2^{s+1}}\beta_{i}))^{1/4^{s}}+8a^{1/4^{s}}A^{2\cdot 4^{s}}+16a^{2/4^{s}}A^{2\cdot 4^{s}},
\end{eqnarray*}
from which it follows by the strengthened version of the $s=1$ case noted above that
\begin{eqnarray*}
\prod_{i=1}^{2^{s+1}} \alpha_{B_{1}}(\beta_{i}) & \leq& A^{2\cdot4^{s}}(A \alpha_{B_{s+2}}(\sum_{i=1}^{2^{s+1}}\beta_{i})^{1/4}+a)^{1/4^{s}}+8a^{1/4^{s}}A^{2\cdot 4^{s}}+16a^{2/4^{s}}A^{2\cdot 4^{s}}\\
& \leq& A^{2\cdot4^{s}+1/4^{s}} \alpha_{B_{s+2}}(\sum_{i=1}^{2^{s+1}}\beta_{i})^{1/4^{s+1}}+A^{2\cdot4^{s}}a^{1/4^{s}}+8a^{1/4^{s}}A^{2\cdot 4^{s}}+16a^{2/4^{s}}A^{2\cdot 4^{s}}.
\end{eqnarray*}
It is easily checked that this expression is bounded above by 
\[A^{4^{s+1}}\alpha_{B_{s+2}}(\sum_{i=1}^{2^{s+1}}\beta_{i})^{1/4^{s+1}}+4a^{1/4^{s+1}}A^{4^{s+1}}\]
as claimed, provided that $\g<1/8$. This concludes the inductive step.
To complete the proof, we apply Lemma \ref{rankdown2} to obtain a statement about the rank with respect to $B_1$. It tells us that 
\[\prod_{i=1}^{2^{s}} \alpha_{B_{1}}(\beta_{i})  \leq A^{4^{s}}\alpha_{B_{s+1}}(\sum_{i=1}^{2^{s}}\beta_{i})^{1/4^{s}}+4a^{1/4^{s}}A^{4^{s}} \leq  A^{4^{s}}(|B_{1}|/|B_{s+1}|)^{1/4^{s}}\alpha_{B_{1}}(\sum_{i=1}^{2^{s}}\beta_{i})^{1/4^{s}}+4a^{1/4^{s}}A^{4^{s}} ,\]
with $|B_{s+1}| \geq (\eps/800d)^{d}|B_{s}| \geq  ... \geq (\eps/800d)^{sd}|B_{1}|$. It follows that
\[\prod_{i=1}^{2^{s}} \alpha_{B_{1}}(\beta_{i}) \leq  A^{4^{s}}(800d/\eps)^{sd/4^{s}}\alpha_{B_{1}}(\sum_{i=1}^{2^{s}}\beta_{i})^{1/4^{s}}+4a^{1/4^{s}}A^{4^{s}}.\]

For general $m$, note that we can add in bilinear forms that are identically zero without affecting the argument.
\end{proof}

Next we state and prove a modified version of Lemma 6.3 and Corollary 6.4 from \cite{Gowers:2009lfuI}. This is the first and only time we make use of the assumption that our system of linear forms is square independent.


\begin{lemma}\label{rankaveragecor}
Let $\eps>0$. Suppose that $L_i(\x)=\sum_{u=1}^d c_{iu} x_u$, $i=1,2, \dots, m$, is a square-independent system. Suppose that each of the (not necessarily distinct) bilinear forms $\beta_i$, $i=1,2, \dots, m$ is defined on a Bohr set $B$, and that $(B_\rho)$ is a Bourgain system of dimension $d$ such that $B_1$ has density $\g$ and $2B_1-2B_1\subseteq B$. Then there exists a pair $(u,v) \in [d]^{2}$ such that the bilinear form $\beta_{uv}=\sum_{i=1}^m c_{iu}c_{iv} \beta_i$ satisfies
\[\alpha_{B_{1}}(\beta_{uv}) \leq \g^{-2m}(800d/\eps)^{d \log m/m^{3}}\alpha_{B_{1}}(\beta_{i})^{1/m^{3}}+4\gamma^{-2m}(\eps^{1/4}/\g^{2})^{1/m^{3}}\]
for any $i=1,2, \dots, m$.
\end{lemma}

\begin{proof}
For each $i=1,2, \dots, m$, let $M_i$ be the $(d \times d)$ matrix $(c_{iu}c_{iv})_{u,v}$. Square independence implies that the matrices $M_i$ are linearly independent. It follows that the rank of the $d^2 \times m$ matrix whose $((u,v),i)$ entry is $c_{iu}c_{iv}$ is $m$. The rows of this matrix are the $(d\times d)$ matrices $M_1, \dots, M_m$. The columns are the vectors $C_{uv}=(c_{1u}c_{1v}, c_{2u}c_{2v}, \dots,
c_{mu}c_{mv})$. Since row rank equals column rank, we can find $m$ linearly independent vectors $C_{uv}$. We have just shown that there is a collection of $m$ forms $\eta_{j}=\sum_{i=1}^{m} B_{ij} \beta_{i}$ for an invertible matrix $B$, so we can write $\beta_{i}=B^{-1}_{ij}\eta_{j}$. 
But in this situation Lemma \ref{multsubaddrank} tells us that
\[(\min_{j}\alpha_{B_{1}}(\eta_{j}))^{m}\leq \prod_{j}\alpha_{B_{1}}(\eta_{j}) \leq  A\alpha_{B_{1}}(\beta_{i})^{1/m^{2}}+a,\]
where we have written $A=\g^{-2m^{2}}(800d/\eps)^{d \log m/m^{2}}$ and  $a=8\g^{-2m^{2}}(\eps^{1/4}/\gamma^{2})^{1/m^{2}}$.
Therefore, there exists an index $j$ such that $\alpha_{B_{1}}(\eta_{j}) \leq A^{1/m}\alpha_{B_{1}}(\beta_{i})^{1/m^{3}}+a^{1/m}$. But $\eta_{j}$ equals $\beta_{uv}$ for some pair $(u,v) \in [d]^{2}$.
\end{proof}

We continue by proving a lemma that says that high-rank bilinear phase functions defined on Bohr sets are quasirandom in the following sense: they do not correlate well with products of functions of one variable.


\begin{lemma}\label{bilinearqr2}
Let $\e>0$ and let $B$ and $B'$ be part of a Bourgain system such that $B' \prec_\eps B$. Let $\beta$ be a bilinear form defined on $B^{2}$, and suppose that $P \subseteq B'$. Let $g$ and $h$ be two functions with $\|g\|_\infty$ and $\|h\|_\infty$ at most 1. Then 
\[|\E_{x,y \in B}\omega^{\beta(x,y)}g(x)h(y)|\leq (\alpha_{P}(\beta)+6\e)^{1/4}.\]
\end{lemma}

\begin{proof}
We have
\[|\E_{x,y \in B}\omega^{\beta(x,y)}g(x)h(y)|^4 \leq (\E_{x \in B} |\E_{y \in B} \omega^{\beta(x,y)} h(y)|^2)^2,\]
which, by Lemma \ref{bohraveraging} (ii) and the difference-of-squares argument used in the proof of Lemma \ref{u2lowerbound}, is to within $4\eps$ equal to
\[(\E_{x \in B}|\E_{y \in B^{-}}\E_{z \in P} \omega^{\beta(x,y+z)} h(y+z)|^2)^2.\]
By the Cauchy-Schwarz inequality this is in turn bounded above by
\[(\E_{x \in B} \E_{y \in B^{-}} |\E_{z \in P} \omega^{\beta(x,y+z)} h(y+z)|^2)^2.\]
Expanding out the inner square and applying the triangle inequality, we can bound this above by 
\[(\E_{y \in B^{-}} \E_{z,z' \in P} | \E_{x \in B} \omega^{\beta(x,z-z')}|)^2.\]
The inner sum is to within $\eps$ equal to $\E_{x \in B^{-}, w \in P} \omega^{\beta(x+w,z-z')}$, so our next upper bound is
\[(\E_{y \in B^{-}} \E_{z,z' \in P} | \E_{x \in B^-,w\in P}\omega^{\beta(x+w,z-z')}|)^2+2\e.\]
Another application of Cauchy-Schwarz shows that this is at most
\[\E_{x,y\in B^-} \E_{z,z' \in P}|\E_{w \in P} \omega^{\beta(w,z-z')}|^2+2\e=\E_{w,w',z,z' \in P} \omega^{\beta(w-w',z-z')}+2\e.\]
We recognize the first part of this expression as the definition of $\a_P(\beta)$. This proves the result.
\end{proof}


\section{Computing with linear combinations of high-rank quadratic averages}

We are now in a position to perform the computation over the structured parts of our decompositions, which will be a key ingredient in the proof of the main result of this paper. The next lemma is very straightforward and will help us keep the proof of the subsequent computation as tidy as possible.


\begin{lemma}\label{replace}
For each $j=1,2, \dots, r$, let $g_{j}$ and $g_{j}'$ be arbitrary functions on $\Z_{N}^{s}$. Let $G=\max_{j}\|g_{j}\|_{\infty}$, $G'=\max_{j}\|g_{j}'\|_{\infty}$ and $C=\max\{G,G'\}$. Then
\[ \E_{x \in A} \prod_{j=1}^{r}g_{j}(x)-\E_{x \in A} \prod_{j=1}^{r}g_{j}'(x) \]
is bounded in absolute value by
\begin{enumerate}
\item[(i)] $r C^{r} \max_{j} \|g_{j}-g_{j}'\|_{2}$ if $A=\Z_{N}^{s}$ or
\item[(ii)] $r C^{r} \max_{j} \|g_{j}-g_{j}'\|_{\infty}$ if $A\subseteq\Z_{N}^{s}$.
\end{enumerate}
\end{lemma}

\begin{proof}
In both cases the bound stated follows from the observation that
\[ \prod_{j=1}^{r}g_{j}(x)-\prod_{j=1}^{r}g_{j}'(x)=\sum_{j=1}^{r} \prod_{i<j} g_{i'}(x)(g_{j}(x)-g_{j}'(x))\prod_{i>j}g_{i}(x).\]
When $A=\Z_N$, this actually implies a stronger upper bound of $rC^r\max_j\|g_j-g_j'\|_1$, though we shall only need the $L_2$ bound. For general $A$ the above identity implies that $|\prod_jg_j(x)-\prod_jg_j'(x)|\leq rC^r\max_j\|g_j-g_j'\|_\infty$ for every $x$, so it holds for the average over $x$ over any set $A$. 
\end{proof}

The next result has a long and complicated-looking proof. However, much of the complication is due to the need to keep track of ever more elaborate parameters as we apply the estimates of the preceding sections. So let us first give a qualitative discussion of the argument, to try to indicate what the underlying ideas are.

Recall that our ultimate aim is to obtain a small upper bound for the quantity
\[\left|\E_{x \in (\Z_{N})^{s}} \prod_{i=1}^{r} f(L_{i}(x))\right|\]
when the linear forms $L_i$ are square independent and $\|f\|_{U^2}$ is sufficiently small. The basic idea behind the proof is to decompose $f$ as a sum of the form $\sum_iQ_iU_i+g+h$, where the $Q_i$ are generalized quadratic averages, the $U_i$ are functions with small $U^2$ dual norm, $g$ has small $L_{1}$ norm and $h$ has small $U^{3}$ norm, and then to substitute this expression in for $f$ and do the computations.  

If that were all there was to it, then this paper would be much shorter than it is. However, replacing the $r$ occurrences of $f$ by quadratic averages in the above expression does not give a small result unless those averages have high rank. So a major task was to show, using the hypothesis that $\|f\|_{U^2}$ is small, that the decomposition could be made into high-rank averages. In the previous section, we proved that high-rank averages do indeed lead to small results.

There is one further difficulty, however. The most obvious thing to do at this stage would be to substitute $\sum_iQ_iU_i+g+h$ for each occurrence of $f$, with every $Q_i$ of high rank. This would give us a big collection of terms to deal with. But not all of them would be small. For example, if we take $g$ from every bracket, we obtain a term that has no reason to be small: the fact that $\|g\|_1$ is small is no guarantee that 
\[\left|\E_{x \in (\Z_{N})^{s}} \prod_{i=1}^{r} g(L_{i}(x))\right|\]
is small. 

Instead, we do something slightly different. We first decompose just one copy of $f$, obtaining an expression of the form
\[\E_{x \in (\Z_{N})^{s}} \prod_{i=1}^{r-1} f(L_{i}(x))(\sum_iQ_{i}^{(r)}(L_r(x))U_{i}^{(r)}(L_r(x))+g_r(L_r(x))+h_r(L_r(x))).\]
The effects of the $g_r$ and $h_r$ terms are now small: to deal with the $h_r$ term (which has small $U^3$ norm) we use a lemma of Green and Tao (Lemma \ref{gvnmod} below), and to deal with the $g_r$ term we use the fact that it has small $L_1$ norm and the rest of the product is bounded. Thus, we can approximate the above expression by 
\[\E_{x \in (\Z_{N})^{s}} \prod_{i=1}^{r-1} f(L_{i}(x))f_r(L_r(x)),\]
where $f_r(x)=\sum_iQ_{i}^{(r)}(x)U_{i}^{(r)}(x)$. At this stage, we would like to repeat the process with the $(r-1)$st copy of $f$, but we have a much worse bound for $\|f_r\|_\infty$ than we had for $\|f\|_\infty$, so we have to choose a new decomposition $f=\sum_iQ_{i}^{(r-1)}U_{i}^{(r-1)}+g_{r-1}+h_{r-1}$ in such a way that $\|g_{r-1}\|_1\|f_r\|_\infty$ is small (and not just $\|g_{r-1}\|_1$). And then we continue the process.

This explains why Proposition \ref{struccomput} below concerns $r$ different functions and $r$ different decompositions. Once we have these decompositions, then the above argument is a sketch proof that we can ignore all the error terms and just concentrate on the terms involving high-rank quadratic averages, which is what we do in the proposition. So our problem is now reduced to obtaining an upper bound for the size of terms of the form
\[\E_{x \in (\Z_{N})^{s}} \prod_{i=1}^{r}(Q_iU_i)(L_{i}(x)),\]
when all the $Q_i$ have high rank and the $U_i$ have not too large $U^2$ dual norm, since the expression we are left wishing to estimate is a sum of a bounded number of terms of this form.

The next complication (or rather, apparent complication, since we have the tools to deal with it) is that the $Q_i$ will have different bases and the high ranks will be with respect to different sets. All we really have to do in order to deal with that kind of problem is intersect everything. We know that sets from Bourgain systems have intersections that are not too small, and will use that fact repeatedly.

The rough idea for dealing with a term of the above form is to find a set $D$ such that for every $i$ the functions $U_i(x)$ and $U_i(x+y)$ are close in $L_2$ for every $y\in D$. This we do by finding one such set for each $U_i$ and intersecting those sets. And for that we use the fact that $\|U_i\|_{U^2}^*$ is small for each $i$. Once we have done that, we use Lemma \ref{rankdown2} to argue that our quadratic averages $Q_i$ still have high rank with respect to a generalized arithmetic progression sitting inside $D$. We then split the average we are trying to estimate into an average of averages taken over translates of $D$, which allows us to assume (after allowing for a small error) that the $U_i$ are constant. At this point we are doing a calculation that just involves high-rank quadratic functions on translates of $D$. The sort of expression we want to bound is
\begin{equation*}
\E_{x_{1}\in z_{1}+D} \E_{x_{2}\in z_{2}+D} \dots \E_{x_{s}\in z_{s}+D} \prod_{j=1}^{r} Q_{{j}}(L_{j}(x)).
\end{equation*}
If we expand out terms such as $Q_{{j}}(L_{j}(x))$, then we obtain sums that involve bilinear functions, at which point we use Lemmas \ref{rankaveragecor} and \ref{bilinearqr2} to show that there is always a high-rank bilinear function involved, and therefore that the corresponding terms are small.

Now let us do the argument in detail.


\begin{proposition}\label{struccomput}
Let $\e,\theta>0$. For each $j=1,2, \dots, r$, let $f_{j}=\sum_{i=1}^{k_j}Q_{i}'^{(j)}U_{i}^{(j)}$ be a linear combination of $(\eps, m_{j})$-special quadratic averages with bases $(B_{i}'^{(j)},q_{i}^{(j)})$ on $\Z_{N}$, each of complexity at most $(d_{j},\eps_{j}\rho_{j}/800d_{j}5^{d_{j}})$. Suppose further that each $Q_{i}'^{(j)}$ is of rank $R_{j}$ with respect to some generalized arithmetic progression $P^{(j)} \subseteq B'^{(j)}=\bigcap_{i=1}^{k_{j}}B_{i}'^{(j)}$ of dimension $d'_{j}\leq k_{j}d_{j}$ and density $\gamma'_{j}$, and that $\sum_{i=1}^{k_{j}}\|U_{i}^{(j)}\|_{\infty} \leq 2C_{j}$ and $\sum_{i=1}^{k_{j}}\|U_{i}^{(j)}\|_{U^{2}}^{*} \leq T_{j}$. 

Set $C=\max_{j} C_{j}$, $T=\max_{j}T_{j}$, $R=\min_{j} R_{j}$, $d=\max_{j}d_{j}$, $k=\max_{j}k_{j}$, $\gamma'=\min_{j}\gamma'_{j}$ and $\rho=\min_{j}\rho_{j}$. Finally, suppose that $r(2kC)^{r}\eps \leq \theta$.

Let $\seq L r$ be a square independent system of $r$ forms in $s$ variables, and set $M= \max_{j} \sum_{u=1}^{s} |c_{ju}|$. Then
\[\left|\E_{x \in (\Z_{N})^{s}} \prod_{j=1}^{r} f_{j}(L_{j}(x))\right| \leq 5\theta+\chi e^{-R/4r^{3}} ,\]
where 
\[\chi= \chi(\eps,\theta)= \left(\frac{2^{239}r^{59}d^{6}5^{2d}(3kC)^{50r}T^{24}M^{4}}{\gamma'\e^{2}\rho^{2}\theta^{50}} \right)^{2^{53}r^{22}k^{2}d^{4}(2kC)^{12r}T^{8}/\theta^{12}}.\]
\end{proposition}

\begin{proof}
We can split the expectation into individual terms of the form 
\begin{equation}\label{first}
\E_{x \in (\Z_{N})^{s}} \prod_{j=1}^{r} (Q_{i_{j}}'^{(j)}U_{i_{j}}^{(j)})(L_{j}(x)) 
\end{equation}
where each sequence $(i_{1}, \dots, i_{r})$ belongs to $[k_{1}]\times \dots \times [k_{r}]$. Let us fix such a sequence, and for ease of notation let us write $Q_{j}'U_{j}$ instead of $Q_{i_j}^{(j)}U_{i_j}^{(j)}$. We shall obtain a bound for (\ref{first}) and then multiply it by $\prod_{j=1}^{r}k_{j} \leq k^{r}$ to obtain a bound for $\left|\E_{x \in (\Z_{N})^{s}} \prod_{j=1}^{r} f_{j}(L_{j}(x))\right|$.

Since $\|U_j\|_{U^2}^*\leq\sum_{i=1}^{k_{j}}\|U_{i}^{(j)}\|_{U^{2}}^{*} \leq T_{j}\leq T$ and $\|U_j\|_\infty\leq\sum_{i=1}^{k_{j}}\|U_{i}^{(j)}\|_\infty\leq 2 C_j\leq 2 C$, Lemma \ref{bogolyubov} gives us, for each $j=1,2, \dots, r$ and any $\xi>0$, a Bohr set $E_{j}$ of complexity at most $((2C/\xi)^{2},\xi)$ such that 
\[\E_{x}| U_{j}(x+y)-U_{j}(x)|^{2} \leq 4\xi^{2}C^{2}+4T^{4/3}\xi^{2/3}\]
for each $y \in E_{j}$. Therefore, for each subset $E \subseteq E_{j}$,
\[\|U_{j}-U_{j}*\mu_{E}\|_{2}^{2} \leq \E_{y \in E}\E_{x}| U_{j}(x+y)-U_{j}(x)|^{2} \leq 4\xi^{2}C^{2}+4T^{4/3}\xi^{2/3},\]
where $\mu_E$ is the characteristic measure of $E$. In particular, if we set $\xi= (\theta/r(2kC)^{r})^{3}/2^{6}T^{2}$ (assuming, as usual, that $T$ is much larger than $C$) and $E=E_{1}\cap \dots \cap E_{r}$, then it is readily checked that
\[\|U_{j}-U_{j}*\mu_{E}\|_{2} \leq \theta/r(2kC)^{r} \]
for all $j=1, 2, \dots,r$. Using Lemma \ref{replace} (i), we can therefore replace the average (\ref{first}) by the expression
\begin{equation}\label{second}
\E_{x \in (\Z_{N})^{s}} \prod_{j=1}^{r} (Q_{{j}}'(U_{{j}}*\mu_{E}))(L_{j}(x))
\end{equation}
at the cost of an error of at most $\theta/k^{r}$.

Now $E$ is a Bohr set $B(K,\xi)$ of dimension $d_{E}\leq r(2C/\xi)^{2}\leq 2^{14}r^{7}(2kC)^{6r}T^{4}/\theta^{6}$ and density $\gamma_{E}\geq \xi^{d_E}\geq(\theta^3/2^{6}r^{3}(2kC)^{3r}T^{2})^{2^{14}r^{7}(2kC)^{6r}T^4/\theta^{6}}$. Moreover, $U_{j}*\mu_{E}$ is roughly constant on translates of central subsets $E'$. More precisely, in order for $U_{j}*\mu_{E}$ to be constant to within $\theta/(r(3kC)^{r})$ on translates of $E'=B(K,\xi')$, it is enough if $E' \prec_{\theta/(r(3kC)^{r})} E$, by Lemma \ref{bohraveraging}. Let us note for the record that in this case the dimension of $E'$ is $d_{E'}=d_{E}$ and the density is $\gamma_{E'}\geq (\theta/r(3kC)^{r}800 d_{E})^{d_{E}} \gamma_{E}$, which is bounded below by $(\theta^{10}/2^{30}r^{11}(3kC)^{10r}T^{6})^{2^{14}r^{7}(2kC)^{6r}T^4/\theta^{6}}$.

Suppose that the linear form $L_{i}(x)$ is given by the formula $\sum_{u=1}^{s}c_{iu}x_{u}$. Let $E''=B(K,\xi'/M)$, where $M=\max_{j} \sum_{u=1}^{s}|c_{ju}|$. As a result, $E''$ has dimension $d_{E''}=d_{E}$ and density $\gamma_{E''} \geq M^{-d_{E'}}\gamma_{E'}$, which is at least $(\theta^{10}/2^{30}r^{11}(3kC)^{10r}T^{6}M)^{2^{14}r^{7}(2kC)^{6r}T^4/\theta^{6}}$. The reason for passing to this smaller Bohr set $E''$ is so that it will have the following property: if $x_u\in E''$ for every $u$, then $\sum_{u=1}^sc_{iu}x_u\in E'$.

Let $B'=B_{1}' \cap \dots \cap B_{r}'$. Then $B'$ is a Bohr set of dimension $d_{B'}\leq rd$ and density $\gamma_{B'}\geq (\e\rho/800d5^d)^{rd}$. Let $B''$ be a narrowing of $B'$ by the same factor $1/M$, so $B''$ is a Bohr set of dimension $d_{B''}=d_{B'}$ and density $\gamma_{B''} \geq (\e\rho/800d5^dM)^{rd}$. Finally, set $D=E'' \cap B''$, which is like a Bohr set but with ``different widths in different directions". Rather than go into the details of this, we merely observe that if $E''=B(K,\xi'')$ and $B''=B(L,\tau'')$, then we can define a Bourgain system $(D_\mu)$ by setting $D_\mu$ to be $B(K,\mu\xi'')\cap B(L,\mu\tau'')$. By Lemma \ref{intBS} and the remark following Lemma \ref{bourgainaveraging} (which says that the dimension of a Bohr set $B(K,\rho)$ considered as part of a Bourgain system is at most $3|K|$), this is a Bourgain system of dimension $d_D\leq 12(d_{E''}+d_{B''})\leq 2^{18}r^{7}(2kC)^{6r}T^{4}/\theta^{6} +2^{4}rd$ such that $D=D_1$ has density $\gamma_{D}\geq 2^{-9(d_{E''}+d_{B''})}\gamma_{E''}\gamma_{B''} \geq (\e\rho/2^{19}d5^{d}M)^{rd}(\theta^{10}/2^{39}r^{11}(3kC)^{10r}T^{6}M)^{2^{14}r^{7}(2kC)^{6r}T^{4}/\theta^{6}}$ by Lemma \ref{intBS}.

We shall cover $\Z_{N}$ with translates of $D$, and compute the expectation
\begin{equation}\label{third}
\E_{x_{1}\in z_{1}+D} \E_{x_{2}\in z_{2}+D} \dots \E_{x_{s}\in z_{s}+D} \prod_{j=1}^{r} (Q_{{j}}'(U_{{j}}*\mu_{E}))(L_{j}(x)),
\end{equation}
for some fixed choice of $z_{1}, \dots, z_{s} \in \Z_{N}$. Now if each $x_{i}$ is confined to a translate $z_{i}+D$, then $L_{j}(x)$ is contained in some particular translate $y_{j}+E'\cap B'$, by our choice of $E''$ and $B''$. On this translate, $U_{j}*\mu_{E}$ is constant to within $\theta/(r(3kC)^{r})$. More precisely, we can write $U_{j}*\mu_{E}(x)=\lambda_{y_{j}}+\eps_{j}(x)$, where $\|\eps_{j}\|_{\infty} \leq \theta/(r(3kC)^{r})$ for all $j=1,2, \dots, r$. Taking into account the fact that $\sum_{i=1}^{k_{j}}\|U_{i}^{(j)}\|_{\infty} \leq 2C_{j}$, we immediately note that $|\lambda_{y_{j}}| \leq 3C_{j}$ for any $j=1,2, \dots, r$. It follows from Lemma \ref{replace} (ii) that at the cost of an error of at most $\theta/k^{r}$, we can focus on evaluating
\begin{equation}\label{fourth}
\left(\prod_{j=1}^{r}\lambda_{y_{j}}\right)\E_{x_{1}\in z_{1}+D} \E_{x_{2}\in z_{2}+D} \dots \E_{x_{s}\in z_{s}+D} \prod_{j=1}^{r} Q_{{j}}'(L_{j}(x)),
\end{equation}
instead of the earlier average (\ref{third}). We recall that each $Q_{j}'$ was an $(\eps, m_{j})$-special average with base $B_{j}'$ and rank at least $R_{j}$ with respect to $P^{(j)} \subseteq B_{j}'$. In particular, for each $j=1,2, \dots, r$, since $D\subseteq E'\cap B' \subseteq B_{j}'$, we find that for all but $\eps N$ choices of $y_{j} \in \Z_{N}$, the restriction of $Q_{j}'$ to $y_{j}+D$ is equal to the restriction of $\omega^{q_{j}'}$ to $y_{j}+D$, where $q_{j}'(v)=q_{j}(v-v_{j})$ for one of at most $m_{j}$ fixed values $v_{j} \in \Z_{N}$. Let us say that $(y_1,\dots,y_r)$ is \textit{good} if this is true for every $j\leq r$. 

Observe that as each $z_{1}, \dots, z_{s}$ runs over $\Z_{N}$, so does $L_{j}(z_{1}, \dots, z_{s})$ for each $j=1,2, \dots, r$. Therefore a proportion of at least $(1-\sum_{j}\eps_{j}) \geq (1-\eps r)$ of all choices of $(z_{1}, \dots, z_{s}) \in (\Z_{N})^{s}$ gives rise to a good sequence $(y_{1}, \dots, y_{r})$.
If $(y_{1}, \dots, y_{r})$ is good, then fix a value $v_{j}$ for each $j=1,2, \dots, r$.  Now since the $\eps_{j}$ were required to satisfy $r(2kC)^{r}\eps\leq \theta$, then incurring an error of at most $\theta/k^{r}$, we can restrict our attention to 
\begin{equation}\label{fifth}
\left(\prod_{j=1}^{r}\lambda_{y_{j}}\right)\E_{x_{1}\in z_{1}+D} \E_{x_{2}\in z_{2}+D} \dots \E_{x_{s}\in z_{s}+D} \prod_{j=1}^{r} \omega^{q_{j}(L_{j}(x)-v_{j})}
\end{equation}
for some fixed choice of $v_{1}, \dots, v_{r}$.
Recall that for each $j=1,2, \dots, r$, the linear form $L_j(\x)$ was given by the formula $\sum_{u=1}^s c_{ju}x_u$. Writing $\b_j$ for the bilinear form associated with $q_j$, we have 
\[\sm j r q_j(L_j(\x))=\sum_{u,v=1}^{s}\sm j r c_{ju}c_{jv}\b_j(x_u,x_v).\]
For each $u$ and $v$, let us write $\b_{uv}$ for the bilinear form $\sm j r c_{ju}c_{jv}\b_j$ as before. 

Set $P=P^{(1)}\cap \dots \cap P^{(r)}$, which is part of a Bourgain system of dimension $d_{P}\leq 4r^{2}\sum_{j}d_{j}' \leq 4r^{3}kd$ and has density $\gamma_{P}\geq 2^{-d_{P}}\prod_{j}\gamma_{j}' \geq 2^{-4r^{3}kd} \gamma'^{r}$. We shall now consider the rank of each $q_{j}$ with respect to $P' = P \cap D'$, where $D' \prec_{\theta^{4}/6(3kC)^{4r}} D$. In order to do so, we need to determine the dimension and density of $P'$, which is the main reason we have been carefully keeping track of our parameters since the start of the proof. 

First note that $D'$ is part of a Bourgain system of dimension $d_{D'}=d_{D} \leq 2^{22}r^{8}d(2kC)^{6r}T^{4}/\theta^{6}$ as determined earlier and has density $\gamma_{D'} \geq (\theta^{4}/2^{15}(3kC)^{4r}d_{D})^{d_{D}}\gamma_{D}$, which is bounded below by $(\e\rho \theta^{20}/2^{95}r^{19}d^{2}5^{d}(3kC)^{20r}T^{10}M^{2})^{2^{22}r^{8}d(2kC)^{6r}T^{4}/\theta^{6}}$. Therefore $P'$ is part of a Bourgain system of dimension
\[d_{P'}\leq4(d_{P}+d_{D'}) \leq 2^{26}r^{11}kd^{2}(2kC)^{6r}T^{4}/\theta^{6}\]
and has density 
\[\gamma_{P'}\geq 2^{-3(d_{P}+d_{D'})}\gamma_{P}\gamma_{D'} \geq \gamma'^{r} \left(\frac{\e\rho \theta^{20}}{2^{99}r^{19}d^{2}5^{d}(3kC)^{20r}T^{10}M^{2}}\right)^{2^{26}r^{11}kd^{2}(2kC)^{6r}T^{4}/\theta^{6}}\]
by Lemma \ref{intBS}.

Finally, we use Lemma \ref{rankdown2} to make the connection between the rank of our quadratic phases with respect to $P$ and $P'$. The lemma tells us that $\alpha_{P'}(\beta_{i}) \leq (\gamma_{P}/\gamma_{P'}) \alpha_{P}(\beta_{i})$ for each $i=1,2,\dots,r$.

Let $\eta=\gamma_{P'}^{8(1+r^{4})}(\theta/4(3kC)^{r})^{16r^{3}}$. Lemma \ref{rankaveragecor} with $\eps=\eta$, $B_{1}=P'$ and $m=r$ tells us that there exists a pair $(u,v) \in [s]^{2}$ such that the bilinear form $\b_{uv}$ defined above satisfies
\[\alpha_{P'}(\beta_{uv}) \leq \gamma_{P'}^{-2r}(800d_{P'}/\eta)^{d_{P'} \log r/r^{3}}\alpha_{P'}(\beta_{i})^{1/r^{3}}+4\gamma_{P'}^{-2r}(\eta^{1/4}/\gamma_{P'}^{2})^{1/r^{3}}.\]
for any $i=1,2,\dots,r$. To conclude the proof, note that Lemma \ref{bilinearqr2} implies that
\[|\E_{x_{1}\in z_{1}+D} \E_{x_{2}\in z_{2}+D} \dots \E_{x_{s}\in z_{s}+D} \prod_{j=1}^{r} \omega^{\sum_{u,v=1}^{d}\beta_{uv}(x_{u},x_{v})+\sum_{u=1}^{d} \phi_{u}(x_{u})+\phi}|\leq  \theta/(3kC)^{r} + \alpha_{P'}(\beta_{uv})^{1/4}\]
for any fixed linear forms $\phi_{u}$ and any constant $\phi$, which is at most 
\[\theta/(3kC)^{r} +\gamma_{P'}^{-r/2}(800d_{P'}/\eta)^{d_{P'} \log r/4r^{3}}\alpha_{P'}(\beta_{i})^{1/4r^{3}}+4\gamma_{P'}^{-r/2}(\eta^{1/4}/\gamma_{P'}^{2})^{1/4r^{3}}\]
and therefore bounded above by
\[\theta/(3kC)^{r}+\gamma_{P'}^{-r/2}(800d_{P'}/\eta)^{d_{P'} \log r/4r^{3}} \left(\frac{\gamma_{P}}{\gamma_{P'}}\right)^{1/4r^{3}}\alpha_{P}(\beta_{i})^{1/4r^{3}}+2\gamma_{P'}^{-r/2}(\eta^{1/4}/\gamma_{P'}^{2})^{1/4r^{3}}.\] 
Our choice of $\eta$ implies that the third term is no larger than the first, and that the second term is at most 
\[ \left(\frac{2^{239}r^{59}d^{6}5^{2d}(3kC)^{50r}T^{24}M^{4}}{\gamma'\e^{2}\rho^{2}\theta^{50}} \right)^{2^{53}r^{22}k^{2}d^{4}(2kC)^{12r}T^{8}/\theta^{12}}\alpha_{P}(\beta_{i})^{1/4r^{3}}.\]

Recalling that in (\ref{fifth}) we had a pre-factor of $\prod_{j=1}^{r}\lambda_{y_{j}}$ with each $|\lambda_{y_{j}}|\leq 3C_{j}$ and in (\ref{second}) a factor of $k^{r}$, and that $\alpha_{P}(\beta_{i}) \leq e^{-R}$ for every $i=1,2, \dots,r$, we obtain the final bound as stated.
\end{proof}

\section{Proof of the main result}

Most of the work towards proving the main result was accomplished in the preceding section. Here we shall formally complete the proof of the following theorem.

\begin{theorem}\label{znresult}
Let $\seq L r$ be a square independent system of linear forms in $s$ variables of Cauchy-Schwarz complexity at most 2. For every $\eta>0$, there exists $c>0$ with the following property. Let $f: \Z_{N} \maps [-1,1]$ be such that $\|f\|_{U^{2}} \leq c$. Then
\[\left|\E_{x \in (\Z_{N})^{s}} \prod_{i=1}^{r} f(L_{i}(x))\right| \leq \eta.\]
Moreover, $c$ can be taken to depend on $\eta$ in a doubly exponential fashion.
\end{theorem}

As in \cite{Gowers:2007tcs, Gowers:2009lfuI}, we need to recall a well-established result that will allow us to neglect the quadratically uniform part of the decomposition.


\begin{theorem}\label{gvnmod}
Let $f_1,\dots,f_r$ be functions on $\Z_{N}$, and let $\seq L r$ be a linear system of Cauchy-Schwarz complexity at most $2$ consisting of $r$ forms in $s$ variables. Then
\[\left|\E_{\x \in (\Z_{N})^s} \prod_{j=1}^r f_j(L_j(\x)) \right|\leq \min_j\|f_j\|_{U^{3}}\prod_{i\neq j} \|f_i\|_\infty.\]
\end{theorem}


\begin{proof}[Proof of Theorem \ref{znresult}]
Let $\eta>0$, and let $c>0$ be chosen in terms of $\eta$ later. Given $f:\Z_{N}\rightarrow[-1,1]$ with $\|f\|_{U^2} \leq c$ we first apply Theorem \ref{HRdecomp} with $\delta_1=\eta/(24r)$ to obtain a decomposition
\[f=f_1+g_1+h_1,\]
where $f_1=\sum_j Q_j^{(1)}U_j^{(1)}$ with
$\sum_j\|U_j^{(1)}\|_\infty\leq 2 C_1$, $\sum_j\|U_j^{(1)}\|_{U^{2}}^{*}\leq T_1$, $\|g_1\|_1 \leq 10 \delta_1$ and
$\|h_1\|_{U^3} \leq 2 \delta_1$. We have carefully ensured that each
quadratic average $Q_j^{(1)}$ has rank at last $R_{1}$ for some $R_{1}$ to be chosen later. Aiming to bound 
\[\E_{\x\in(\Z_{N})^s}\prod_{j=1}^r f(L_j(\x))\]
above in absolute value by $\eta$ for sufficiently uniform $f$, we first replace the first instance of $f$ in the product by $g_1+h_1$. The product involving $g_1$ yields an error term of $10 \delta_1$ since all the remaining factors have $L_{\infty}$ norm bounded by $1$, while the product involving $h_1$ yields an error of $2\delta_1$ by Theorem \ref{gvnmod} above. Our choice of $\delta_1$ implies that the sum of these two errors is at most $\eta/(2r)$.

Now we apply Theorem \ref{HRdecomp} again, this time with $\delta_2=\eta/(48rC_1)$, to obtain a decomposition
\[f=f_2+g_2+h_2,\]
where $f_2=\sum_j Q_j^{(2)}U_j^{(2)}$ with
$\sum_j\|U_j^{(2)}\|_\infty\leq 2C_2$, $\sum_j\|U_j^{(2)}\|_{U^{2}}^{*}\leq T_2$, $\|g_2\|_1 \leq 10 \delta_2$ and
$\|h_2\|_{U^3} \leq 2 \delta_2$. When replacing the first instance of $f$ in the new product
\[\E_{\x\in(\Z_{N})^s}f_1(L_1(\x))\prod_{j=2}^{r} f(L_j(\x))\]
with $g_2+h_2$, the product involving $g_2$ now contributes an
error term of at most $20\delta_2 C_1$ (since $\|f_1\|_{\infty}\leq 2 C_1$). By Theorem \ref{gvnmod} it follows that the contribution from the product involving $h_2$ is bounded above by $4\delta_2 C_1$. Therefore the total error incurred is at most $24\delta_2 C_1$, which is at most $\eta/(2r)$ by our choice of $\delta_2$.

When we come to apply Theorem \ref{HRdecomp} to the $k$th instance of $f$ in the original product, we need to do so with $\delta_k$ satisfying $12\cdot 2^{k-1}\delta_k C_1 \dots C_{k-1} \leq \eta/(2r)$ for $k=2,\dots,r$. This
ensures that up to an error of $\eta/2$, it suffices to consider the product
\[\E_{\x\in(\Z_{N})^s}\prod_{j=1}^r f_j(L_j(\x)),\]
where each function $f_{j}$ is quadratically structured. The key estimate, Proposition \ref{struccomput} with $\theta=\eta/20$, now implies that 
\[|\E_{\x\in(\Z_{N})^s}\prod_{i=1}^r f_i(L_i(\x))| \leq \eta/4 + \chi e^{-R/4r^{3}},\]
where
\[\chi(\eta)= \left(\frac{2^{439}r^{59}d^{6}5^{2d}(3kC)^{50r}T^{24}M^{4}}{\gamma'\e^{2}\rho^{2}\eta^{50}} \right)^{2^{113}r^{22}k^{2}d^{4}(2kC)^{12r}T^{8}/\eta^{12}}\]
with $C=\max_{j} C_{j}$, $T=\max_{j}T_{j}$, $R=\min_{j} R_{j}$, $d=\max_{j}d_{j}$, $k=\max_{j}k_{j}$, $\gamma'=\min_{j}\gamma'_{j}, \rho=\min_{j}\rho_{j}$ and $\eps=\max_{j}\eps_{j}$. Choosing $R_{j}$ large enough at each stage, we will be able to force $\chi e^{-R/4r^{3}}\leq \eta/4$.

The argument is essentially complete; it remains to check that the dependence we obtain is doubly exponential. First, note that every application of Theorem \ref{HRdecomp} returns $C_{j},d_{j},k_{j}$ as well as $\rho_{j}$ and (the upper bound on) $\eps_{j}$ as parameters that are polynomial in $\delta_{j}$, and hence polynomial in $\eta$. Only $T_{j}$ is exponential in $\eta$.

Also, in order to apply Proposition \ref{struccomput}, we needed to assume that the parameters $\eps_{j}$ satisfy $r(2kC)^{r}\eps \leq \eta/20$. This means that the density $\gamma_{j}'$ of the progression $P_{j}$ used in the $j^{th}$ decomposition is at least $(\eta \rho/2^{17}r(2kC)^{r}d^{2}k5^{d})^{dk}$, which does not affect the doubly exponential nature of $\chi(\eta)$. Hence it is possible to choose $R_{j}$ to be an exponential function of $\eta$ at each stage. By Theorem \ref{HRdecomp}, this is possible provided that $\|f\|_{U^{2}} \leq c$, where $c$ is bounded above by
\[e^{-2^{15k} d^{7k} k^{6k}R}\left(\frac{\delta^{7}\rho^{3}\eps^{3}}{2^{102k}k^{8}d^{8}5^{2d}C^{3}T^{2}} \right)^{2^{65k}md^{12k}k^{15k}T^{4}C^{2}/\d^{6}},\]
where  $\delta=\min_{j} \delta_{j}$ and $R$ was chosen to satisfy $\chi e^{-R/4r^{3}}\leq \eta/4$. More precisely, the average $\E_{x}\prod_{j=1}^r f(L_j(\x))$ is less than $\eta$ provided that $c$ is at most
\[\left(\frac{\eta^{54}\rho^{3}}{2^{466} k r^{62}d^{8}5^{3d}(3kC)^{52r}T^{24}M^{4}} \right)^{2^{118k}r^{25}k^{8k} d^{11k}(2kC)^{12r}T^{8}/\eta^{12}}
\left(\frac{\delta^{7}\rho^{3}\eps^{3}}{2^{102k}k^{8}d^{8}5^{2d}C^{3}T^{2}} \right)^{2^{65k}md^{12k}k^{15k}T^{4}C^{2}/\d^{6}},\]
With $m$ and $r$ being fixed constants, $M$ being a constant depending on the coefficients of the linear forms, $C, \delta, d, k$ and $\rho$ depending polynomially on $\eta$ and $T$ depending exponentially on $\eta$, this bound on $c$ is indeed doubly exponential in $\eta$ as claimed.

\end{proof}

%

\end{document}